\documentclass{jT} 
\usepackage{amssymb,amsmath,latexsym,amsthm}
\usepackage[english]{}

\def\notdiv{\nmid}
\def\div{\,\vert\,}

\def\too{\relbar\lien\rightarrow}
\def\tooo{\relbar\lien\relbar\lien\too}

\let\ds=\displaystyle
\let\st=\scriptstyle

  \def\N{\mathbb{N}}
  
  \def\C{\mathbb{C}}
  \def\Q{\mathbb{Q}}
  \def\Z{\mathbb{Z}}
  \def\F{\mathbb{F}}

 \def\zz{\mathbb{Z}}

\def\virg{\raise 2pt \hbox{,}\,\,}

\def\No{{\rm N}}
\def\Sauf{\!\setminus\!}

\def\Cl{{\mathcal C}\hskip-2pt{\ell}}
\def\cl{c\hskip-1pt{\ell}}

\def\Frac#1#2{\hbox{\footnotesize $\displaystyle \frac{#1}{#2}$}}
\def\plus{\ds\mathop{\raise 2.0pt \hbox{$\bigoplus $}}\limits}
\def\mult{\ds\mathop{\raise 2.0pt \hbox{$\bigotimes$}}\limits}
\def\prd{ \ds\mathop{\raise 2.0pt \hbox{$  \prod   $}}\limits}
\def\Cap{ \ds\mathop{\raise 2.0pt \hbox{$\bigcap   $}}\limits}
\def\Cup{ \ds\mathop{\raise 2.0pt \hbox{$\bigcup   $}}\limits}
\def\sm{  \ds\mathop{\raise 2.0pt \hbox{$  \sum    $}}\limits}
\def\ev{\emptyset}

\def\fin{\vbox{\hrule\hbox to 7.2pt{\vrule height 
7pt\hfil\vrule}\hrule}}
\def \tensorZp{\otimes{\raise -0.8pt \hbox{\!\!$_{_{\zz_{\!p}}}$}}}
\def \tensorZ{\otimes{\raise -0.8pt \hbox{\!\!$_{_{\zz}}$}}}
\let\st=\scriptstyle  

\def\fin{\vbox{\hrule\hbox to 7.2pt{\vrule height 
7pt\hfil\vrule}\hrule}}
\def\lien{\mathrel{\mkern-4mu}}
\def\wt{\widetilde}
\def\ov{\overline} 
\def\wh{\widehat} 

\newtheorem{theo}{Theorem}
\newtheorem{prop}{Proposition}
\newtheorem{rema}{Remark}
\newtheorem{conj}{Conjecture}

\newtheorem{lemm}{Lemma}
\newtheorem{cor}{Corollary}
\newtheorem{ex}{Example}
\newtheorem{defi}{Definition}

\begin{document}
 
 \title{ Some works of Furtw\"angler and Vandiver revisited
 and Fermat${}'$s last theorem}

\author[Georges {\sc Gras} {\rm and} Roland {\sc Qu\^eme}]{{\sc 
Georges} Gras {\rm and} {\sc Roland} Qu\^eme}

\address{Georges {\sc Gras} \hfill   Roland {\sc Qu\^eme}\ \ \ \ \ \ \ \ \ \ \ \ \ \ \ \ \ \ \\
Villa la Gardette, chemin Ch\^ateau Gagni\`ere \hfill  13 Avenue du ch\^ateau d'eau \\
F-38520 Le Bourg d'Oisans \hfill  F-31490 Brax \ \ \ \ \ \ \ \ \ \ \ \ \ \ \ \ \ \ \ \\
g.mn.gras@wanadoo.fr   \hfill  roland.queme@wanadoo.fr \ \  \\
http://maths.g.mn.gras.monsite-orange.fr/  \hfill\, http://roland.queme.free.fr/  \\
 ${}$\\ April 11, 2011}


\keywords{Fermat${}'$s last theorem, Furtw\"angler${}'$s theorems,  
cyclotomic fields, cyclotomic units, class field theory, \v Cebotarev 
density theorem} 

\subjclass{11D41, 11R18}

\begin{abstract} From some works of P. Furtw\"angler and H.S. 
Vandiver, we put the basis of a new cyclotomic approach to
Fermat${}'$s last theorem for~$p>3$ and to a stronger version
called SFLT, by introducing governing fields of 
the form $\Q(\mu_{q-1})$ for prime numbers $q$.
 We prove for instance that if there exist infinitely many primes $q$,
 $q \not \equiv 1 \pmod {p}$, $q^{p-1} \not \equiv 1$ $\pmod {p^2}$,
such that for ${\mathfrak q} \div q$ in $\Q(\mu_{q-1})$, we have
${\mathfrak q}^{1-c} = {\mathfrak a}^p\,(\alpha)$
with $\alpha \equiv 1 \pmod {p^2}$ (where $c$ is the complex 
conjugation), then Fermat${}'$s last theorem holds for $p$.

\smallskip\noindent
More generally, the main purpose of the paper is to show
that the existence of nontrivial solutions for SFLT implies 
some strong constraints on the arithmetic of the
fields $\Q(\mu_{q-1})$. From there, we give sufficient conditions
of nonexistence that would require further investigations to lead to
a proof of SFLT, and we formulate various conjectures.
This text must be considered as a 
basic tool for future researches (probably of analytic or geometric nature).

\medskip\noindent
{\sc R\'esum\'e.} Reprenant des travaux de P. Furtw\"angler et H.S. 
Vandiver, nous posons les bases d'une nouvelle approche cyclotomique du
dernier th\'eor\`eme de Fermat pour $p>3$ et d'une version plus forte
appel\'ee SFLT, en introduisant des corps 
gouvernants de la forme $\Q(\mu_{q-1})$ pour $q$ premier.
Nous prouvons par exemple que s'il existe une infinit\'e de nombres premiers $q$,
$q \not \equiv 1 \pmod {p}$, $q^{p-1} \not \equiv 1 \pmod {p^2}$,
tels que   pour ${\mathfrak q} \div  q$  dans $\Q(\mu_{q-1})$,
on ait  ${\mathfrak q}^{1-c} = {\mathfrak a}^p\,(\alpha)$
 avec $\alpha \equiv 1 \pmod {p^2}$
(o\`u $c$ est la conjugaison complexe), alors le th\'eor\`eme de Fermat
est vrai pour $p$.

\smallskip\noindent
Plus g\'en\'eralement, le but principal de l'article est de montrer
que l'existence de solutions non triviales pour SFLT implique 
de fortes contraintes sur l'arith\-m\'etique des corps $\Q(\mu_{q-1})$.
A partir de l\`a, nous donnons des conditions suffisantes de non existence
qui n\'eces\-site\-raient des investigations suppl\'ementaires pour
conduire \`a une preuve de SFLT, et nous formulons diverses conjectures.
Ce texte doit \^etre consid\'er\'e comme un outil de base pour 
de futures  recherches (probablement analytiques ou g\'eom\'etriques).

\smallskip\smallskip
\centerline{*******}

\vspace{-0.05cm}
\noindent
This second version includes some corrections in the English 
language, an in depth study of the case $p=3$ (especially Theorem 8), 
further details on some conjectures, and some minor mathematical 
improvements.

\end{abstract}

\maketitle 

\section{Introduction}

This paper is devoted to the study of the following phenomenon.
Consider the maximal abelian extension $\ov\Q^{\rm nr}$ of $\Q$, 
unramified (= nonramified) at a given prime $p>2$; from class field theory,
we get $\ov\Q^{\rm nr} = \Cup{}_{n,\, p \notdiv n}\,\Q(\mu_n)$.
Then denote by $H_{\ov\Q^{\rm nr}}$ the maximal $p$-ramified (i.e., unramified 
outside $p$)
abelian $p$-extension of $\ov\Q^{\rm nr}$; this extension is given
by $\Cup{}_{n,\, p \notdiv n}\,H_{\Q(\mu_n)}$ where
$H_{\Q(\mu_n)}$ is the maximal $p$-ramified
abelian $p$-extension of $\Q(\mu_n)$.

Then consider
$H_{\ov\Q^{\rm nr}}\,{\st [p]} := \Cup{}_{n,\, p \notdiv n}\,H_{\Q(\mu_n)}{\st [p]}$,
the  maximal $p$-elementary $p$-ramified
extension of $\ov\Q^{\rm nr}$, union of the corresponding
maximal  $p$-elemen\-tary $p$-ramified
extensions  of $\Q(\mu_n)$.

\smallskip
We have found that any nontrivial solution $(u,v)$ of
a classical diophantine equation,
associated to Fermat${}'$s equation, and called the SFLT 
equation\,\footnote{\,Equation
$(u+v \,\zeta)\,\Z[\zeta] ={\mathfrak w}_1^p$ or
${\mathfrak p}\,{\mathfrak w}_1^p$,
in integers $u$, $v$ with g.c.d.\,$(u,v)=1$, equivalent to
$\No_{K/\Q}(u+v\,\zeta) = w_1^p$ or $p\,w_1^p$, where 
$\zeta := e^{2i\pi/p}$, $K:=\Q(\zeta)$, ${\mathfrak p}:=(\zeta-1)\,\Z[\zeta]$
(see Conjecture 1). Remark that the important condition g.c.d.\,$(u,v)=1$
implies ${\mathfrak w}_1$ prime to $p$.

\  Note that if $u\,v = 0$, the condition g.c.d.\,$(u, v) = 1$
implies $(u, v) = (\pm 1, 0)$ or $(0, \pm 1)$.},
implies some constraints on the law of 
decomposition of every prime $q\ne p$ in $H_{\ov\Q^{\rm nr}}\,{\st [p]}/\ov\Q^{\rm nr}$.

\smallskip
These constraints  may be characterized  at some finite steps  via the
law of decomposition of $q$ in a canonical
family ${\mathcal F}_n$ of $p$-cyclic subextensions of
$H_{\Q(\mu_n)}{\st [p]}/\Q(\mu_n)$, where $n \div q-1$ depends on $q$, $u$, $v$
(see Theorem 4).

\smallskip
Some aspects needed to prove this relation can be found in some former technics
of  Furtw\"angler and Vandiver, in a different viewpoint from ours,
to try to give a classical cyclotomic proof of Fermat${}'$s last theorem
(FLT).

\smallskip
Of course the problem is now empty for Fermat${}'$s equation, except if we 
wish to prove  FLT by this way; but we will see that
for the SFLT equation the result is unknown for $p>3$
(but conjecturally similar) and, moreover, leads to infinitely many solutions
for $p=3$. But as we will show, the case $p=3$ is exceptional
and we will explain in Section 9 for what reasons.

\smallskip
Unfortunately, we have no deep results to propose, but only
some material which may be helpful for those interested
in going further.
 
\section {Generalities on the method -- The $\omega$-SFLT equation}

\subsection{Prerequisites on Fermat${}'$s last theorem}

Let $p$ be a prime number, $p > 2$.
Let  $a$, $b$, $c$  be pairwise relatively
prime nonzero integers, such that $\ a^p+b^p+c^p = 0$.

\smallskip
We can find for instance in [Gr1, Ri, Wa] the following obvious 
properties concerning such a specu\-lative counterexample to  FLT,
where $\zeta$ is a primitive $p$th root of unity, $K := \Q(\zeta)$,
${\mathfrak p}:=(\zeta-1) \Z[\zeta]$, and
$\No_{K/\Q}$ is the norm map in $K/\Q$ (for a detailed proof, a more 
complete
bibliography, and an analysis of the classical cyclotomic approach to 
FLT, we refer to [Gr1]):

 \smallskip
(i) We have (where $\nu \geq 0$ is the $p$-adic valuation of $c$)\,\footnote{\,If
$\nu \geq 1$, then  $\alpha:= \frac{a + c\,\zeta}{a + c\,\zeta^{-1}}$ is a pseudo-unit
(i.e., the $p$th power of an ideal), congruent to 1 modulo $p$; 
so, from [Gr1, Theorem 2.2, Remark 2.3, (ii)], $\alpha$ is locally a $p$th power
in $K$ giving easily $\alpha \equiv 1 \pmod {{\mathfrak p}^{p+1}}$, then
$\frac{c\,(\zeta-\zeta^{-1})}{a + c\,\zeta^{-1}} \equiv 0 \pmod {{\mathfrak p}^{p+1}}$,
hence  $c \equiv 0 \pmod {p^2}$. }:

\medskip
\centerline{ $a+b =  c_0^p$\,\  or\,\ 
$p^{\nu p -1} c_0^p$ \  with $\nu \geq 2$,
 \  and \  $\No_{K/\Q}(a+b \,\zeta) =  c_1^p$\,\  or\,\   $p\, c_1^p$,}

\smallskip\smallskip\noindent
with  $-c = c_0\, c_1$ or $p^\nu c_0\,c_1$, and  $p \notdiv c_0\,c_1$.
  By permutation, since $p\notdiv a b$, we have the following 
analogous relations:
 \begin{eqnarray*}
&&b+c =  a_0^p,  \ \ \  \No_{K/\Q}(b+c \,\zeta) =  a_1^p,  \ \,{\rm
with} \ -a = a_0\, a_1, \\
 &&c+a = b_0^p,  \ \ \ \No_{K/\Q}(c+a \,\zeta) =  b_1^p,  \ \,{\rm
 with}\ -b = b_0\, b_1.
 \end{eqnarray*} 

(ii) We have:

\medskip
\centerline{ $(a + b \,\zeta)\,\Z[\zeta] =
{\mathfrak c}_1^p$\  or\  ${\mathfrak p}\,{\mathfrak c}_1^p$,\ \,
with  $\No_{K/\Q}({\mathfrak c}_1) = c_1\Z$,}

\medskip\noindent
where ${\mathfrak c}_1$ is an integer ideal of $K$ prime to 
${\mathfrak p}$, and the analogous relations:
\begin{eqnarray*}
&&(b+c\,\zeta)\,\Z[\zeta]  = {\mathfrak a}_1^p, \ \ \hbox{with $
\No_{K/\Q}({\mathfrak a}_1) = a_1\Z$},   \\
&&(c+a\,\zeta)\,\Z[\zeta] = {\mathfrak b}_1^p, \ \ \hbox{with $
\No_{K/\Q}({\mathfrak b}_1) = b_1\Z$}. \\
 \end{eqnarray*} 

\vspace{-0.5cm}
(iii)  The positive numbers $a_1$, $b_1$, $c_1$ have prime divisors
all congruent to 1 modulo $p$.
 
\begin{lemm} We can choose $x,\,y,\,z \in \{a, b, c  \}$ 
     in the following manner:

     \smallskip
(i) First case of FLT, $p>3$: 
  \begin{eqnarray*}
 y-x  \ \not\equiv \  0 , \ \   y+x &\not\equiv& 0 \pmod {p}, \\
 y-z  \ \not\equiv \  0 , \ \  y+z &\not\equiv& 0 \pmod {p}, \\   
 \ \   x+z &\not\equiv& 0 \pmod {p}.
  \end{eqnarray*}

\vspace{-0.1cm}
(ii) First case of FLT, $p = 3$: 
   \begin{eqnarray*}
 y-x \ \equiv \  0 , \ \   y+x &\not\equiv& 0 \pmod {3}, \\
  y-z  \ \equiv  \ 0 , \ \   y+z &\not\equiv& 0 \pmod {3}, \\
   x-z  \ \equiv  \ 0 , \ \     x+z &\not\equiv& 0 \pmod {3}.
 \end{eqnarray*}
 
 \vspace{-0.1cm}
(iii) Second case of FLT, $p\geq 3$: 
 \begin{eqnarray*}
 y   &\equiv&   0 \pmod {p},         \\
 y-x \  \not\equiv \  0 ,  \ \   y+x &\not\equiv& 0 \pmod {p},  \\
 y-z \  \not\equiv \  0 ,  \ \   y+z &\not\equiv& 0 \pmod {p},  \\
   x-z \  \not\equiv \  0 ,  \ \    x+z &\equiv& 0 \pmod {p}.
 \end{eqnarray*}
 \end{lemm} 
 
 \begin{proof} Consider the differences $a-b$, $b-c$, $c-a$
     in the first case of FLT.
     If two of them are divisible by $p$, we obtain
     $a \equiv b \equiv c \not \equiv 0 \pmod {p}$, then since
     $a+b+c \equiv 0 \pmod {p}$, we get $3\,a\equiv 0 \pmod {p}$ which
     implies $p=3$. So, if $p>3$, there exist two differences having 
     the first required property, and called $y-x$, $y-z$.
     
     \smallskip
     The second condition is
     satisfied for any sum and any $p\geq 3$.
     
     \smallskip
     The case $p=3$ in the first case of FLT is 
     clear since   $a \equiv b \equiv c \equiv \pm 1$ $\pmod {3}$.
      
     \smallskip
   In the second case of FLT, we take $y=c \equiv 0 \pmod {p}$ so that all
   the conditions in (iii) are satisfied (we put $y=c$
   instead of $z=c$, to get, for $x+y\,\zeta$,  a $p$-primary pseudo-unit instead of
   a number  $x+y\,\zeta \in {\mathfrak p}$).
     \end{proof}
     
 Note that for $p>3$ in the first case, $x-z$ may be divisible by $p$ under some
circumstances (e.g. under the necessary condition
$2^{p-1} \equiv 1 \pmod {p^2}$ since, from 
$x^p+y^p+z^p=0$, we get $2\,z^p+ y^p \equiv 0 \pmod {p^2}$).

 \subsection{Statement of a stronger conjecture than FLT}
 
 We have given in [Gr1] a conjecture which implies FLT and which is 
 not covered by Wiles proof; we recall here
 the  statement, which will be called the 
  strong Fermat last theorem (SFLT).
  
\begin{conj}  Let $p$ be a prime number, $p>2$.
Then for $u,\,v \in \Z$, with g.c.d.\,$(u,v)=1$, the equation:
$$(u+v \,\zeta)\,\Z[\zeta] ={\mathfrak w}_1^p \  {\rm or}\  
{\mathfrak p}\, {\mathfrak w}_1^p $$
(depending on whether $u+v\not \equiv 0 \pmod {p}$ or not),
equivalent to:
$$\No_{K/\Q}(u+v \,\zeta) = w_1^p \  {\rm or}\  p\,w_1^p,
\ \,w_1= \No_{K/\Q}({\mathfrak w}_1) \in 1 + p\,\Z,$$
where ${\mathfrak w}_1$ is an ideal of $K$ (necessarily prime to ${\mathfrak p}$),
has no solution for $p>3$ except the trivial ones:
 $u+v \,\zeta = \pm 1$, $\pm \zeta$, $\pm (1+\zeta)$,
 and $\pm (1-\zeta)$.~\hfill\fin
\end{conj}

The difference between FLT and SFLT is the following. A solution
of Fermat${}'$s equation $u^p+v^p+w^p=0$ comes from a  solution
of $(u+v \,\zeta)\,\Z[\zeta] ={\mathfrak w}_1^p \  {\rm or}\  
{\mathfrak p}\, {\mathfrak w}_1^p$ (with the same $u, v$ as above), if and only if
there exists  $w_0 \in \Z$ such that $u+v = w_0^p$ or $p^{\nu p - 1}\, 
w_0^p$ since $\No_{K/\Q}(u+v \,\zeta) = w_1^p$  or $p\,w_1^p$,
giving $w := -w_0w_1$ or $- p^{\nu} w_0w_1$
for a solution of Fermat${}'$s equation.

\smallskip
As for FLT we can speak of the {\it first case} of the conjecture
(or of the equation) when:
$$u\,v\,(u+v) \not\equiv 0 \pmod {p}$$
and of the {\it second case} when: 
$$u\,v\equiv 0 \pmod {p}$$ (which implies $u$ or
$v\equiv 0 \pmod {p^2}$ as for  Fermat${}'$s equation); then the case:
$$u+v \equiv 0 \pmod {p}$$
will be called  the {\it special case} for SFLT.

\smallskip 
For the first case of SFLT, we have not necessarily 
$u-v \not\equiv 0 \pmod {p}$,%
\,\footnote{\,If $u-v \equiv 0 \pmod {p}$,
then  $\alpha:= \frac{u\zeta+v}{u+ v\zeta}$ is a pseudo-unit
congruent to 1 modulo $p$; so, from [Gr1, Theorem 2.2, Remark 2.3, (ii)], $\alpha$
is locally a $p$th power giving 
$\alpha \equiv 1 \pmod {{\mathfrak p}^{p+1}}$, then
$\frac{(u-v)(\zeta-1)}{u+ v\zeta} \equiv 0 \pmod {{\mathfrak p}^{p+1}}$,
hence $u-v \equiv 0 \pmod {p^2}$. This is valid in the Fermat case if 
$x-z\equiv 0\pmod {p}$, and gives $x-z\equiv 0 \pmod {p^2}$.}
except for $p=3$ since
$u\,v\,(u+v) \not\equiv 0 \pmod {3}$ implies $u\equiv v\equiv \pm 1
 \pmod {3}$, hence $u-v \equiv 0 \pmod {3}$. See the forthcoming
 Remark 1 for  $p=3$.

\medskip
In the sequel, we will assume that $(x,\,y,\,z)$ is a solution
of Fermat${}'$s equation such that the conditions of Lemma~1 are satisfied 
(i.e., $y-x$ and $y-z$ are prime to $p$ when $p>3$, and
if $p \div  xyz$, we suppose that $p \div  y$).

\smallskip
In that case we will have  two similar counterexamples to the 
above conjecture:
 $(x+y \,\zeta)\,\Z[\zeta] = {\mathfrak z}_1^p$,
 $(z+y \,\zeta)\,\Z[\zeta] = {\mathfrak x}_1^p$ (this concerns the first or second case 
 of SFLT).
Then it will exist the third counterexample
$(x+z \,\zeta)\,\Z[\zeta] = {\mathfrak y}_1^p$ (if $p\notdiv y$)
or ${\mathfrak p}\, {\mathfrak y}_1^p$ (if $p \div y$). 

\smallskip
More precisely, the first case of SFLT implies the first case of FLT,
the second or the special case of SFLT implies the second case of FLT,
and FLT holds as soon as first and second cases, or first and special cases
of SFLT, hold.

\begin{rema}{\rm  Conjecture 1 is false
for $p=3$  since for $\zeta = j$ of order 3 we have 
the six kind of parametric formulas giving all solutions:

\medskip \smallskip
\centerline{$u+v\,j = j^h\,(s+t\,j)^3,\ {\rm or}\ j^h\,(1-j)\,(s+t\,j)^3,
\  s,\,t \in \Z,\  s+t \not\equiv 0 \pmod {3},$}

\medskip\noindent
${\rm g.c.d.\,}(s,t)=1$, and $0\leq h< 3$.
These solutions concern all the cases:

\smallskip \smallskip\noindent
-- first case (for which $u-v \equiv 0 \pmod {9}$):

\smallskip\noindent
\ \ {\footnotesize $\bullet$} $(u,v) = (-s^3 - t^3+3s^2 t, -s^3-t^3+3s t^2)$,
from  $u+v\,j = j^2\,(s+t\,j)^3$;

\smallskip\smallskip\noindent
-- second case (for which $u$ or $v \equiv 0 \pmod {9}$):

\smallskip\smallskip\noindent
\ \ {\footnotesize $\bullet$} $(u,v) = (3s t^2 -3s^2 t, s^3+t^3 - 3 s^2 t)$,
from  $u+v\,j = j\,(s+t\,j)^3$;

\smallskip\noindent
\ \ {\footnotesize $\bullet$} $(u,v) = (s^3+t^3-3s t^2, 3 s^2 t -3s t^2)$,
from  $u+v\,j = (s+t\,j)^3$;

\smallskip\smallskip\noindent
-- special cases (for which $u+v \equiv 0 \pmod {3}$):

\smallskip\noindent
\ \ {\footnotesize $\bullet$} $(u,v) = (s^3+t^3+3s^2 t -6s t^2, -s^3-t^3+6s^2 t -3s t^2)$,
from  $u+v\,j =  (1-j)\,(s+t\,j)^3$;

\smallskip\smallskip\noindent
\ \ {\footnotesize $\bullet$} $(u,v) = (s^3+t^3 - 6 s^2 t + 3s t^2, 2s^3 + 2t^3 -3 s^2 t -3s t^2)$,
from  $u+v\,j =  j\,(1-j)\,(s+t\,j)^3$;

\smallskip\smallskip\noindent
\ \ {\footnotesize $\bullet$} $(u,v) = (-2s^3 -2t^3+3s^2 t +3s t^2, -s^3-t^3 -3 s^2 t +6 s t^2)$,
from  $u+v\,j =  j^2\,(1-j)\,(s+t\,j)^3$.

\medskip
The  special cases are not similar since for the first solution $u+v \equiv 
0$ $\pmod {9}$ and for the others, $u+v \equiv \pm 3 (s^3+t^3) \equiv\pm 3 (s+t) 
 \equiv \pm 3  \pmod {9}$.
 
\smallskip
Contrary to the case of Fermat${}'$s equation,
we will not take into account the symmetries of the writing of the 
solutions $(u,v)$, especially for the second case (this will be important 
in  Section 9)
but we will not distinguish $(u,v)$ from $(-u,-v)$.~\hfill\fin }
  \end{rema}

\medskip
 Thus a proof of SFLT must eliminate, in a natural way,
 the case $p=3$ which is an obstruction for the relevance of
 the method developed  here. We will explain
 later (Section 9) for what reasons this case is exceptional
 and finally does not matter, a priori, for the general theory; we are 
 obliged to differ this justification because we need many general
 material. Meanwhile, for a more comprehensive information, we do not
always suppose $p>3$ in the development of the first parts of the study.

\subsection{The cyclotomic field $\Q(\zeta)$ and the character $\omega$ }

We first recall the algebraic context concerning the
cyclotomic field $K = \Q(\zeta)$.

 \begin{defi} {\rm (i) Let $g:={\rm Gal\,}(K/\Q)$ and
let $\omega$ be the character of Teich\-m\"uller of $g$ (i.e.,
the character with values in $\mu_{p-1}(\Q_p)$ such that
for the $s_k \in g$ defined by $s_k(\zeta) = \zeta^k$,
$k \not\equiv 0 \pmod {p}$, $\omega(s_k)$ (also denoted $\omega(k)$)
is the unique  $(p-1)$th root
of unity in $\Q_p$, congruent to $k$ modulo $p$).

\smallskip
(ii) The idempotent corresponding to $\omega$ is:
\vspace{-0.1cm}
$$ e_\omega :=  \hbox{$\frac{1}{p-1}$} \sm_{s \in g} \omega^{-1}(s)\, 
s =
\hbox{$\frac{1}{p-1}$} \sm_{k=1}^{p-1} \omega^{-1}(k)\, s_k\,\in 
\Z_p[g] .$$

(iii) We represent  $e_\omega$ in $\Z[g]$ modulo $p$ and still denote 
it $e_\omega$ (this means that $e_\omega\, s_k \equiv
\omega(k)\, e_\omega \equiv k\, e_\omega \pmod {p\,\Z[g]}$ and that
$e_\omega \,(1-e_\omega) \in p\,\Z[g]$).

\smallskip
Put $e_\omega := \sum_{k=1}^{p-1}u_k\,s_k$, $u_k \in\Z$, $u_k \equiv
\frac{1}{p-1} \omega^{-1}(k)  \equiv \frac{k^{-1}}{p-1} \pmod {p}$.

\smallskip
We have $\omega^{-1}(s_{p-k}) = -\omega^{-1}(s_k)$ since $\omega 
(s_{-1})=-1$; thus we can 
suppose that $u_{p-k} =- u_k$ for $1 \leq k \leq \frac{p-1}{2}$.
Then we have  $e_\omega = (1- s_{-1})\,e'_\omega$ with $e'_\omega = 
\sum_{k=1}^{\frac{p-1}{2}}u_k\,s_k$.

\smallskip
In that case, if an element $A$ of a multiplicative $\Z[g]$-module 
${\mathcal M}$ is fixed  by the complex conjugation $s_{-1}$ of $K$,
we then have $A^{e_\omega} = 1_{\mathcal M}$ (the unit element of ${\mathcal M}$).
 
\smallskip
(iv) We have $\zeta^{e_\omega} = \zeta$ for any representative 
$e_\omega$.~\hfill\fin}
\end{defi}

\begin{ex} {\rm For $p=3$ we have  $e_\omega = \frac{1}{2}(1 - s)$, 
with $s := s_{-1}$. Thus a representative with integer coefficients may be
 $e_\omega = s - 1$.
    
    \smallskip
    For $p=5$, we have for instance $e_\omega = -1+2\,s_2 - 2\,s_3
    + s_4 =  -1+2\,s + s^2 - 2\,s^3 = (1-s^2)\,(2\, s -1)$, with $s 
   := s_2$.~\hfill\fin}
    \end{ex}

Recall that the unit group $E$ of $K$ is equal to
$\langle\,\zeta\,\rangle \oplus E^+$, where $E^+$ is the group of
units of the maximal real subfield $K^+$ of $K$ (see [Wa, Prop.\,1.5]).

\smallskip
Thus if $\varepsilon = \zeta^h\,\varepsilon^+$,
$\varepsilon^+ \in E^+$,  we get
$\varepsilon^{e_\omega} = \zeta^h$.

\subsection{The principles of the method -- The fundamental relation}

The purpose of this text is to exa\-mine some properties of the 
arithmetic of the fields $\Q(\mu_{q-1})$, in relation with a 
nontrivial solution of the SFLT equation:
 $$(u+v\,\zeta)\,\Z[\zeta] = 
 {\mathfrak w}_1^p\,\ {\rm or}\ \,{\mathfrak p}\,{\mathfrak w}_1^p, $$
 with g.c.d.\,$(u,v) = 1$,
 for  prime numbers $q$ such that $q \notdiv u\,v\,$ and 
the order $n$ of $\frac{v}{u}$ modulo $q$ is prime to $p$.

\smallskip
The cases where $n\leq 2$ (i.e., $q \div  u^2-v^2$) are particular, 
especially in the case where $(u,v)$ is a part of a solution
 $(x, y, z)$ of Fermat${}'$s equation, and give  Furtw\"angler${}'$s 
 theorems [Fur] (see Corollaries 2 and 3
 to Lemma 3 for a generalization of  Furtw\"angler${}'$s
 theorems for the SFLT equation, and Remark 3 for the classical case of 
 the FLT equation; see also [Mih] in the context of a
 Nagell--Ljunggren equation, which is the particular case
 of the SFLT equation with $v=1$).

\smallskip
The cases where $n$ is divisible by any nontrivial power of $p$ give technical
complications and are of a different nature. Some complements in this 
direction are developed in [Que] where similar studies are proposed.

\begin{lemm}   Let $u$, $v$ be  relatively prime 
 integers, let $n \geq 1$, and let $q$ be a prime number.
 Then the two following properties are equivalent:

\smallskip
(i) $q \notdiv n$ and $q \div \Phi_n(u,v) :=
 \prd_{\xi'\,{\rm of\, order}\,\,n} (u\,\xi' - v)$;

\smallskip
(ii) $q \notdiv u \,v$ and $\frac{v}{u}$ is of order $n$ modulo $q$.
\end{lemm}

\begin{proof} Suppose that $q \div \Phi_n(u,v)$ and $q\notdiv 
n$. Then  $q \notdiv u\,v$ since $\Phi_n(u,v)$ is an homogeneous form
 $u^{\phi(n)} \pm\cdots \pm v^{\phi(n)}$
in coprime integers $u$, $v$, where $\phi(n)$
is the Euler indicator.

\smallskip
For $\xi$ of order $n$ fixed, the  ideal $(q, u\,\xi - v)$ of 
 the field $\Q(\mu_n)$ is a prime ideal dividing $q$ because of the
relation $q \,\big\vert\, \Phi_n(u,v) =\!\! \prd_{\xi'\,{\rm of\, order}\,\,n} (u\,\xi'- v)$;
moreover $(q, u\,\xi - v)$ is of degree 1,  unramified in $\Q(\mu_n)/\Q$ 
(since $q \notdiv n$), thus we get $q \equiv 1 \pmod {n}$ and the fact that
$\frac{v}{u}$ is of order $n$ modulo~$q$.

\smallskip
If  $q \notdiv u\,v$ and $\frac{v}{u}$ is of order $n$ modulo $q$, 
then 
$u^n-v^n \equiv 0 \pmod {q}$. From the relation $u^n-v^n =
\prd_{d \div n} \Phi_d(u,v)$ we deduce that there exists $m \div n$
such that $q \div \Phi_m(u,v)$, which implies $q  \div u^m-v^m$,
hence  $m=n$ by definition of the order; since we have 
$(\frac{v}{u})^q \equiv \frac{v}{u} \pmod {q}$, it is clear that
the order $n$ cannot be divisible by $q$, proving the lemma.
\end{proof}

\begin{cor}  Consider the set of numbers of the form $\Phi_n(u,v)$
when $n$ varies in $\N \Sauf \{0\}$.

\smallskip\noindent
Then a prime number $q$ divides
one of the numbers $\Phi_n(u,v)$, $n \not\equiv 0 \!\pmod {q}$, 
if and only if $q\notdiv u\,v$.
When these conditions ($q \notdiv n$, $q \div \Phi_n(u,v)$) are satisfied, then
$n$ is unique.~\hfill\fin
\end{cor}

\smallskip
If $q$ is an arbitrary given prime number, to have $q  \div  
\Phi_n (u,v)$ with $n>2$ and $q \notdiv n$, we must first verify that 
$q\notdiv u\,v(u^2-v^2)$ and then compute the order $n$ of 
$\frac{v}{u}$ modulo $q$ which is then a divisor of 
$q-1$.\,\footnote{\,It is clear that the trivial solutions
$u+v \,\zeta =  \pm 1$, $\pm \zeta$, $\pm (1\pm\zeta)$ of the SFLT 
equation are precisely such
that $u\,v\,(u^2-v^2) = 0$, in which case such primes $q$ do not exist,
which has perhaps a significant meaning. }

\begin{defi} {\rm  Let $q \ne p$ be a prime number.

\smallskip
(i) Fermat quotients.
 Let $f$ be the residue degree of $q$ in $K/\Q$ and 
let   $\kappa := \Frac{q^f-1}{p}$.
Since $f  \div  p-1$, we have $\kappa \equiv 0 \pmod {p}$ if and 
only if $q^{p-1}\equiv 1$ $\pmod  {p^2}$.

The integer  $\ov \kappa :=  
\Frac{q^{p-1}-1}{p}$ is called the Fermat quotient of $q$.
We have the relation $\ov \kappa  \equiv \Frac {p-1}{f} \,\kappa
\equiv -\frac{1}{p}\,{\rm log}(q) \pmod {p}$, where ${\rm log}$
is the $p$-adic logarithm.

\smallskip
  (ii) Power residue symbols. Let us recall the definition
  and properties of the $p$th power 
residue symbols  $\big(\Frac{\hbox{\footnotesize$\ \bullet\ $} 
}{\hbox{\footnotesize $\ \bullet\ $} }\big)$
in $K$ and $M:= \Q(\mu_{n}) K$, $n \div q-1$,
with values in  $\mu_p$. Let ${\mathfrak q}$ be a prime ideal dividing $q$ in
$\Q(\mu_{n})$.

\smallskip
If $\alpha \in M$ is prime to
${\mathfrak Q} \div  {\mathfrak q}$ in $M$, then let $\ov\alpha$ 
be the image of $\alpha$ in
the residue field $Z_M/{\mathfrak Q}\simeq Z_K/{\mathfrak q}_K
\simeq \F_{q^f}$ for ${\mathfrak q}_K = Z_K \cap {\mathfrak Q}$
(indeed, $q$ totally splits in $M/K$); since
$\zeta \in Z_M$, the image $\ov \zeta$ of $\zeta$ is of order $p$
(since $\zeta \not\equiv 1 \pmod {\mathfrak Q}$) and we can put
$\ov\alpha_{}^{\,\kappa} = \ov \zeta^{\,\mu}$, $\mu \in \Z/p\Z$,
which defines the $p$th power residue symbol
$\Big(\Frac{\alpha}{{\mathfrak Q}}\Big)_{\!\!M} :=
\zeta^\mu$; this symbol is equal to 1 if and only if $\alpha$ is a
local $p$th power at ${\mathfrak Q}$ (see e.g. [Gr2,\,I.3.2.1, Ex.\,1]).

\smallskip
With this definition, for any automorphism $\tau\in {\rm Gal}(M/\Q)$ one
obtains, from $\alpha_{}^{\,\kappa} \equiv \zeta^{\,\mu}
\pmod {\mathfrak Q}$,
$\tau \alpha_{}^{\,\kappa}  \equiv \tau \zeta^\mu
\pmod {\tau {\mathfrak Q}}$,
thus:
$$\Big(\Frac{\tau \alpha}{\tau {\mathfrak Q}}\Big)_{\!\!M}=
\tau \Big(\Frac{\alpha}{{\mathfrak Q}}\Big)_{\!\!M} =
 \Big(\Frac{\alpha}{{\mathfrak Q}}\Big)^{\omega(\tau)}_{\!\!M}
 = \zeta^{\mu\,\omega(\tau)}. $$

If $\alpha \in K$, for any ${\mathfrak q}_K \div q$ in $K$ we get
$\Big(\Frac{\alpha}{{\mathfrak q}_K}\Big)_{\!\!K} =
\Big(\Frac{\alpha}{{\mathfrak Q}}\Big)_{\!\!M}$
for any ${\mathfrak Q} \div {\mathfrak q}_K$ in~$M$.

\smallskip    
In particular this implies $\Big(\Frac{\zeta}{{\mathfrak 
q}_K}\Big)_{\!\!K} = \zeta^{\kappa}$
(the symbol of $\zeta$ does not depend on the choice of ${\mathfrak 
q}_K  \div  q$).~\hfill\fin }
\end{defi}
 
\smallskip
We return to the context of  the SFLT equation
 $(u+v\,\zeta)\Z[\zeta] = {\mathfrak w}_1^p\ \,{\rm or}\ \,
{\mathfrak p}\, {\mathfrak w}_1^p$,  g.c.d.\,$(u,v)=1$ 
 (the second case corresponds to 
$p \div u\,v$ and the special case to $p \div  u+v$).

\smallskip
Put $\gamma_\omega := (u+v\,\zeta)^{e_\omega}$ for a solution
$(u,v)$ of the above SFLT equation.

\medskip 
In the context of a solution $(x,y,z)$ of Fermat${}'$s equation we will have 
analogous computations with $\gamma_\omega := 
(x+y\,\zeta)^{e_\omega}$ and the relation 
$(x+y\,\zeta)\Z[\zeta] = {\mathfrak z}_1^p$,
and also with $\gamma'_\omega := (z+y\,\zeta)^{e_\omega}$
and the relation $(z+y\,\zeta)\Z[\zeta] = {\mathfrak x}_1^p$.
Then in the first case, $\gamma''_\omega := (x+z\,\zeta)^{e_\omega}$
with the relation $(x+z\,\zeta)\Z[\zeta] = {\mathfrak y}_1^p$
can be used knowing that $z-x$ may be divisible by $p$.
In the second case, $\gamma''_\omega$ is of ${\mathfrak p}$-valuation 
1 since  $(x+z\,\zeta)\Z[\zeta] = {\mathfrak p}\, {\mathfrak y}_1^p$
and this gives a special case of the equation associated to SFLT.

\smallskip
We know, from Stickelberger, that the $\omega$-component of the 
$p$-class group of $K$ is 
trivial  (also an application of the reflection theorem, see
 [Gr2, II.5.4.6.3]); so the ideal class  $\cl({\mathfrak w}_1)^{e_\omega}$
 is trivial.\,\footnote{\,Since here the class of ${\mathfrak w}_1$ is 
 of order 1 or $p$, the choice of any representative $e_\omega$ does 
 not affect this property.}

\smallskip
Write:
$${\mathfrak w}_1^{e_\omega} =
 \delta_\omega\,\Z[\zeta],\ \, \delta_\omega \in K^\times. $$
 
 Then we have:
 $$\gamma_\omega := (u+v\,\zeta)^{e_\omega} = \varepsilon_\omega\,
 \delta_\omega^p \ \,{\rm or}\ \, (\zeta-1)^{e_\omega} \varepsilon_\omega\,
 \delta_\omega^p, $$
 where $\varepsilon_\omega\in E$. To simplify, we put $\pi := \zeta-1$.
 
\begin{lemm} {\rm (The fundamental relation).}
Let $(u,v)$ be a solution of the equation  $(u+v\,\zeta)\Z[\zeta] = 
 {\mathfrak w}_1^p$   or ${\mathfrak p}\, {\mathfrak w}_1^p$,
  g.c.d.\,$(u,v)=1$ (since the 
  cases where $u\,v\,(u+v)=0$ are obvious directly, we exclude them). 
   
    (i)  In the nonspecial cases ($u+v \not\equiv 
0 \pmod {p}$) for  $p\geq 3$, we have $\gamma_\omega = 
(\frac{u}{v}+\zeta)^{e_\omega} =
(1+\frac{v}{u}\,\zeta)^{e_\omega} = 
\big(1+\frac{v}{u+v}\,\pi \big)^{e_\omega}
\in \zeta^{\frac{v}{u+v}}\cdot  K^{\times p}$.

\smallskip
(ii) In the special case ($u+v \equiv 
0 \pmod {p}$) for $p>3$, we have 
$\gamma_\omega = (1+\frac{v}{u}\,\zeta)^{e_\omega}
\in \zeta^{\frac{1}{2}} \cdot K^{\times p}$.\,\footnote{\,Where
$\zeta^{\frac{1}{2}}$ is the unique $p$th root of unity such 
that $(\zeta^{\frac{1}{2}})^2 = \zeta$; this convention will be used in a 
systematic way in the paper.}

\smallskip
(iii) In the special case ($u+v \not\equiv 
0 \pmod {3}$) for $p=3$, then
$\gamma_\omega = (1+\frac{v}{u}\,\zeta)^{e_\omega}
\in \zeta^{\frac{1}{2} - \frac{u+v}{3\,v}}\cdot K^{\times 3}$.
\end{lemm}

\begin{proof}
    (i)  We have $\gamma_\omega = \varepsilon_\omega\, 
\delta_\omega^p$
    with $\varepsilon_\omega = \zeta^h \varepsilon^{+}$, 
$\varepsilon^{+}
    \in E^{+}$, for some $h$; then applying
    again $e_\omega$ we get
    $\gamma_\omega^{e_\omega} = \varepsilon_\omega^{e_\omega}\,
    \delta_\omega^{e_{\omega }p}
    \in \zeta^h \,\cdot\, K^{\times p}$.
    Since $e_\omega^2 \equiv e_\omega \pmod {p\,\Z[g]}$,
    any factor $A^{e_\omega^2}$ may be written $A^{e_\omega}$ up 
to a $p$th power; thus $\gamma_\omega \in\, \zeta^h\,\cdot\,K^{\times p}$.
Since $u+v\,\zeta =
(u+v)\,\big(1+\frac{v}{u+v}\,\pi \big)$,
$(u+v\,\zeta)^{e_\omega}\,\in\, \zeta^h\,\cdot\,K^{\times p}$ is
equivalent to 
$(1+\hbox{$\frac{v}{u+v}$}\,\pi)^{e_\omega}\,\in\,
\zeta^h\,\cdot\,K^{\times p}$;
then using [Gr1, Remark 3.4]:
$$\big(1+\hbox{$\frac{v}{u+v}$}\,\pi \big)^{e_\omega}\equiv
 1+\hbox{$\frac{v}{u+v}$}\,\pi  \pmod {\pi^2},$$
 we get immediately $h \equiv \frac{v}{u+v} \pmod {p}$.
     
\smallskip
Similarly we have $u+v\,\zeta =
v\, (\frac{u}{v}+\zeta) =
u\, (1+\frac{v}{u}\,\zeta)$ for which $(u+v\,\zeta)^{e_\omega} =
(\frac{u}{v}+\zeta)^{e_\omega} =
(1+\frac{v}{u}\,\zeta)^{e_\omega}$,  proving the point (i).

\smallskip
(ii)  Suppose that $u+v \equiv 0$ $\pmod {p}$; put
 $\frac{u}{v} = -1 + \lambda\,p$, then
 $\hbox{$\frac{u}{v}$} + \zeta =
 \pi +\lambda\,p = \pi\,\alpha$,
 where $\alpha := 1+ \frac{\lambda\,p}{\pi}\equiv 1 \pmod {\pi^{p-2}}$.
 
 \smallskip
 Then we get  $\gamma_\omega := (u+v\,\zeta)^{e_\omega} =
  (1+\frac{v}{u} \, \zeta)^{e_\omega}=
 (\frac{u}{v} + \zeta)^{e_\omega} =
 \pi^{e_\omega}\, \alpha^{e_\omega}$. But from 
 the relation $(u+v\,\zeta)\,\Z[\zeta] = (\pi)\,{\mathfrak w}_1^p$, 
 we obtain $(u+v\zeta)^{e_\omega} \in
 \pi^{e_\omega}\,\zeta^{h} \, K^{\times p}$,
 for some $h$, giving
 $\alpha^{e_\omega} \in \,\zeta^{h} \, K^{\times p}$ hence
$h \equiv 0 \pmod {p}$ in that case since $p>3$.
 Then $\big(1+\frac{v}{u}\,\zeta\big)^{e_\omega}  \in
 \pi^{e_\omega}\,K^{\times p}$.
 
Put $\alpha \sim \beta$ in $K^\times$
     if $\alpha\,\beta^{-1}\in K^{\times p}$.
From $(\zeta - 1)\,(\zeta + 1)= \zeta^2-1$, we get:
 $$(\zeta - 1)^{e_\omega}\,(\zeta + 1)^{e_\omega} = 
 (\zeta^2-1)^{e_\omega} = (\zeta - 1)^{s_2 e_\omega}  \sim 
  (\zeta - 1)^{2 e_\omega}, $$
  giving $(\zeta + 1)^{e_\omega}
  \sim (\zeta - 1)^{e_\omega}$.
But $\zeta + 1 = \zeta^{\frac{1}{2}} 
(\zeta^{\frac{1}{2}}+\zeta^{-\frac{1}{2}})$
yields  $(\zeta + 1)^{e_\omega}  \sim \zeta^{\frac{1}{2}}$
since $\zeta^{\frac{1}{2}}+\zeta^{-\frac{1}{2}} \in K^+$.
Then we have the relation $(\zeta - 1)^{e_\omega} \sim 
 (\zeta + 1)^{e_\omega} \sim  \zeta^{\frac{1}{2}}$, hence the 
point (ii) of the lemma.

\medskip
(iii) If $p=3$ in the special case, we get from the computations in the proof of (ii),
$\gamma_\omega =  \pi^{e_\omega}\, \alpha^{e_\omega}
 \in \pi^{e_\omega}\,\zeta^{h}\, K^{\times 3}$,
 for some $h$, with
 $\alpha = 1 + \frac{3\,\lambda}{\pi}$ and 
 $\lambda = \frac{u+v}{3\,v}$. Thus $\alpha = 1+(\zeta^2-1)
 \,\frac{u+v}{3\,v} \equiv 1 - \pi\,\frac{u+v}{3\,v} \pmod {\pi^2}$,
 giving the congruence $h \equiv -\frac{u+v}{3\,v}
 \pmod {3}$ and $\gamma_\omega \in 
 \zeta^{\frac{1}{2} -\frac{u+v}{3\,v}}.\, K^{\times 3}$.
\end{proof}
 
In the second case of SFLT we have
$\gamma_\omega \in  K^{\times p}$ (resp. $\zeta\,.\, K^{\times p}$)
if  $p \div v$ (resp. $p \div u$)
since  in this case  $\frac{v}{u+v} \equiv 0$ 
(resp. $\frac{v}{u+v} \equiv 1$) $\pmod {p}$.
Note that the condition  $u+v \equiv 0 \pmod {9}$,
when $p=3$ in the special case, is not necessarily satisfied for the SFLT
equation (use Remark~1) 
but is true when $(u,v)$ is a part of a solution $(u,y,v)$ or
$(v,y,u)$ of Fermat${}'$s equation when $3 \div y$ and more 
generally when $p>3$,  $p \div y$ (see Subsection 2.1, (i)).

\begin{cor} {\rm (Generalization of the first theorem of Furtw\"angler).}
Let $q \ne p$ be a prime number such that
 $q \div  u\,v$  for a nontrivial solution of the equation 
 $(u+v\,\zeta)\Z[\zeta] = 
 {\mathfrak w}_1^p\ \,{\rm or}\ \,{\mathfrak p}\, {\mathfrak w}_1^p$,
 g.c.d.\,$(u,v)=1$.

\smallskip\noindent
Then, in the nonspecial cases, $u\,\kappa \equiv 0$ 
(resp. $v \,\kappa \equiv 0$) $\pmod {p}$
if $q \div  u$ (resp. $q \div  v$), for $p\geq 3$;
 in the first  case we get
 $\kappa \equiv 0 \pmod {p}$.
 
\smallskip \noindent
 For $p>3$ in the special case, then $\kappa \equiv 0 \pmod {p}$.
 For $p=3$ in the special case, we get
$\frac{u - 2v}{3v}\, \kappa\equiv 0$
(resp. $\frac{v - 2u}{3v}  \,\kappa\equiv 0$)  $\pmod {3}$
if $q \div  u$ (resp. $q \div  v$); thus if $u+v \equiv 0 \pmod {9}$,
then $\kappa \equiv 0 \pmod {3}$.
If $u+v = 3\,e$, $e \not\equiv 0 \pmod {3}$, then 
$\kappa \equiv 0 \pmod {3}$ if $q \div u$ (resp. $q \div v$)
when $u \equiv e$ (resp. $v \equiv e$) $\pmod {3}$.
\end {cor}

\begin{proof}
 We have $(u+v\,\zeta)^{e_\omega} \in\, \zeta^h.\,K^{\times p}$
 with $h \equiv \frac{v}{u+v} \pmod {p}$ in the nonspecial cases,
 $p \geq 3$, $h  \equiv \frac{1}{2} \pmod {p}$ 
 in the special case, $p>3$,
 and  $h  \equiv \frac{1}{2} - \frac{u+v}{3\,v}\pmod {3}$ in the 
 special case, $p=3$.
 
 \smallskip
 Let ${\mathfrak q}_K$ be any prime ideal of $K$ dividing $q$.
 We use the $p$th power residue symbol in $K$ (see  Definition 2,
 (ii)).
 
 \smallskip
 Since $u+v\,\zeta  \equiv v\,\zeta$ (resp. $u+v\,\zeta  
 \equiv u$) $ \pmod {q}$ if $q \div u$ (resp. $q \div v$),
 we get $\Big(\frac{(u+v\,\zeta)^{e_\omega}}{{\mathfrak q}_K}\Big)_{\!\!K}
 = \zeta^\kappa$ (resp. $1$) if $q \div u$ (resp. $q \div v$); but
 we have $\Big(\frac{\zeta^h}{{\mathfrak q}_K}\Big)_{\!\!K}
 = \zeta^{\frac{v}{u+v}\,\kappa}$
(resp.  $\zeta^{\frac{1}{2}\,\kappa}$,
$\zeta^{(\frac{1}{2}-\frac{v+u}{3\,v})\,\kappa}$)
in the nonspecial cases (resp. in the special case, $p>3$, $p=3$).
 
 \smallskip
 This gives in the nonspecial cases for $q \div u$,
 $\frac{v}{u+v}\,\kappa\equiv \kappa \pmod {p}$ equivalent
 to  $\frac{u}{u+v}\,\kappa \equiv 0 \pmod {p}$,
 hence $u \,\kappa \equiv 0 \pmod {p}$.
 If $q \div v$, we get $\frac{v}{u+v}\,\kappa\equiv 0 \pmod {p}$
 giving $v\,\kappa \equiv 0 \pmod {p}$.
    
   \smallskip
The special case for $p>3$ yields to $\frac{1}{2}\,\kappa
 \equiv \kappa $ (resp. $\frac{1}{2}\,\kappa
 \equiv 0)$ $\pmod {p}$ if  $q \div u$ (resp. $q \div v$), giving 
 $\kappa \equiv 0 \pmod {p}$ in any case.
 
 \smallskip
For $p=3$ in the special case we get
 $(\frac{1}{2} - \frac{u+v}{3\,v}) \, \kappa 
 \equiv \kappa$ (resp. $(\frac{1}{2} - 
 \frac{u+v}{3\,v}) \,\kappa  \equiv 0$) $\pmod {3}$
 if $q \div u$ (resp. $q \div v$), giving  
 $\frac{u - 2v}{3}\,\kappa  \equiv 0$
(resp. $\frac{v - 2u}{3} \,\kappa \equiv 0$) $\pmod {3}$.
The case $u+v  \equiv 0 \pmod {9}$ is clear as well as the case
$u+v  \equiv \pm 3 \pmod {9}$.
\end{proof}

\begin{cor} {\rm (Generalization of the second theorem of Furtw\"angler).}
  Let $q \ne p$ be a prime number such that  $q \div  u^2 - v^2$ 
  for a nontrivial solution of the equation $(u+v\,\zeta)\Z[\zeta] = 
 {\mathfrak w}_1^p\ \,{\rm or}\ \,{\mathfrak p}\, {\mathfrak w}_1^p$,
g.c.d.\,$(u,v)=1$.

\smallskip\noindent
Then, in the nonspecial cases, $(v-u)\,\kappa \equiv 0 \pmod {p}$
for  $p \geq 3$; in parti\-cular, in the second 
case, $\kappa \equiv 0 \pmod {p}$. In the first 
case for $p=3$, the information is empty
since $u \equiv v \equiv \pm 1 \pmod {3}$.

\smallskip\noindent
 For $p>3$ in the special case,
 the information is empty. For $p=3$ in the special case we get
$\frac{u+v}{3\,v}\,\kappa \equiv 0  \pmod {3}$.
Thus  $\kappa \equiv 0 \pmod {3}$ as soon as $v+u \not\equiv 0 \pmod {9}$.
\end {cor}

\begin{proof}
 We have $(u+v\,\zeta)^{e_\omega} \in\, \zeta^h.\,K^{\times p}$
 with $h  \equiv \frac{v}{u+v} \pmod {p}$ in the nonspecial cases,
 $p\geq 3$, $h \equiv \frac{1}{2} \pmod {p}$ in the special case, $p>3$,
 and  $h  \equiv \frac{1}{2} - \frac{u+v}{3\,v}$ in the 
 special case, $p=3$.
 
 Then we have $(1+\frac{v}{u} \,\zeta)^{e_\omega}\,
 \zeta^{-\frac{1}{2}}
 \in\, \zeta^{\ov h}.\,K^{\times p}$
  with ${\ov h}  \equiv \frac{1}{2}\frac{v-u}{u+v} \pmod {p}$ in the nonspecial cases, 
 ${\ov h}  \equiv 0 \pmod {p}$ in the special case,  $p>3$,
 and  ${\ov h} \equiv -\frac{u+v}{3\,v} \pmod {3}$ in the 
 special case, $p=3$.
 
 \smallskip
 Let ${\mathfrak q}_K$ be any prime ideal of $K$ dividing $q$.
 If $q \div  u^2 - v^2$, then $\frac{v}{u} \equiv \pm 1$ $\pmod {q}$
 and we get 
 $(1+ \frac{v}{u}\,\zeta)^{e_\omega} \,\zeta^{-\frac{1}{2}}
 \equiv (1 \pm \zeta)^{e_\omega}\,\zeta^{-\frac{1}{2}} \equiv 
 1 \pmod  {{\mathfrak q}_K}$ since $(1 \pm \zeta)^{e_\omega} \sim \zeta^{\frac{1}{2}}$
 (see proof of Lemma 3).  Thus we obtain 
${\ov h} \,\kappa \equiv 0 \pmod {p}$ in every case.
 
 \smallskip
 The nonspecial cases yield to 
$\frac{v-u}{u+v} \,\kappa \equiv 0 \pmod {p}$, hence 
$\kappa \equiv 0 \pmod {p}$ if $u-v \not \equiv 0 \pmod {p}$.
Thus the case $p=3$ is empty since $u\equiv v \equiv \pm 1 \pmod {3}$.
   
\smallskip
The special case for $p>3$ is empty.
The special case for $p=3$ gives 
$\frac{v+u}{3\,v} \,\kappa \equiv 0 \pmod {3}$.
 \end{proof}
 
\subsection{Consequences of Lemma 3}

We  make the following comments on the fundamental Lemma 3 and 
its corollaries to introduce the $\omega$-SFLT equation and
suitable cyclotomic units.

\smallskip
 For {\it arbitrary} relatively prime integers $u$, $v$,
we still have (excluding the obvious cases where $u\,v\,\,(u+v)=0$):
$$\gamma_\omega := (u+v \,\zeta)^{e_\omega} = 
\big(\Frac{u}{v}+\zeta\big)^{e_\omega} =
\big(1+\Frac{v}{u}\,\zeta\big)^{e_\omega} =
\big(1+\Frac{v}{v+u}\,\pi\big)^{e_\omega}$$
and the various congruences
of Lemma~3, $\gamma_\omega \equiv \zeta^h \pmod {\pi^2}$,
with $h = \frac{v}{u+v}$ (nonspecial cases, $p\geq 3$),
$h = \frac{1}{2}$ (special case, $p>3$),
and $h = \frac{1}{2} - \frac{u+v}{3\,v}$ (special case, $p=3$).
Then we obtain  $\gamma_\omega\ \zeta^{-h}\equiv
1 \pmod {\pi^2}$, which implies easily that
 $\gamma_\omega\ \zeta^{-h}$ is a $p$-primary number
(use [Gr1, Lemma 3.15]);
but since this number is not in general the $p$th power of an
ideal it may not be a global $p$th power.\,\footnote{\,In the case where
$(u + v\, \zeta)\, \Z[\zeta]= {\mathfrak w}_1^p$, 
Lemma 3 shows that $(1+\hbox{$\frac{v}{u}$}\, 
\zeta)^{e_\omega}\,\zeta^{-h} \in K^{\times p}$; so in this
particular case, where $(1+\hbox{$\frac{v}{u}$}\, \zeta)^{e_\omega}\,
\zeta^{-h}$ is a pseudo-unit, local $p$th 
power at $p$,  we get a necessary and sufficient condition 
to get a global $p$th power.}

So, from class field theory, there exist infinitely many prime ideals
${\mathfrak q}_K$ of $K$, prime to $u\,v$,
such that $(1+\frac{v}{u}\,\zeta)^{e_\omega}\,
\zeta^{-h}$ is not a local $p$th power
at ${\mathfrak q}_K$, except if we have a counterexample $(u,v)$
to SFLT in which case such primes do not exist.

\smallskip
The $p$th power residue symbol of this number is invariant by conjugation of 
${\mathfrak q}_K$ since:
$$\big ((1+\Frac{v}{u}\,\zeta)\,\zeta^{-h}\big )^{e_\omega\,\kappa} \equiv
\zeta' \pmod {{\mathfrak q}_K}$$
implies, by conjugation by $s_k \in g$:
$$\big ((1+\Frac{v}{u}\,\zeta^k)\,\zeta^{-k\,h}\big)^{e_\omega\,\kappa}\sim
\big ((1+\Frac{v}{u}\,\zeta)\,\zeta^{-h}\big )^{k\,e_\omega\,\kappa}
 \equiv\zeta'{}^k  \pmod {s_k({\mathfrak q}_K)}$$
 equivalent (up to  $p$th powers) to:
$$\big ((1+\Frac{v}{u}\,\zeta)\,\zeta^{-h}\big )^{e_\omega\,\kappa}\equiv
\zeta'{} \pmod {s_k({\mathfrak q}_K)}; $$
so the symbol only depends on $q$,
the prime number under ${\mathfrak q}_K$ which does not divide $u\,v$.

\smallskip
We suppose  $\frac{v}{u}$  of order $n$
modulo $q$ (which is equivalent to
$q \div  \Phi_n(u,v)$ and $q \notdiv n$),
and we suppose $n$ prime to $p$.

\smallskip
Let  ${\mathfrak q}$ be a prime ideal above $q$ in $\Q(\mu_{n})$.
We have equality of the $p$th power residue symbols of
$(1+\frac{v}{u}\,\zeta)^{e_\omega}\, \zeta^{-h}$
at any ${\mathfrak q}_K$ in $K$, and
of the cyclotomic unit $(1+\xi\,\zeta)^{e_\omega}\, \zeta^{-h}$
at ${\mathfrak Q}$ in $\Q(\mu_{n})K$, where $\xi$ is a
suitable $n$th root of unity and ${\mathfrak Q}$ is any
prime ideal above ${\mathfrak q}$ in $\Q(\mu_{n})K$
($\xi$ is characterized by the congruence 
$\xi \equiv \frac{v}{u} \pmod {\mathfrak q}$
in $\Q(\mu_{n})$, and ${\mathfrak q}_K$ must be ${\mathfrak Q}\cap \Z[\zeta]$).

\smallskip
Of course $h$ is a priori unknown (but constant with respect to $q$) 
and the local study of 
$(1+\xi\,\zeta)^{e_\omega}\,\zeta^{-h}$ is uneffective 
in general, but we may use some partial 
informations, as the following ones in the context of FLT.

\smallskip
Let $(x,y,z)$ be a solution of Fermat${}'$s equation (first or second cases).

\medskip
{\bf a)} Take for instance $u:=x$, $v:=y$ (so we are in the nonspecial 
cases of SFLT) which gives $h = \frac{y}{y+x}$.

\smallskip
$ \ \ \bullet\ $ If  $\zeta$ is not a local $p$th power at 
${\mathfrak q}_K$
(which is  equivalent to  $\kappa \not\equiv 0$ $\pmod {p}$),
we will consider the $p$th power residue symbol at ${\mathfrak Q}$ of
the real cyclotomic unit
$\eta_1 := (1+\xi\,\zeta)^{e_\omega}\,\zeta^{-\frac{1}{2}}$
 (see Definition 4 in Subsection 3.1)  which must be  that of $\zeta^{\,h-\frac{1}{2}} = 
\zeta^{\frac{1}{2}\frac{y-x}{y+x}}$.
For FLT we have some informations on the differences 
like $y-x$, $y-z$, which are  prime to $p$ for $p>3$
or $p=3$ in the second case; in these cases a contradiction to
the existence of such a solution of Fermat${}'$s equation
is that the unit
$(1+\xi\,\zeta)^{e_\omega} \,\zeta^{-\frac{1}{2}}$
be a local $p$th power at ${\mathfrak Q}$ or does not give the 
``good'' symbol.

\smallskip
 For $p=3$ in the first case, we know that $x \equiv
y \equiv z \equiv \pm 1 \pmod {3}$; so a contradiction is that this 
unit be not a local $3$th power at ${\mathfrak Q}$.

\smallskip
$ \ \ \bullet\ $ If  $\zeta$ is a local $p$th power at ${\mathfrak 
q}_K$ (which is  equivalent to  $\kappa \equiv 0 \pmod {p}$),
a contradiction is that the unit $\eta_1$ be not a local $p$th 
power at ${\mathfrak Q}$. 

\medskip
{\bf b)} 
In the second case of FLT
($p\div y$) with $u=x$, $v=z$ (special case of SFLT)
we have different but similar reasonings using the value of 
$h-\frac{1}{2}$ given by Lemma~3 for  $p\geq 3$ since
$x+z \equiv 0 \pmod {9}$ when $p=3$.

\medskip
The hope in this attempt is that, 
the arithmetical properties of the fields 
$\Q(\mu_n) \subseteq \Q(\mu_{q-1})$ being
a priori independent of the SFLT problem, they may
give valuable indications on the local properties of
$\eta_1 = (1+\xi\,\zeta)^{e_\omega}\,\zeta^{-\frac{1}{2}}$,
especially in an analytic point of view. In some sense the fields
$\Q(\mu_{q-1})$ will play the role of governing fields for this 
problem. Indeed, under a solution of the SFLT equation, the residue symbol 
above ${\mathfrak q}$ of this unit is, {\it independently of the choice of $q$},
equal to the symbol of a  {\it constant} power of $\zeta$, which may be absurd.

\smallskip
These cyclotomic fields have been
introduced  by Vandiver in some papers as [Van1, Van2, Van3] to 
generalize classical results of Kummer and some
congruences giving Furtw\"angler${}'$s theorems and Wie\-ferich${}'$s
criteria; these papers essentially depend on the Stickelberger 
element $S := \frac{1}{p} \sum_{k=1}^{p-1} k\,s_k^{-1}$, related to the
generalized Bernoulli numbers and the annihilation
of the $p$-class group of $K$.

Meanwhile some of the relations of Lemma 3 are
considered by Vandiver for other purposes than our${}'$s.
In Vandiver  papers, analytic results (like
\v Cebotarev${}'$s theorem) or class field theory  
are not used, and it seems that no method of contradiction
can be deduced from these computations which are essentially {\it 
local at $p$}, and it has been explained in [Gr1] the probable 
inefficiency of  such local studies.
Our present work is mainly global and does not concern the arithmetic 
of $K$ as in the historical researches.

\begin{lemm} The equation $(u+v\,\zeta)\,\Z[\zeta] = 
{\mathfrak w}_1^p \ \, {\rm or}\ \, {\mathfrak p}\,{\mathfrak w}_1^p$,
in integers $u$, $v$, with g.c.d.\,$(u,v)=1$, is  equivalent to the  equation
$(u+v\,\zeta)^{e_\omega} \in  \zeta^{h}\,K^{\times p}$
with $h \equiv \frac{v}{u+v} \pmod {p}$ in the nonspecial cases,
$p \geq 3$, $h  \equiv \frac{1}{2} \pmod {p}$ 
in the special case, $p>3$, and  $h \equiv \frac{1}{2}- 
\frac{u+v}{3\,v}\pmod {3}$ in the  special case, $p=3$.
\end{lemm}

\begin{proof} A direction being proved (Lemma 3),
let ${\mathfrak l}\ne {\mathfrak p}$ be any prime ideal dividing
the ideal $(u+v\,\zeta)\,\Z[\zeta]$; the use of
the congruences $u+v\,\zeta\equiv 0 \pmod {\mathfrak l}$ 
implies that  $(u+v\,\zeta)\,\Z[\zeta] =  {\mathfrak p}^\delta\,
\prod_{\ell} {\mathfrak l}^{\alpha_\ell}$, $\delta = 0$ or 1,
$\alpha_\ell \geq 1$,  for distinct prime 
numbers $\ell$ with a single  ${\mathfrak l} \div \ell$ (otherwise,
using appropriate conjugations, we get 
$u\equiv v\equiv 0 \pmod {\mathfrak l}$). Moreover, it shows
that ${\mathfrak l}$ is of degree 1 and that
$g$ operates transitively on the set of conjugates of
${\mathfrak l}$; hence, since $ {\mathfrak p}^{e_\omega}=\Z[\zeta]$ 
and since $(u+v\,\zeta)^{e_\omega}\,\Z[\zeta] = 
\prod_{\ell} {\mathfrak l}^{\alpha_\ell \,e_\omega}$ is a $p$th power
by assumption, we get $\alpha_\ell \equiv 0 \pmod {p}$ for all $\ell$.
\end{proof}

From any relation $(u'+v'\,\zeta)^{e_\omega} \in 
\langle\,\zeta\,\rangle\,.\,K^{\times p}$, $u', v' \in \Z$, we deduce the 
solution $(u,v) := \frac{1}{{\rm g.c.d.}\,(u',v')}\,(u',v')$ of the 
SFLT equation with a unique $h$.

\medskip
 We call the second equation the $\omega$-SFLT  equation;
the corres\-ponding form of the SFLT conjecture for $p>3$ seems reasonable
as soon as $p$ is sufficiently large  since
it enunciates (for $u\,v\,(u^2-v^2)\ne 0$) that there exists a sum $\sum_{k=1}^{p-1} 
\lambda_k\,\zeta^k$, $\lambda_k\in\Q$, whose $p$th power is of the 
form:
$$(u\,\zeta^{-\frac{v}{u+v}} + 
v\,\,\zeta^{\frac{u}{u+v}})^{e_\omega}\ \,
{\rm (resp.\ of\ the\ form}\ \, (u\,\zeta^{-\frac{1}{2}} + 
v\,\zeta^{\frac{1}{2}})^{e_\omega}), $$
depending on two coefficients  $u$, $v$ instead of $p-1$ in general.
It will be interesting to have the response at least for $p=5$.

\smallskip
So for $p>3$ the truth of SFLT would imply  FLT; 
in this paper we concentrate our attention mainly on SFLT, using 
the simpler $\omega$-SFLT context  which does not concern mainly the 
arithmetic of $K$, the nerve center of the unsuccessful classical theory.

\smallskip\smallskip
For a recent critical history on FLT see [Co]. For some complements
on these cyclotomic technics, see [He1, He2, Ter, Ri].

\section {Introduction of the governing fields $\Q(\mu_{q-1})$}

\smallskip
\subsection {Furtw\"angler and Vandiver revisited}

Consider the SFLT equation:
$$(u+v\,\zeta)\,\Z[\zeta] = 
{\mathfrak w}_1^p \ \, {\rm or}\ \, {\mathfrak p}\,{\mathfrak w}_1^p, $$
in  integers $u$, $v$, with g.c.d.\,$(u,v)=1$, independently of FLT.

\smallskip
Let $q$ be a prime number such that  $q \notdiv u\,v$ and such that  
$\frac{v}{u}$
is of order $n$ modulo $q$ (which is equivalent from Lemma 2 to $q\notdiv n$ and
$q \div \Phi_n(u,v)$), with  $n$ prime to $p$.

 \smallskip
In another point of view, for a given $n$ prime to $p$,
the primes  $q \div  \Phi_n(u,v)$ are solutions (i.e., $\frac{v}{u}$
is of order $n$ modulo $q$), if and only if $q \notdiv n$. In 
practice the condition
$q \notdiv n$ is always satisfied because we are only concerned with 
large primes $q$, so that $q\equiv 1 \pmod {n}$.
 
 \smallskip
 Consider the 
following diagram, where $L :=\Q(\mu_n)$, $M:= LK$, and  $G = {\rm 
Gal}(M/L) \simeq g$ (we have $L\cap K = \Q$):

\medskip
\unitlength=0.6cm
$$\vbox{\hbox{\begin{picture}(12.1,3.2)
\put(4.7,3.0){\line(1,0){2.3}}
\put(4.5,0.0){\line(1,0){2.3}}
\put(4.00,1.9){\line(0,1){0.75}}
\put(4.00,0.4){\line(0,1){2.}}
\put(7.5,1.9){\line(0,1){0.75}}
\put(7.5,0.4){\line(0,1){2.}}
\put(7.2,2.9){$M$}%
\put(1.8,2.9){$L\!=\!\Q(\mu_n)$}%
\put(3.8,-0.1){$\Q$}%
\put(7.2,-0.1){$K\!=\!\Q(\zeta)$}%
\put(5.6,3.25){$G$}%
\put(5.6,0.25){$g$}%
\end{picture}   }} $$

\smallskip
\begin{defi}{\rm
The prime $q \equiv 1 \pmod {n}$ being totally split in $L/\Q$, if 
${\mathfrak q}$ is a prime ideal of $L$ over $q$,
there exists a {\it unique} primitive $n$th root of unity $\xi$ 
such that $\xi \equiv \hbox{$\frac{v}{u}$} \pmod {\mathfrak q}$.
Reciprocally, if $\xi$ is a primitive $n$th root of unity, there 
exists a {\it unique} prime ideal ${\mathfrak q}$  of $L$ over $q$ such that
$\xi \equiv \frac{v}{u} \pmod {\mathfrak q}$.
%
This ideal is $(q, u\,\xi-v)$ and will also be denoted ${\mathfrak q}_\xi$
(it depends on $u$, $v$).

\smallskip
We associate with $q$ (for $u$, $v$ fixed)
a pair $(\xi, {\mathfrak q})$ where the prime ideal 
${\mathfrak q} := {\mathfrak q}_\xi$ above $q$
and the primitive $n$th root of unity  $\xi$
are characterized by the congruence $\xi \equiv
\frac{v}{u} \pmod {\mathfrak q}$ in $L$.

\smallskip
This pair is defined up to $\Q$-conjugation since $\xi \equiv
\frac{v}{u} \pmod {{\mathfrak q}_\xi}$ is equivalent to $\xi^t \equiv
\frac{v}{u} \pmod {{\mathfrak q}_\xi^t  = {\mathfrak q}_{\xi^t}}$,
for all $t \in {\rm Gal}(L/\Q)$. 
We obtain an equivalence relation.
The class only depends on $q$ for $u$, $v$ given.~\hfill\fin}
\end{defi}

\smallskip
Taking a representative pair, we will fix $\xi$ 
(for instance  $\xi = {\rm exp}(2i\pi/n)$) which defines
${\mathfrak q}_\xi$.

\smallskip
Since $\frac{v}{u}$ modulo $q$ is unknown 
but well-defined, we must note  that, in what follows, the class is 
uneffective
among $\phi(n)$ possible classes, and for each $n$ prime to $p$ 
dividing $q-1$.
This explains that, in some circumstances, we will have to take $q\not\equiv 1 
\pmod {p}$ since, if not, it is not possible to assert that $\frac{v}{u}$ is of order 
modulo $q$ prime to $p$.

\begin{defi}{\rm For the given $n$th root of unity $\xi$, $n\not\equiv 
0 \pmod {p}$, we consider the cyclotomic number of $M$ associated to 
$\xi$:\,\footnote{\,We know that $1+\xi\,\zeta$ is a 
(cyclotomic) unit except if $-\xi\zeta$ is of prime power order, 
which is the case if and only if $\xi = -1$ (i.e., $n=2$) in which
case $1 + \xi\, \zeta$ is a $p$-unit.}
$$\eta := \eta(\xi) := (1+\xi\,\zeta) \,\zeta^{-\frac{1}{2}} \in M, $$
where $\zeta^{\frac{1}{2}}$ is the unique $p$th root of unity such 
that $(\zeta^{\frac{1}{2}})^2 = \zeta$ (so, the exponent
$\frac{1}{2}$ is seen as a $p$-adic integer or as an element of
$(\Z/p\Z)^\times$). This is coherent with the context of $p$-Kummer 
theory above~$M$. 

Then we put:
$$\eta_1 := \eta^{e_\omega} = (1+\xi\,\zeta)^{e_\omega} \,\zeta^{-\frac{1}{2}}
\in M; $$
we have  $\eta_1 \in M^+$,
where $M^+$ is the maximal real subfield of $M$:
indeed, if $c$ is the complex conjugation, then:
$$\eta_1^c = (1+\xi^{-1}\,\zeta^{-1})^{e_\omega} \,\zeta^{\frac{1}{2}}
=\big( (1+\xi\,\zeta)\xi^{-1}\,\zeta^{-1} \,\zeta^{\frac{1}{2}}\big)^{e_\omega}
= (1+\xi\,\zeta)^{e_\omega} \,\zeta^{-\frac{1}{2}} = \eta_1, $$
since  $\xi^{e_\omega} = 1$ and
$\zeta'{}^{e_\omega} = \zeta'$ for any $\zeta' \in \mu_p$.~\hfill\fin}
\end{defi}

\smallskip
We note that $\eta_1$ is  a cyclotomic unit and 
that  $\eta_1 \equiv 1 \pmod {\pi\,Z_M}$.

\smallskip
Starting from $\xi \equiv \frac{v}{u} \pmod {\mathfrak q}$ and
extending ${\mathfrak q}$ to $M$ we obtain:\,\footnote{\,Warning:
if for instance $\frac{v}{u}$ is of order $d p$ modulo $q$,
with $p\notdiv d$, $q$ is totally split in $M/\Q$ and we have the 
congruence
$\frac{v}{u} \equiv \xi=: \psi\,\zeta_1 \pmod {\mathfrak Q}$,
for some ${\mathfrak Q} \div q$ in $M$,
where $\psi$ is of order $d$ and $\zeta_1$ of order $p$; but in the relation
$1 + \frac{v}{u} \,\zeta \equiv 1 + \xi\,\zeta \pmod {\mathfrak Q}$,
the root $\xi=\psi\,\zeta_1$ is not invariant by $G$
so that the congruence $\big(1 + \frac{v}{u} \,\zeta\big)^{e_\omega} \equiv 
(1 + \xi\,\zeta)^{e_\omega} \pmod {\mathfrak Q}$ does not exist.
From $\gamma := 1 + \frac{v}{u} \,\zeta$ we get $s_k(\gamma) := 1 + 
\frac{v}{u} \,\zeta^k$ and if $e_\omega = \sum_k u_k s_k$ we obtain 
instead
$\gamma^{e_\omega} = \big(1 + \frac{v}{u} \,\zeta\big)^{e_\omega} \equiv 
\prod_k (1 + \psi\,\zeta_1\,\zeta^k)^{u_k} \pmod {\mathfrak Q}$,
in which  the term $1 + \psi$ is not always a 
cyclotomic unit (see [Que]).}
$$\eta_1 \equiv \big(1+\hbox{$\frac{v}{u}$}\,\zeta \big)^{e_\omega}\,\zeta^{-\frac{1}{2}}
\pmod {\prd_{{\mathfrak Q}  \div  {\mathfrak q}}{\mathfrak Q}} .$$

We note that these prime ideals ${\mathfrak Q}$ of $M$ may be written
${\mathfrak Q}_\xi$ since they are above ${\mathfrak q}_\xi$; for 
$\xi$ fixed, they are conjugated by the elements of $G$.

\smallskip
From Lemma 3, we get
$\big(1+\frac{v}{u}\,\zeta\big)^{e_\omega} = 
\zeta^{\frac{v}{v+u}}\,\cdot\,\delta_\omega^p$ (nonspecial cases, 
$p\geq 3$) \,or\,
$\zeta^{\frac{1}{2}}\,\cdot\,\delta_\omega^p$  (special case, $p>3$)
 \, or\, $\zeta_{}^{\frac{1}{2}-\frac{v+u}{3\,v}}\,\cdot\, \delta_\omega^3$ 
 (special case, $p=3$), with
$\delta_\omega \in K^\times$,  giving
$\eta_1 \equiv \zeta_{}^{\frac{1}{2}\frac{v-u}{v+u}}\,\cdot\,\delta_\omega^p
\  \ {\rm or}\ \ \delta_\omega^p\  \ {\rm or}\ \ 
\zeta_{}^{\frac{1}{2}\,\frac{v+u}{3\,v}}\,\cdot\,
\delta_\omega^3 \pmod {\prod_{{\mathfrak Q}\div {\mathfrak q}}{\mathfrak Q}}$.

\smallskip
From the congruences on $\eta_1$, Definition 2,  Corollaries 2 and 3,
we then have obtained in the only context of SFLT
the following result which includes the case $p=3$:

\smallskip
\begin{theo} \!\! Let $p$ be a prime number, $p\geq 3$. Suppose given the
relation $(u+v \,\zeta)\,\Z[\zeta] = {\mathfrak w}_1^p \ {\rm or}\ 
{\mathfrak p}\,{\mathfrak w}_1^p$ in coprime integers $u,\,v$,
where ${\mathfrak w}_1$ is an ideal  of $K:=\Q(\zeta)$,
$\zeta^p=1$, $\zeta \ne 1$, and ${\mathfrak p}:= (\zeta-1)\,\Z[\zeta]$.
    
\smallskip\noindent
Let $q \ne p$, $q \notdiv u\,v$, be a prime number such that 
$\frac{v}{u}$ is of order
$n$ modulo~$q$, with $n$ prime to $p$;
 put $\eta := (1+\xi\,\zeta) \,\zeta^{-\frac{1}{2}}$, 
$\eta_1 := \eta^{e_\omega}$, where $\xi$ is a primitive $n${\rm th} root of 
unity (see Definition 4). Put ${\mathfrak q} := (q, u\,\xi-v)$ in $L:=\Q(\mu_n)$.

\smallskip\noindent
We get in  $M:=LK$:

\smallskip
$\Big(\Frac{\eta_1}{{\mathfrak Q}}\Big)_{\!\!M} =
\zeta^{\frac{1}{2}\,\frac{v-u}{v+u}\,\kappa}$,
 $\forall {\mathfrak Q}  \div  {\mathfrak q}$, in the
 nonspecial cases ($p \notdiv u+v$),  $p\geq 3$,\,\footnote{\,In the
 first case of SFLT for $p>3$ we may have
 $u-v \equiv0 \pmod {p}$ (then $u-v \equiv 0 \pmod {p^2}$) contrary to FLT with
 $v:=y$ and $u:= x$ or $z$. For $p=3$, $u-v \equiv 0 \pmod {9}$
 in the first case.}
 
 \smallskip
$\Big(\Frac{\eta_1}{{\mathfrak Q}}\Big)_{\!\!M} = 1$,
  $\forall {\mathfrak Q}  \div  {\mathfrak q}$, in the 
special case ($p  \div  u+v$),  $p>3$,
 
\smallskip
$\Big(\Frac{\eta_1}{{\mathfrak Q}}\Big)_{\!\!M} = 
\zeta^{\frac{1}{2}\,\frac{v+u}{3\,v}\,\kappa}$,
 $\forall {\mathfrak Q}  \div  {\mathfrak q}$, in the 
special case, $p=3$.~\hfill\fin
\end{theo}

These relations show that $\Big(\Frac{\eta_1}{{\mathfrak Q}}\Big)_{\!\!M}$
only depends on the Fermat quotient of $q$ once $u$, $v$ are given.
Note that the class of
the pairs $(\eta_1^t, {\mathfrak Q}^t)$, $t \in {\rm Gal}(M/K)$,
for any choice of ${\mathfrak Q} \div  {\mathfrak q}$ in $M$, 
corresponds
canonically to the class of the $(\xi^t, {\mathfrak q}^t)$, since we 
have the relation $\Big(\Frac{\eta_1}{{\mathfrak Q}}\Big)_{\!\!M}^t =
\Big(\Frac{\eta_1}{{\mathfrak Q}}\Big)_{\!\!M} =
\Big(\Frac{\eta_1^t}{{\mathfrak Q}^t}\Big)_{\!\!M}$,
where ${\mathfrak Q}^t \div  {\mathfrak q}^t$,
and $\eta_1^t = (1+ \xi^t\,\zeta)^{e_\omega} \,\zeta^{-\frac{1}{2}}$
(see Definitions 2, 3, and 4).

The symbol $\Big(\Frac{\eta_1}{{\mathfrak Q}}\Big)_{\!\!M}$
may be different from
$\Big(\Frac{\eta_1}{{\mathfrak Q}^t}\Big)_{\!\!M}$, $t \ne 1$,
since there is no local information on
$\Frac{1+\xi^t\,\zeta}{1+\xi\,\zeta}$. But as we have seen,
$\Big(\Frac{\eta_1}{{\mathfrak Q}}\Big)_{\!\!M} =
\Big(\Frac{\eta_1}{s\, {\mathfrak Q}}\Big)_{\!\!M}$
for any $s \in G$.

\begin{rema} {\rm Since for g.c.d.\,$(u,v) = 1$ the
equation $(u+v\,\zeta)\,\Z[\zeta] ={\mathfrak 
w}_1^p$ or ${\mathfrak p} \,{\mathfrak w}_1^p$ is equivalent to 
the equation
$\No_{K/\Q}(u+v\,\zeta) = w_1^p$ or $p\, w_1^p$, we deduce from 
$u+v\,\zeta \equiv u\,(1 + \xi\,\zeta) \pmod {\mathfrak Q}$
for all ${\mathfrak Q} \div {\mathfrak q}$, that (the case $n= 2$ 
giving   $\Big(\Frac{u}{{\mathfrak q}_K}\Big)_{\!\!K} = 
\Big(\Frac{v}{{\mathfrak q}_K}\Big)_{\!\!K} = 
\Big(\Frac{p}{{\mathfrak q}_K}\Big)_{\!\!K}$ in the nonspecial cases,
 $=1$ otherwise):
$$\No_{M/L}(u+v\,\zeta) \equiv \No_{M/L}(u\, (1+\xi\,\zeta)) \equiv 
u^{p-1}\,\Frac{1+\xi^p}{1+\xi} = u^{p-1}\,(1+\xi)^{t_p-1}
\!\!\!\!\pmod {\!\mathfrak Q}, $$ for all ${\mathfrak Q} \div {\mathfrak q}$,
where $t_p$ is the Frobenius automorphism of $p$ in $L$.
This gives:
$$\Big(\Frac{(1+\xi)^{t_p-1}}{{\mathfrak Q}}\Big)_{\!\!M} =
\Big(\Frac{u}{{\mathfrak Q}}\Big)_{\!\!M} = 
\Big(\Frac{v}{{\mathfrak Q}}\Big)_{\!\!M}
\ \,\Big({\rm resp.}\,\ = \Big(\Frac{pu}{{\mathfrak Q}}\Big)_{\!\!M} = 
\Big(\Frac{pv}{{\mathfrak Q}}\Big)_{\!\!M} \Big),$$
in the nonspecial cases (resp. the special case),
 for all ${\mathfrak Q} \div {\mathfrak q}$, 
with ${\mathfrak q} = {\mathfrak q}_\xi$.}~\hfill\fin
\end{rema}
    
From a solution $(x, y, z)$ of Fermat${}'$s equation, 
we get the three relations (see Subsection 2.1):
$$(x+y\,\zeta)\,\Z[\zeta] = {\mathfrak z}_1^p,\ \ 
(z+y\,\zeta)\,\Z[\zeta] =
{\mathfrak x}_1^p,\ \ (x+z\,\zeta)\,\Z[\zeta] ={\mathfrak y}_1^p\ \,
{\rm or}\ \,{\mathfrak p} \,{\mathfrak y}_1^p. $$

For $p>3$, the conditions:
$$p\notdiv x^2-y^2,\  p\notdiv z^2-y^2,
\  p\notdiv x+z\ ({\rm resp.}\  p\notdiv x-z), $$
in  the first (resp. second)  case, and the conditions:
$$p\notdiv x^2-y^2, \ \ p\notdiv z^2-y^2,$$
in the  second\ case, are satisfied by choice of the notations (see Lemma 1).
 
 \smallskip
  If $n\leq 2$ about the relation 
  $(x+y\,\zeta)\,\Z[\zeta] = {\mathfrak z}_1^p$ 
or $(z+y\,\zeta)\,\Z[\zeta] = {\mathfrak x}_1^p$ 
  (i.e.,  $q \div  x^2-y^2$ or  $q \div  z^2-y^2$), then
$M=K$, ${\mathfrak Q}={\mathfrak q}_K \div q$ in $K$.
Since $\eta_1 = (1 \pm \zeta)^{e_\omega}\,\zeta^{-\frac{1}{2}}\in 
K^{\times p}$, we get:
$$\zeta^{\frac{1}{2}\,\frac{y-x}{y+x}\,\kappa}=1 \ \,{\rm or}\ \, 
\zeta^{\frac{1}{2}\,\frac{y-z}{y+z}\,\kappa}=1.$$

Then these two values of $n$ give again the second theorem of Furtw\"angler 
[Fur] in the context of FLT for $p>3$,
that is the fact that  when $q \div  x^2-y^2$  or $q \div  z^2-y^2$,
then $\zeta^{\kappa} =1$, which means  $\kappa \equiv 0 \pmod {p}$ (see 
Corollaries 2 and 3  generalizing the FLT situation to SFLT).

 \smallskip
 We have the same conclusion
in the first case of FLT, with the supplementary
condition $p\notdiv x-z$, if $q \div  x^2-z^2$
(in the second case of FLT, this does not work for $(x, z)$ since 
$p \div  x+z$).

\begin{rema}{\rm (Furtw\"angler${}'$s theorems and FLT;
    see e.g. [Gr1, Appendix]  or [Ri,\,IX,\,3]).}
    {\rm  Let $(x,y,z)$ be a solution of Fermat${}'$s equation for $p>3$,
 under the conditions of Lemma 1.

\smallskip
(i) Recall that  the first theorem of Furtw\"angler giving Wieferich
criteria is that for any prime number $q \ne p$,
if $q \div  x$ or $z$ (or $y$ in the first case),
then $\kappa \equiv 0 \pmod {p}$.

\smallskip
Of course, if  $q \div  x+y$ or $z+y$ (or $x+z$ in the first case), 
then  from Subsection 2.1, (i) with obvious notations,
 $q  \div  z_0$ or $x_0$ (or $y_0$ in the first case), giving 
$\kappa \equiv 0 \pmod {p}$ (from  the first theorem of Furtw\"angler)
what we can call the first part of the second theorem of 
Furtw\"angler, the second part being that if  
$q \div  x-y$ or $z-y$ (or $x-z$ in the second case), 
then $\kappa \equiv 0 \pmod {p}$.

\smallskip
 (ii)   If  $q \div  x$ or $z$ (or $y$ in the first case)
 when $q\not\equiv 1 \pmod {p}$, then 
 from Subsection 2.1, (iii),
    $q  \div  x_0$ or $z_0$ (or $y_0$ in the first case).
    Then we deduce that $q^p \div  y+z =x_0^p$
     or $x+y = z_0^p$ (or $x+z = y_0^p$ in the first case).
    This means, since $q\notdiv y\,z$ or $x\,y$ (or $x\,z$ in the 
    first  case), that $\frac{y}{z}$
    or $\frac{y}{x}$ (or $\frac{x}{z}$ in the first case) is of order 
    2 modulo $q$, giving again the first part of
    the second theorem of Furtw\"angler and
   $\kappa \equiv 0 \pmod {p}$.
    
 \smallskip
 The two results are not independent in the case $q\not\equiv 1 \pmod {p}$.
 For some other remarks on  Furtw\"angler${}'$s theorems, see [Que].

\smallskip
(iii) So, if we choose $q\not\equiv 1 \pmod {p}$ such that $\kappa \not\equiv 
0 \pmod {p}$, this implies that $q \notdiv xyz$ in the first case 
of FLT, and $q\notdiv xz$ in the second case of FLT. 
Thus, under these assumptions on $q$, the hypothesis $q \notdiv xyz$ (in the 
first case) or $q\notdiv xz$ (in the second case) are useless for the 
development of our method and give effective tests in practice.

\smallskip
It remains the case $q  \div  y$ in the second case ($p  \div y$). When
$q\not\equiv 1 \pmod {p}$, $q  \div  y_0$, then $q  \div  x+z$; 
we obtain that $q \notdiv x\,z$ and $q  \div  x+z$ but we cannot 
conclude, except that the root $\xi''$ associated to $\frac{x}{z}$ is
$-1$.
To eliminate the case $q \div y$ in the second case we must 
suppose $q$ large enough, which is not effective.

\smallskip
(iv) In any case of FLT we have the following result (see [Ri, IV.3]). 
If $q \ne p$ divides $y$ and does not divide $x+z$ then
$q \equiv 1 \pmod {p^2}$. Indeed, since $q \notdiv x+z = y_0^p$ or
$p^{\nu\,p-1} y_0^p$, we have $q \div  y_1$ and $q = 1+d\,p$. 

\smallskip
Suppose that $p \notdiv d$; since $y+x = z_0^p$ and $y+z = x_0^p$,
we see that $x$ and $z$ are $p$th powers modulo $q$ and that
$x^d \equiv z^d \equiv 1 \pmod {q}$ giving
$x^d - z^d \equiv 0$ with $x^p + z^p \equiv 0$ $\pmod {q}$.
Since $d$ is even this may be written 
$x^d  \equiv (- z)^d$ and $x^p \equiv  (-z)^p$ $\pmod {q}$
with g.c.d.\,$(d,p) = 1$ which yields to $x \equiv -z \pmod {q}$ 
(absurd). So $q \equiv 1 \pmod {p^2}$.

\smallskip
This result is valid by cyclic permutation of $x$, $y$, $z$, only 
in the first case of FLT since $p$ (in $p^{\nu\,p-1} y_0^p$) may not 
be a $p$th power modulo $q$.}~\hfill\fin
\end{rema}

If we suppose that $(x, y, z)$ (with the choices of Lemma 1)
is a solution of Fermat${}'$s equation, we obtain, from Theorem 1
and the fact that in the second case for $p=3$, $x+z
\equiv 0 \pmod {9}$ (special case $(u,v) =
(x,z)$ with $u+v \equiv 0 \pmod {9}$):

\begin{cor}  Suppose that the prime $q \ne p$ is given such that
 $q \notdiv x y z$ and  such that $\frac{y}{x}$, $\frac{y}{z}$,
$\frac{x}{z}$ are of orders $n$, $n'$, $n''$ (modulo $q$) prime to $p$.

\smallskip\noindent
Let $\xi$, $\xi'$, $\xi''$ in $\Q(\mu_{q-1})$,
of orders $n$, $n'$, $n''$, and let
 ${\mathfrak q}$, ${\mathfrak q}'$,
${\mathfrak q}''$ dividing $q$ in $L=\Q(\mu_{n})$,
$L'=\Q(\mu_{n'})$, $L''=\Q(\mu_{n''})$,
built from $\frac{y}{x}$, $\frac{y}{z}$, $\frac{x}{z}$; then
consider the corresponding cyclotomic units
$\eta_1$, $\eta'_1$, $\eta''_1$. We then have:

\smallskip
  (i) First case of FLT for $p>3$:
$$\Big(\frac{\eta_1}{{\mathfrak
Q}}\Big)_{\!\!M}=\zeta^{\frac{1}{2}\,\frac{y-x}{y+x}\,\kappa},\ \
\Big(\frac{\eta'_1}{{\mathfrak
Q'}}\Big)_{\!\!M'}=\zeta^{\frac{1}{2}\,\frac{y-z}{y+z}\,\kappa},\ \
\Big(\frac{\eta''_1}{{\mathfrak
Q''}}\Big)_{\!\!M''}=\zeta^{\frac{1}{2}\,\frac{x-z}{x+z}\,\kappa}, $$
with $y-x \not\equiv 0$ and $y-z \not\equiv0 \pmod {p}$.%
\,\footnote{\,Recall that for $p>3$ we have no information
on $x-z$ modulo $p$ in the first case, so that we cannot consider 
the third symbol in some reasonings using the above property.}

\smallskip
 (ii) First case of FLT for $p=3$:
 $$\Big(\frac{\eta_1}{{\mathfrak Q}}\Big)_{\!\!M}=
 \Big(\frac{\eta'_1}{{\mathfrak Q'}}\Big)_{\!\!M'}=
 \Big(\frac{\eta''_1}{{\mathfrak Q''}}\Big)_{\!\!M''}=1. $$
 
 (iii) Second case of FLT for $p \geq 3$ ($y \equiv x+z \equiv 0 
 \pmod {p}$):
 $$\Big(\frac{\eta_1}{{\mathfrak Q}}\Big)_{\!\!M}=
 \zeta^{-\frac{1}{2}\,\kappa},\ \
 \Big(\frac{\eta'_1}{{\mathfrak Q'}}\Big)_{\!\!M'}=
 \zeta^{-\frac{1}{2}\,\kappa},\ \
 \Big(\frac{\eta''_1}{{\mathfrak Q''}}\Big)_{\!\!M''}= 1. $$
 $$\vspace{-1.1cm}$$\hfill\fin$$\vspace{-0.5cm}$$
 \end{cor}

\begin{rema}{\rm  (a)
 Suppose that we are in the first case of FLT for $p>3$; let
$q \ne p$ be a prime number such that
 $\kappa \not\equiv 0 \pmod {p}$,
 and let $n$ and $n'$ be the orders of $\frac{y}{x}$
 and $\frac{y}{z}$ modulo $q$; we suppose that $p \notdiv n\,n'$
 (we have $n, n' > 2$ from the second theorem 
 of Furw\"angler, and  from Remark 3, (i), on the first theorem of 
 Furtw\"angler, we know that $q \notdiv xyz$).
 
 \smallskip
  If we find, with independent reasons, that at least one of the 
symbols $\Big(\Frac{\eta_1}{{\mathfrak Q}}\Big)_{\!\!M}$ or
  $\Big(\Frac{\eta'_1}{{\mathfrak Q'}}\Big)_{\!\!M'}$ is trivial,
 this is absurd since by definition
  $x - y \not\equiv 0$ and $z - y \not\equiv 0 \pmod {p}$
  under a solution of Fermat${}'$s equation (cf. (i)).

 \smallskip
  The reasoning on the
  third symbol does not work since $x-z$ can be divisible by $p$.

   \smallskip
 (b) For $p=3$ in the first case, all the right members are trivial 
under a solution of the first case of FLT and the above reasoning is 
   different but a contradiction arises as soon as an independent 
fact  implies the nontriviality of one of these symbols  (cf. (ii)).
   
   \smallskip
 (c)  In the second case for  $p\geq 3$, when  $\kappa \not\equiv 0 \pmod {p}$,
 we  know that $q \notdiv xz$.
 Since $p \notdiv n\,n'$, we deduce that $p \notdiv n''$.
 The symbol $\Big(\Frac{\eta''_1}{{\mathfrak Q''}}\Big)_{\!\!M''}$ is trivial under a 
solution  of Fermat${}'$s equation (cf. (iii)) and a contradiction arises if not.
    
  \smallskip
To have a similar reasoning as for the first case
with the two other nontri\-vial symbols associated to $\xi$ and $\xi'$, we need 
the condition $q \notdiv y$, so that we must 
suppose $q$ large enough (in practice, to get a contradiction,
we need the existence of infinitely 
many $q$ such that at least one of the symbols
$\Big(\Frac{\eta_1}{{\mathfrak Q}}\Big)_{\!\!M}$,
  $\Big(\Frac{\eta'_1}{{\mathfrak Q'}}\Big)_{\!\!M'}$ is trivial).

\medskip 
(d) If  $\kappa \equiv 0 \pmod {p}$,
in any case all the symbols are trivial under a solution of Fermat${}'$s 
equation.
So a contradiction supposes that for infinitely many such $q$ we get,
independently, nontrivial symbols.

\smallskip 
(e)  We can use the above remarks to give the following reciprocal 
statements, for $p>3$ to simplify; we suppose that any solution
$(x,y,z)$ of Fermat${}'$s equation satisfies the conventions of
Lemma 1.

 \smallskip
Let $\xi$ be a primitive $n$th root of unity with $n \not\equiv 0 \pmod {p}$,
let $\eta_1:=  (1+\xi\,\zeta)^{e_\omega}\,\zeta^{-\frac{1}{2}}$
(Definition 4), 
and let $q \equiv 1 \pmod {n}$. Consider an {\it arbitrary} ideal
${\mathfrak q}  \div  q$ in $L:=\Q(\mu_n)$ and any prime ideal
${\mathfrak Q} \div  {\mathfrak q}$ in $M:=LK$.
  
\smallskip 
We suppose 
given integers $u$, $v$ with g.c.d.\,$(u,v) = 1$, such that $q \notdiv u\,v$ and
$\frac{v}{u} \equiv \xi \pmod {\mathfrak q}$.

\smallskip 
$ \ \ \bullet\ $  If  $\kappa \not\equiv 0 \pmod {p}$
and $\Big(\Frac{\eta_1}{{\mathfrak Q}}\Big)_{\!\!M} = 1$, we then have:

\noindent
{\it If $u+v \not\equiv 0 \pmod {p}$, $(u,v)$ cannot be
a part of a solution $(x,y,z) = (u,v,z)$, $(v,u,z)$,
$(x,v,u)$, or $(x,u,v)$ of Fermat\,${}'$\!s equation.}

\smallskip 
$ \ \ \bullet\ $  If $\kappa \not\equiv 0 \pmod {p}$
and $\Big(\Frac{\eta_1}{{\mathfrak Q}}\Big)_{\!\!M} \ne 1$, we then 
have:

\noindent
{\it If $u + v \equiv 0 \pmod {p}$,
$(u,v)$ cannot be a part of a solution $(x,y,z) = (u,y,v)$
or $(v,y,u)$ of the second case of Fermat\,${}'$\!s equation.}

\smallskip 
$ \ \ \bullet\ $  If  $\kappa  \equiv 0 \pmod {p}$
and $\Big(\Frac{\eta_1}{{\mathfrak Q}}\Big)_{\!\!M} \ne 1$, we then 
have:

\noindent
{\it The pair $(u,v)$ cannot be a part of a solution $(x,y,z) $
of any case of Fermat\,${}'$\!s equation.}~\hfill\fin} 
 \end{rema}
 
\begin{prop} Let $(x, y, z)$ be a solution of Fermat\,${}'$\!s equation; 
then let $q \notdiv xyz$ be such that 
$\frac{y}{x}$, $\frac{y}{z}$, $\frac{x}{z}$ are of orders
$n$, $n'$, $n''$ (modulo $q$)
prime to $p$, and let $\wt L:= \Q(\mu_{d})$ where $q =: 1+ d\,p^r$,
$r \geq 0$, $p \notdiv d$.

\smallskip \noindent
Let  $\xi$, $\xi'$, $\xi''$, of orders $n$, $n'$, $n''$, 
be such that 
 $\xi \equiv \frac{y}{x} \pmod {{\mathfrak q}_{\xi}}$, 
 $\xi' \equiv \frac{y}{z} \pmod {{\mathfrak q}_{\xi'}}$, 
 $\xi'' \equiv \frac{x}{z} \pmod {{\mathfrak q}_{\xi''}}$
 in $L$, $L'$, $L''$, respectively.
 
\noindent
 Then there exist a prime ideal $\wt {\mathfrak q} \div q$ of $\wt L$
 and $t',\,t'' \in {\rm Gal}(\wt L/\Q)$ such that the following  
congruences hold:
 
 \smallskip
 (i) $\ \xi'{}^{t'} \equiv \frac{-\xi}{\xi+1} \pmod {\wt {\mathfrak q}}$,
 
\smallskip
 (ii) $\xi''{}^{t''} \equiv \frac{-1}{\xi+1} \pmod {\wt {\mathfrak q}}$.
\end{prop}

\begin{proof} Since $L$, $L'$, $L''$ are subfields of  $\wt L$,
    there exist prime ideals $\wt {\mathfrak q}_0$, 
    $\wt {\mathfrak q}'_0$, $\wt {\mathfrak q}''_0$ of $\wt L$
    dividing ${\mathfrak q}_{\xi}$,  ${\mathfrak q}_{\xi'}$, 
${\mathfrak q}_{\xi''}$, respectively, such that:
$$ \xi \equiv \hbox{$\frac{y}{x}$} \pmod {\wt {\mathfrak q}_{0}}, 
\ \ \xi' \equiv \hbox{$\frac{y}{z}$} \pmod {\wt {\mathfrak q}'_{0}}, 
\ \  \xi'' \equiv \hbox{$\frac{x}{z}$} \pmod {\wt {\mathfrak q}''_{0}}. $$
The ideals 
$\wt {\mathfrak q}'_{0}$ and $\wt {\mathfrak q}''_{0}$ are some conjugates
of $\wt {\mathfrak q}_{0}$ and there exist $t',\,t'' \in {\rm Gal}(\wt L/\Q)$
such that $\xi \equiv \frac{y}{x}$, $\xi'{}^{t'} \equiv \frac{y}{z}$,
$\xi''{}^{t''} \equiv \frac{x}{z} \pmod {\wt {\mathfrak q}_{0}}$.

\smallskip
From $x^p+y^p+z^p = 0$ we get 
$\big (\frac{x}{y}\big )^p + \big (\frac{z}{y}\big )^p = -1$
giving:
$$\xi^{-p}+(\xi'{}^{t'})^{-p} \equiv -1 \pmod {\wt {\mathfrak q}_{0}}. $$
Since $p\notdiv d$, we can use the inverse of the Frobenius 
automorphism 
$t_p$ of $p$ in $\wt L/\Q$, which gives easily the relation (i) (for 
$\wt {\mathfrak q}:=t_p^{-1}
(\wt {\mathfrak q}_{0})$).

\smallskip
From the obvious relation $\xi''{}^{t''} \xi'{}^{-t'} \xi \equiv 1
 \pmod {\wt {\mathfrak q}_0}$, which implies the equality $\xi''{}^{t''} 
\xi'{}^{-t'} \xi = 1$,
 we obtain the point (ii)  since $\xi \ne -1$.\,\footnote{\,The case
 $\xi = -1$ means $y+x = z_0^p \equiv 0 \pmod {q}$, i.e., $q\div z$, 
 which is excluded; in the same way, $\xi' \ne -1$, $\xi'' \ne -1$.}
\end{proof}

\begin{cor} Let $m := {\rm l.c.m.\,}(n', n'')$; then we have
    $\phi(m) > \Frac{{\rm log}(q)}{{\rm log}(3)}$.
\end{cor}

\begin{proof} We have
$\xi''{}^{t''} + \xi'{}^{t'} + 1 \equiv 0 \pmod {\wt {\mathfrak q}}$;
then $\xi''{}^{t''} + \xi'{}^{t'} + 1 \in \Q(\mu_m)$ by definition of $m$, and
$\No_{\Q(\mu_m)/\Q}(\xi''{}^{t''} + \xi'{}^{t'} + 1) = q \,N$,
$N\geq 1$. Since 
$\No_{\Q(\mu_m)/\Q}(\xi''{}^{t''} + \xi'{}^{t'} + 1)
 < 3^{\phi(m)}$, we get $N < \frac{1}{q}\,3^{\phi(m)}$, 
 giving the result.
\end{proof}
   
Same results for $m' := {\rm l.c.m.\,}(n, n'')$ and
$m'' := {\rm l.c.m.\,}(n, n')$.
  
\begin{cor} We can choose the representative pairs $(\xi,{\mathfrak 
q})$, $(\xi',{\mathfrak q}')$, $(\xi'',{\mathfrak q}'')$ such that
 $\ \xi' \equiv \frac{-\xi}{\xi+1}$ and
 $\xi'' \equiv \frac{-1}{\xi+1} \pmod {\wt {\mathfrak q}}$
 for a suitable $\wt {\mathfrak q} \div  q$ in $\wt L$.
 
\smallskip \noindent
 In such a way, we have $\xi'' =\xi^{-1}\, \xi'$.~\hfill\fin
\end{cor}

\subsection{Case of an odd character $\chi \ne \omega$}

We suppose that $\chi$ is an odd character of $g$ distinct from 
$\omega$; then $\chi = \omega^k$, $k$ odd, $k\not\equiv 1 \pmod {p-1}$,
which excludes the case $p=3$.

\smallskip
As for the case $k=1$, we can represent modulo $p$ the corresponding 
idempotent
by an element in $\Z[g]$ of the form $e_\chi = (1 - s_{-1})\, 
e'_\chi$, $e'_\chi \in \Z[g]$ (see Subsection 2.3).

\smallskip
We suppose that the $\chi$-component of the $p$-class group of $K$ is 
trivial;
for this, a necessary and sufficient condition is that the Bernoulli 
number
$B_{p-k}$ be prime to $p$ (see e.g. [Gr1, Section 2] for more details).

\smallskip
So, for any relation of the form $(u + v\,\zeta)\,\Z[\zeta] = 
{\mathfrak w}_1^p\ {\rm or}\ {\mathfrak p}\,{\mathfrak w}_1^p$
where  g.c.d.\,$(u,v)=1$, we get immediately:
$$(u + v\,\zeta)^{e_\chi}  = \delta_\chi^p, \ \delta_\chi\in K^{\times},$$
 since any $\chi$-unit of $K$ is a $p$th power for $\chi$ odd distinct 
from $\omega$ (moreover in the special case $(1-\zeta)^{e_\chi}$ is a 
$\chi$-unit).

 \smallskip
It is clear that  Lemma 4 is valid for the 
character $\chi$ and that the two equations are equivalent.

 \smallskip
The  relation
$(u + v\,\zeta)^{e_\chi}  = \delta_\chi^p$ may be considered as the 
$\chi$-SFLT equation associated to SFLT under the triviality of the 
$\chi$-class group.

\smallskip
As in the previous subsection, let $q \ne p$ be a prime number such 
that $q \notdiv u\,v$ and
$\frac{v}{u}$ is of order $n$ modulo $q$, with $n$ prime
to $p$ (see Lemma~2).
 
\smallskip
Then let $\xi$ of order $n$ and ${\mathfrak q} := {\mathfrak 
q}_\xi \div q$
in $L = \Q(\mu_n)$, characterized by the relation
$\xi \equiv \frac{v}{u} \pmod {\mathfrak q}$.
From $\eta = (1 + \xi\,\zeta) \,\zeta^{-\frac{1}{2}}$,
 put $\eta_k := \eta^{e_\chi} \in M$, where $M:=LK$; then
$\eta_k =  (1+\xi\,\zeta)^{e_\chi}$,
since $\zeta^{e_\chi} = 1$. Thus $\eta_k \in M^+$
and $\eta_k \in K^{\times p}$ if $n \leq 2$.

\smallskip
We deduce the congruence in $M$:
$$\eta_k \equiv \big (1 + \hbox{$\frac{v}{u}$}\,\zeta \big)^{e_\chi} = 
\delta_\chi^p
\pmod {\prd_{{\mathfrak Q} \div {\mathfrak q}}\,{\mathfrak Q}}. $$

We then have the relation
 $\,\Big(\Frac{\eta_k}{{\mathfrak Q}}\Big)_{\!\!M} = 1$,
 for  all  ${\mathfrak Q}  \div  {\mathfrak q}$,
so that, in this situation, a contradiction to the existence of a 
nontrivial solution of the SFLT equation is that this symbol be nontrivial
for some $q$.

\smallskip
Here the value of $\kappa$ does not enter.

\smallskip
From  Kummer duality, the extension $M(\sqrt[p]{\eta_k})/M$ is 
splitted, by means of a $p$-cyclic extension,
over the extension $L K_{\chi^*}$, where $\chi^* = \omega^{1-k}$
and $K_{\chi^*}$ is the subfield of $K$ fixed by the kernel of $\chi^*$;
this field $K_{\chi^*}$ is real. Of course, $L K_{\chi^*} = L$ if and 
only if $K_{\chi^*}=\Q$, i.e., $\chi=\omega$.

\medskip
This criterion may be used for any odd character $\chi \ne \omega$ such that 
the $\chi$-component of the $p$-class group of $K$ is trivial, which 
may have some interest. In some sense, it is similar to the case 
$\kappa \equiv 0 \pmod {p}$ of the preceding case $\chi=\omega$,
the symbols being trivial independently of $q$.

\smallskip
But unfortunately, the corresponding extensions
$M(\sqrt[p]{\eta_k})/L$ are meta\-belian (nonabelian) extensions and
do not define intrinsic arithmetic properties of the field $L$ as 
with the use of the single character $\omega$ to which we will return  to study 
its properties. 

\subsection{Computation of the $\F_p$-dimension of a group of units}

Since $\eta_1$ is considered in $(E_M/E_M^p)^{e_\omega}$,
it is necessary to precise the $\F_p$-dimension of this group. The 
computation is the same for any odd character $\chi$ (this may be useful
for  Subsection 3.2).

\begin{prop} Let  $M = LK$, where   $L=\Q(\mu_n)$, $n>2$, $p\notdiv n$.
   Let $E_M$ be the group of units of $M$ and let $\chi = \omega^k$
be   an odd character of $g$.
   
\smallskip\noindent
Then the $\F_p$-dimension of $(E_M/E_M^p\,.\,\mu_p)^{e_\chi}$ is equal to 
$\frac{1}{2}\,[L:\Q] = \frac{1}{2}\,\phi(n)$.
\end{prop}

\begin{proof}
 Put $\Gamma := {\rm Gal}(M/\Q) = G \oplus H$ where
$G := {\rm Gal}(M/L)$ and where $H := {\rm Gal}(M/K)$.
Let $\wh \Gamma = \wh G\oplus\wh H$ be the group
of  irreducible characters of
$\Gamma$; for any $\psi \in \wh \Gamma$, let $\varepsilon_\psi$ 
be the idempotent:
$$\varepsilon_\psi := \Frac{1}{\vert\,\Gamma\,\vert} \sm_{\sigma 
\in \Gamma} \psi^{-1}(\sigma)\,\sigma \in \C_p[\Gamma]. $$
If $\psi = \omega^i .\,\theta$, $\ \omega^i\in \wh G$,
$1\leq i\leq p-1$, $\theta \in \wh H$,
then $\varepsilon_\psi = \varepsilon_{\omega^i}\,\cdot\,
\varepsilon_\theta$.

\smallskip
From the Dirichlet--Herbrand theorem on units (see e.g. [Gr2, I.3.7]) we know
that  the representation $\C_p \oplus (\C_p \tensorZ E_M)$ is given by
the representation of permutation:
$$\C_p[\Gamma]\,\Frac{1}{2}(1+c) =
\plus_{\psi \ \rm even} \C_p[\Gamma]\,\varepsilon_\psi. $$

\smallskip
Then, since the character $\chi$ is odd,
 $(\C_p \oplus (\C_p \tensorZ E_M))^{\varepsilon_\chi} =
\big( \C_p\tensorZ E_M\big)^{\varepsilon_\chi}$ is the representation
 $\plus_{\psi\ \rm even}
\C_p[\Gamma]\,\varepsilon_\psi\,\cdot\,\varepsilon_\chi$.

\smallskip
 Put $\psi = \omega^i.\,\theta$; then $\varepsilon_\psi =
\varepsilon_{\omega^i}.\, \varepsilon_\theta$ and 
   $\varepsilon_\psi .\,\varepsilon_\chi = 0$ except
   if $i=k$. The sum is  over $\psi = \chi\,\theta$
   with $\theta$ odd since $\psi$ must be even. 
   Then:
   $$\big( \C_p\tensorZ E_M\big)^{\varepsilon_\chi} \simeq 
   \plus_{\theta\in \wh H,\, \rm odd} 
   \C_p[\Gamma]\,\varepsilon_{\chi.\,\theta}\,. $$
   
   We deduce that the $\C_p$-dimension of 
   $\big( \C_p\tensorZ E_M\big)^{\varepsilon_\chi}$
   is   $\frac{1}{2}[L:\Q]$. Hence the proposition follows
   since $\varepsilon_\chi \equiv e_\chi \pmod {p\,\Z_p[g]}$.
   \end{proof}
    
    In particular, the $\F_p$-dimension of 
$(E_M/E_M^p\,\cdot\,\mu_{p})^{e_\omega}$
    is equal to $\frac{1}{2}[L:\Q]$. Thus the subgroup of
    $(E_M/E_M^p\,\cdot\,\mu_{p})^{e_\omega}$ generated by the images 
    of the units $t\,\eta_1$, $t \in {\rm Gal}(M/K) / 
    \langle\,t_{-1}\,\rangle$, is of
     $\F_p$-dimension less or equal to $\frac{1}{2}[L:\Q] = \frac{1}{2}\,\phi(n)$.

\section {Study of the cyclotomic units $\eta_1$ and the extensions 
$F_\xi$}

In this Section we use some classical elements of Kummer theory and 
of decomposition of a Kummer extension over a subfield.

\subsection{The cyclotomic unit $\eta_1$}

We consider, independently of any relation of the form $(u + v
\,\zeta)\,\Z[\zeta] = {\mathfrak w}_1^p \ {\rm or}\ 
{\mathfrak p}\,{\mathfrak w}_1^p$, the  cyclotomic number:
 $$\eta := (1+\xi\,\zeta) \,\zeta^{-\frac{1}{2}} ,$$
 where $\xi$ is a primitive $n$th root of unity with $p \notdiv n$, and
 the real cyclotomic unit $\eta_1 := \eta^{e_\omega}$ defined in Definition 4.
We exclude the cases $n=1$ and $n=2$ seen above
for which $\eta_1 \in K^{\times p}$.

\smallskip
For $n>2$, $L:= \Q(\xi)$  is an imaginary cyclotomic field, hence we 
can consider the biquadratic extension $M/L^+ K^+$,
where $M := L\,K$; then $M^+$ is the subfield of $M$ of relative 
degree 2, distinct from
$LK^+$ and from $KL^+$.

\smallskip
 Let $f$ be the residue degree of $q$ in $K/\Q$.
We note that the residue degree of $q$ in $M^+/\Q$ is equal to $f$.
     
\smallskip
Since $\eta_1$ is a unit, the extension
$M(\sqrt[p]{\eta_1})/M$ is $p$-ramified (i.e., unramified outside $p$).
Put $\pi = \zeta-1$; $\pi$ is still an uniformizing parameter at $p$ 
in $M$ (indeed, ${\mathfrak p}$ is not ramified in $M/K$). 
We have:
$$\eta \equiv 1 + \xi + \Frac{1}{2}\,(\xi - 1)\,\pi\pmod {\pi^2}, $$
giving by the usual computation modulo $\pi^2$:
$$\eta_1 := \eta^{e_\omega} \equiv 1+
\Frac{1}{2}\,\Frac{\xi-1}{\xi+1} \,\pi \, \pmod {\pi^2}; $$
since $n>2$, $\Frac{\xi-1}{\xi+1}$ is a local unit at
$p$, showing that $\eta_1$ is not $p$-primary; thus
in particular, the extension $M(\sqrt[p]{\eta_1})/M$ 
is cyclic of degree $p$.

\smallskip
 Kummer theory shows that the conductor of $M(\sqrt[p]{\eta_1})/M$
is ${\mathfrak p}^p$ extended to $M$ (see [Gr2, II.1.6.3]).
In some sense, $M(\sqrt[p]{\eta_1})/M$ is maximally wildly $p$-ramified
and has the same conductor as $M(\sqrt[p]{\zeta})/M$.

\smallskip
Moreover, this extension does not depend on the choice of $\zeta$
since we have: 
$$\big ((1+\xi\,\zeta^k) \,(\zeta^k)^{-\frac{1}{2}}\big )^{e_\omega}
= \big ((1+\xi\,\zeta)\,\zeta^{-\frac{1}{2}}\big )^{s_k\,e_\omega}, $$
with $s_k\,e_\omega \equiv k\,e_\omega \pmod {p\,\Z[g]}$
for any  $k$ prime to $p$, giving the same radical.
  
\subsection{The abelian extension $F_\xi/L$}
By definition of the character $\omega$, whose reflect
is $\omega^* = \chi_0$ (the unit character),
the extension $M(\sqrt[p]{\eta_1})/M$ is splitted over
$L$ by means of a cyclic $p$-ramified extension $F_\xi$, of degree $p$
over $L=\Q(\mu_n)$ (i.e., $F_\xi M = M(\sqrt[p]{\eta_1})$).

This extension, as extension of $L$, only depends on $\xi$ of 
order $n$. The family $(F_{\xi'})_{\xi'\,{\rm of\,order}\ n}$ is canonical.

\smallskip
Since $\eta_1$ is real, $\eta_1 = (1+\xi^{-1} \,\zeta^{-1})^{e_\omega}\, 
\zeta^\frac{1}{2}$ which defines the same extension as 
$(1+\xi^{-1}\zeta)^{e_\omega}\,\zeta^{-\frac{1}{2}}$
as we have seen at the end of Subsection 4.1.
Then we get $F_\xi = F_{\xi^{-1}}$.
In the cases $n \leq 2$, we have $L=\Q$,
$\eta_1 \in K^{\times p}$, and $F_{\pm 1}  = \Q$.

\smallskip
It is easy to see that for any $t \in {\rm Gal}(L/\Q)$ we have the 
relation  $F_{\xi^t} = t F_{\xi}$, where by abuse of notation $t 
F_{\xi}$
means $t'F_{\xi}$ for any $\Q$-automorphism $t'$ of $F_{\xi}$ such 
that $t'{{\big \vert}_L} = t$. Indeed, we have in the same way, 
$t'(\sqrt[p]{\eta_1}) = \sqrt[p]{t\,(\eta_1)}$ (up to a $p$th root of 
unity) where $t\,(\eta_1) = (1+\xi^t\,\zeta)^{e_\omega}\,
\zeta^{-\frac{1}{2}}$.\,\footnote{\,We use the same notations 
for the elements of the Galois groups ${\rm Gal}(M/K)$ and ${\rm 
Gal}(L/\Q)$, then $G={\rm Gal}(M/L)$ and $g={\rm Gal}(K/\Q)$ and similarly for
${\rm Gal}(M(\sqrt[p]{\eta_1})/M)$ and ${\rm Gal}(F_\xi/L)$. }
  
\smallskip
Suppose now that we have chosen a prime number $q$ such that
$q \equiv 1$ $\pmod {n}$, $p\notdiv n$, and let ${\mathfrak q}$ 
be a fixed prime ideal above $q$ in $L$; later, we will have
${\mathfrak q}={\mathfrak q}_\xi$ when $\xi$ is associated to the 
usual integers $u$, $v$, but in this subsection ${\mathfrak q}$ is arbitrary.

Consider the symbol
$\Big(\Frac{\eta_1}{{\mathfrak Q}}\Big)_{\!\!M}$ (which is independent of the
choice of ${\mathfrak Q}  \div  {\mathfrak q}$ in $M$); this symbol
is equal to 1 if and only if the image of $\eta_1$ in the residue 
field $Z_M/{\mathfrak Q}$ is a $p$th power, thus if and only if 
${\mathfrak Q}$
splits in $M(\sqrt[p]{\eta_1})/M$ (Hensel${\,}'s$\,Lemma)
which is equivalent to the splitting
of ${\mathfrak q}$ in $F_\xi/L$.

\medskip
Let $H_L$ be the  maximal abelian $p$-ramified
$p$-extension of $L$;
it contains all the extensions $F_{\xi'}$, $\xi'$ of order $n$, the cyclotomic
$\Z_p$-extension $L_\infty = L \Q_\infty$ of $L$ which is abelian
over $\Q$, and $\frac{1}{2}[L:\Q]$ other independent $\Z_p$-extensions
of $L$. This extension $H_L$ will be studied in more details in 
Section 5.

\medskip
Since $q$ totally splits in $L/\Q$, the
decomposition field of $q$ in $L_\infty/\Q$ is  $L_e= 
L\,\Q_e$, where $\Q_e \subset \Q_\infty$ is
the stage of degree $p^e$ over $\Q$
where $q^f = 1+ p^{e+1} d$, $e \geq 0$,
$p\notdiv d$; note that $e=0$ is equivalent to $\kappa \not\equiv 0 \pmod {p}$.

\section {Study of the extensions  $H_L/L$ and $F_n/L$}

In this section we recall some class field theory results concerning 
the abelian $p$-ramification over $L$.

\subsection{Class field theory and $p$-ramification}

Let $H_L$ be the  maximal abelian  $p$-ramified $p$-extension of $L:= 
\Q(\mu_n)$ in the case $n>2$, $p\notdiv n$ (so that $L$ is an imaginary
cyclotomic field of even degree) and let
$H_L{\st [p]} \subseteq H_L$ be the  maximal $p$-elementary $p$-ramified
extension of~$L$. 

\smallskip
We consider its Galois group as a vector
space over $\F_p$. 

\smallskip
Its dimension is given by the following \v Safarevi\v c formula 
(see e.g. [Gr2, II.5.4.1, (ii)]):
$${\rm dim}_{\F_p}({\rm Gal}(H_L{\st [p]}/L)) =
{\rm dim}_{\F_p}(V_L/L^{\times p}) + \hbox{$\frac{1}{2}$}[L:\Q] + 1, $$
where $V_L$ is the group of pseudo-units of $L$ which are local $p$th powers
at each place dividing $p$ in $L$.

\begin{lemm} The conductor of $H_L{\st [p]}/L$ divides $(p^2)$
 as ideal of $L$.
 \end{lemm}
 
\begin{proof} From Hensel${}'$s Lemma, since $p>2$ is not ramified in $L/\Q$
($p \notdiv n$), the modulus $(p^2)$ is sufficient for any $\alpha \in L^\times$,
$\alpha \equiv 1 \pmod {p^2}$, to be locally a $p$th power at each place
dividing $p$ in $L$.
 \end{proof}

Thus $H_L{\st [p]}$ is contained in the ray class field $L{\st
(p^2)}$ and this yields:
$${\rm Gal}(H_L{\st [p]}/L) \simeq I/I^p \,R,$$
where $I$ is the group of fractional 
ideals of $L$ prime to $p$ and $R$  is the ray group
modulo $p^2$, i.e., $\big\{(\alpha) \in I$, $\alpha \equiv 1 \pmod {p^2} \big\}$.

\subsection{The subextension $F_n$}

 Let $t_{-1}$ be the element
of order 2 of the group ${\rm Gal}(M/L^+ K)$ and $s_{-1} \in G$ be the element
of order 2 of  ${\rm Gal}(M/K^+ L)$ (the complex conjugation is $c = s_{-1}\,t_{-1}$
as generator of ${\rm Gal}(M/M^+)$).

\smallskip
Since we have the relations $\eta_1^c = \eta_1$,
$\eta_1^{s_{-1}} =  \eta^{e_\omega \,\cdot\, s_{-1}}= \eta_1^{-1}$,
giving the relation $\eta_1^{t_{-1}} = \eta_1^{-1}$, we deduce that:
$${\rm Gal}(M(\sqrt[p]{\eta_1})/L^+ K) \simeq {\rm Gal}(F_\xi/L^+)
\simeq D_{2p}, $$
the diedral group of order $2p$.\,\footnote{\,Let $A:={\rm Gal}(M/L^+) =
G \oplus \langle \,t_{-1}\,\rangle$. Let $\chi_1$ be the character 
of $A$ defined by $\chi_1(s)=1$ for all $s \in G$ and 
$\chi_1(t_{-1})=-1$. Put $\chi = \omega\,\chi_1$; is is easy to see
that $\chi$ is the character of the radical
$\langle \,\eta_1\,\rangle \,M^{\times p}/M^{\times p}$
as $A$-module, since $\eta_1 =
\eta^{e_\omega}$ and  $\eta_1^{t_{-1}} = \eta_1^{-1}$.
From  Kummer duality, the character of 
${\rm Gal}(M(\sqrt[p]{\eta_1})/M)$ is $\chi^* := \omega\,\chi^{-1} = 
\chi_1$ proving that 
${\rm Gal}(M(\sqrt[p]{\eta_1})/L^+) \simeq G\times
{\rm Gal}(F_\xi/L^+)$, with ${\rm Gal}(F_\xi/L^+) \simeq D_{2p}$.
We also have ${\rm Gal}(M(\sqrt[p]{\eta_1})/M^+) \simeq D_{2p}$. }

\smallskip
In other words, ${\rm Gal}(L/L^+)$ acts on ${\rm Gal}(F_\xi/L)$ by
$\sigma^{t_{-1}} := t'_{-1}\,\cdot\,\sigma \,\cdot\,t'_{-1} = \sigma^{-1}$ for all $\sigma
\in {\rm Gal}(F_\xi/L)$ and any extension $t'_{-1}$ of $t_{-1}$ in
${\rm Gal}(F_\xi/L^+)$.

\smallskip
It will be necessary to consider the compositum of all the
extensions $M(\sqrt[p]{\eta_1})$ when $\xi$ (of order $n$) varies.
Indeed, in the situation of a nontrivial solution $(u,v)$ of the SFLT
equation, for any $n>2$, $p \notdiv n$, the root $\xi$ such that $\xi \equiv
\frac{v}{u} \pmod {\mathfrak q}$, for ${\mathfrak q}={\mathfrak q}_\xi$,
is uneffective and the properties of the symbols 
$\Big(\Frac{\eta_1}{{\mathfrak Q}}\Big)_{\!\!M}$,
${\mathfrak Q} \div {\mathfrak q}$ in $M$, for the pairs
$(\eta_1,{\mathfrak Q})$, can be studied in this extension.

\smallskip
Let $F_n$ be the compositum of the corresponding extensions
$F_{\xi'}$, $\xi'$ of order $n$, so that 
 $F_n$ is also the compositum of the  $F_{\xi^t}$, $t\in {\rm 
 Gal}(M/K)$, $\xi$ fixed;
since $t_{-1}\eta_1 = \eta_1^{-1}$, we can consider
the  $t\eta_1$  with
$t$ modulo $\langle\,t_{-1}\,\rangle$ (this is coherent with the 
relation $F_{\xi} = F_{\xi^{-1}}$).

\smallskip
We have the equality $F_n M = M \big(\sqrt[p]
{\langle\, t\, \eta_1\, \rangle_{{t \bmod < t_{-1}>}} } \,\big)$.

\smallskip
Then as above
${\rm Gal}(L/L^+)$ acts on ${\rm Gal}(F_n/L)$ by
$\sigma^{t_{-1}} = \sigma^{-1}$ for all $\sigma
\in {\rm Gal}(F_n/L)$.

\begin{lemm} The Galois closure of $F_\xi$ over $\Q$ is
    $F_n$ which  is linearly disjoint from $L_\infty/L$.
\end{lemm}

\begin{proof} Over the field $K$, the Galois closure of 
$M(\sqrt[p]{\eta_1})$ is given by the radical
$\langle\,t\,\eta_1 \,\rangle_{t \bmod <t_{-1}>}$ with
 $t\,\eta_1 = (1+\xi^t\,\zeta)^{e_\omega}\,\zeta^{-\frac{1}{2}}$, $t\in {\rm 
 Gal}(M/K)$, giving the first part of the lemma.
The relation $L_1 \subseteq F_n$ is equivalent to:
$$M(\sqrt[p]{\zeta\,})  \subseteq
M\big(\hbox{$\sqrt[p]{\langle\, t\, \eta_1\, \rangle_{t \bmod <t_{-1}>}}\,$} \big), $$
then to the existence of a relation of the form
$\prd_{t \bmod <t_{-1}>} \! (t\, \eta_1)^{\lambda_t} = \zeta\,\delta^p$,
$\lambda_t \in \Z$, $\delta \in M^\times$; but since the left member 
is real, the use of complex conjugation implies $\zeta^2 \in  M^{\times p}$,
which is absurd.
\end{proof}

\begin{rema} {\rm The $\F_p$-dimension of the above radical depends 
on the study of the relation $\prd_{t \bmod  <t_{-1}>} ( t\, \eta_1)^{\lambda_t} 
\in M^{\times p}$; this yields to (see Subsection 4.1): 
$$\prd_{t \bmod  <t_{-1}>} \Big(1+\Frac{1}{2}\Frac{\xi^t-1}{\xi^t+1}\,\pi 
\Big)^{\lambda_t e_\omega}\equiv 1+ \Big(\sm_{t \bmod  <t_{-1}>} \lambda_t 
\,\Frac{1}{2}\Frac{\xi^t-1}{\xi^t+1}\Big)\,\pi \! \pmod {\pi^2}. $$
Thus if the  numbers 
$\Frac{\xi^t-1}{\xi^t+1}$, $t \bmod \langle\,t_{-1}\,\rangle$, are 
linearly independent
modulo $p$, we get the dimension $\frac{1}{2}[L: \Q]$ and 
${\rm dim}_{\F_p}({\rm Gal}(F_n/L)) = \frac{1}{2} [L:\Q]= 
\frac{1}{2}\phi(n)$.

\medskip
Since $\eta_1$ is a cyclotomic unit of $M$, the classical study 
of the whole group of cyclotomic units of $M$ (of finite index in 
$E_M$) may give the exact $\F_p$-dimension of the radical
(see Washington  book, Chap.\,8); but this study depends, in a complicate 
manner, on the Galois group of $M/\Q$ and the law of decomposition 
of the prime divisors of $n$ in this extension.~\hfill\fin }
\end{rema}

\subsection{Canonical decomposition of ${\rm Gal}(H_L{\st [p]}/L)$}

Consider the Galois group $C_L :={\rm Gal}(H_L{\st [p]}/L)$
as a module over $\F_p[{\rm Gal}(L/L^+)]$. Write:
$$C_L = C_L^+ \oplus C_L^-, \ \,{\rm with}\ \, C_L^+ :=
C_L^{\frac{1}{2}(1+t_{-1})} , \ \,C_L^- := C_L^{\frac{1}{2}(1-t_{-1})}. $$

We denote by $H_L^-{\st [p]}$ the subfield of $H_L{\st [p]}$ fixed
by $ C_L^+$ and by  $H_L^+{\st [p]}$ the subfield of $H_L{\st [p]}$
fixed by $ C_L^-$. We then have 
$F_n \subseteq H_L^-{\st [p]}$ and the diagram: 
\unitlength=0.6cm
$$\vbox{\hbox{\begin{picture}(11.1,3.8)
\put(4.6,3.0){\line(1,0){2.}}
\put(4.6,0.0){\line(1,0){2.}}
\put(4.0,1.9){\line(0,1){0.6}}
\put(4.0,0.6){\line(0,1){2.}}
\put(7.5,1.9){\line(0,1){0.6}}
\put(7.5,0.6){\line(0,1){2.}}
\put(6.9,2.9){$H_L{\st [p]}$}%
\put(2.5,2.9){$H_L^+{\st [p]}$}%
\put(3.8,-0.1){$L$}%
\put(6.94,-0.1){$H_L^-{\st [p]}$}%
\put(5.,3.6){$C_L^-$}%
\put(8.6,1.4){$C_L^+$}%
\end{picture}   }} $$

\begin{lemm} Put $\ov V_{\!L} := V_L/L^{\times p}$
 (see Subsection 5.1)   and
  $\ov V_{\!L} =\ov V_{\!L}^+ \oplus  \ov V_{\!L}^-$ as above.
 Then $\ov V_{\!L}^+ \simeq V_{\!L^+}/(L^+)^{\times p}$ giving:
 $${\rm dim}_{\F_p}(C_L^+ ) =
{\rm dim}_{\F_p}(\ov V_{\!L}^+) + 1;\ \
{\rm dim}_{\F_p}( C_L^-) = {\rm dim}_{\F_p}( \ov V_{\!L}^-)
    + \hbox{$\frac{1}{2}$} [L:\Q]. $$
\end{lemm}

\begin{proof} Since $p\ne 2$, we have $C_L^+ \simeq
{\rm Gal}(H_{L^+}{\st [p]}/L^+)$ for which the \v Safarevi\v c formula is 
${\rm dim}_{\F_p}( C_L^+) = {\rm dim}_{\F_p}( \ov V_{\!L}^+) + 1$,
proving the lemma.
\end{proof}

By this way,  the case where  ${\rm dim}_{\F_p}({\rm Gal}(F_n/L))=
\frac{1}{2} [L:\Q]$ is compatible  with the $\F_p$-dimension
of $C_L^-$ since when the invariant $ C_L^-$ is minimal
(which is equivalent to ${\rm dim}_{\F_p}(\ov V_L^-)=0$) then 
$F_n = H_L^-{\st [p]}$ as soon as the $t\eta_1$, $t \,\bmod <t_{-1}>$,
are independent in $M^\times/M^{\times p}$.

\smallskip
Note that the group of pseudo-units
$Y_L := \big\{ \alpha \in L^\times,\, (\alpha) = 
{\mathfrak a}^p \big\}$ is eluci\-dated by the following obvious
exact sequence:
$$1 \too E_L/E_L^p \tooo \ov Y_{\!L} \tooo {}_p\Cl_L \too 1, $$
where $\Cl_L$ is the $p$-class group of $L$, Ê${}_p\Cl_L$Ê
the subgroup of $\Cl_L$ of classes killed by $p$,
$E_L$ the group of units of $L$,
and $\ov Y_{\!L} := Y_L/L^{\times p}$.

\smallskip
For $L^+$ we get the analogous exact sequence:
$$1 \too E_{L^+}/E_{L^+}^p \tooo \ov Y_{\!L^+} \tooo {}_p\Cl_{L^+} 
\too 1. $$

We have, with usual notations $\pm$, the relations $(E_L/E_L^p)^+ \simeq 
E_{L^+}/E_{L^+}^p$ and $(E_L/E_L^p)^- = 1$, so that
$\ov Y_{\!L}^- \simeq {}_p\Cl_L^-$ 
and $\ov V_{\!L}^- \subseteq \ov Y_{\!L}^-$ only depends on the
minus part of the $p$-class group of $L$ and is often trivial. 

\smallskip
The group $\ov V_{\!L}^+ \simeq \ov V_{\!L^+} \subseteq \ov Y_{\!L^+}$ depends on 
the $p$-class group of $L^+$ (in general trivial) and more essentially
on the units locally $p$th power at $p$ in the
group of units $E_{L^+}$ of $L^+$ which is of $\Z$-rank $\frac{1}{2} 
[L:\Q] - 1$; but $\varepsilon \in E_{L^+}$ is a local $p$th power at 
each place dividing $p$ if and only if $\varepsilon^{p^\delta-1} 
\equiv 1 \pmod {p^2}$, where $\delta \div  \frac{1}{2} \phi(n)$ is the 
residue degree of $p$ in $L^+$, which is also very rare, giving often a 
trivial $\ov V_{\!L}^+$.

\begin{rema} {\rm Suppose that the group $\ov V_{\!L}$ is 
trivial.\,\footnote{\,This situation is by definition equivalent to 
the $p$-rationality of the field $L$ (see [Gr2, IV.3.5] for some 
equivalent conditions).}%
Then ${\rm dim}_{\F_p}( C_L^+) = 1$ and  ${\rm dim}_{\F_p}( C_L^-)
= \hbox{$\frac{1}{2}$} [L:\Q]$. In this case $H_L$ is the compositum 
of the
$\Z_p$-extensions of $L$ which is of the form $H_L^+ H_L^-$
where $H_L^+ = L_\infty$ is the cyclotomic $\Z_p$-extension of $L$
and $ H_L^-$ the compositum of $\frac{1}{2} [L:\Q]$
independent relative $\Z_p$-extensions of $L$ (i.e., which are 
pro-diedral   over  $L^+$).

\smallskip
Then $H_L{\st [p]}$ is the compositum of the first stages of these
$\Z_p$-extensions, the extension $H_L^+{\st [p]}$ is $L_1$,
and $H^-_L{\st [p]}\,M$ may be the Kummer extension
 defined by the radical generated by the $t \eta_1$
 as soon as its $\F_p$-dimension is $\frac{1}{2} [L:\Q]$.
 See Subsection 3.3 about these questions of dimensions.~\hfill\fin}
\end{rema}

\subsection {Conclusion} We have established, from Corollary 4 and 
Remark 4 (Subsection 3.1), that, under a solution of
Fermat${}'$s equation ($p>3$), for infinitely  many particular 
prime numbers $q$ in the case $\kappa \not \equiv 0 \pmod {p}$,
there exist privilegiate pairs $(F_\xi,{\mathfrak q}_\xi)$, 
$(F_{\xi'},{\mathfrak q}_{\xi'})$ for the first case  (resp. 
$(F_\xi,{\mathfrak q}_\xi)$, 
$(F_{\xi'},{\mathfrak q}_{\xi'})$, $(F_{\xi''},{\mathfrak q}_{\xi''})$
for the second case), defined up to conjugation, with
$p$-cyclic $p$-ramified extensions $F_\xi/L$, $F_{\xi'}/L'$,  
$F_{\xi''}/L''$ and
prime ideals ${\mathfrak q}_{\xi}$, ${\mathfrak q}_{\xi'}$,
 ${\mathfrak q}_{\xi''}$, for which
${\mathfrak q}_\xi$, ${\mathfrak q}_{\xi'}$ are inert for the first 
case
(resp. ${\mathfrak q}_\xi$, ${\mathfrak q}_{\xi'}$   are inert and
${\mathfrak q}_{\xi''}$ splits, for the second case)
in the corresponding extensions $F_\xi/L$, $F_{\xi'}/L'$, 
$F_{\xi''}/L''$.

\smallskip
In the case $\kappa \equiv 0 \pmod {p}$,  for all the above pairs,
 the ideals split in the corresponding extensions.

\smallskip
This intricacy may be in contradiction, for mostly  primes 
$q$,  since the arithmetical properties of
the governing fields $\Q(\mu_{q-1})$ are independent
of the Fermat problem; more precisely, a general philosophy is that
the decomposition groups of prime ideals in Galois extensions do not 
fulfill any other laws than standard ones, and may be analyzed in a
statistical point of view (see Section 7 for a direct study of these aspects).

\smallskip
About this, we will explain in  Section 9 that the case $p=3$ is precisely
an exceptional counterexample to the above claim, since some constraints 
do exist; but we will show that these constraints are not 
in contradiction with statistical considerations
because of the structure of the {\it infinite} set of solutions.

\smallskip
One may object that $F_\xi$ comes from the radical:
$$\big\langle\,(1+\xi\,\zeta)^{ 
e_\omega}\,\zeta^{-\frac{1}{2}}\,\big\rangle\,M^{\times p}$$ over $M$,
which is directly   associated to a problem
of SFLT type, and in a standard algebraic point of view the
above circumstances on the laws of decomposition may be
{\it equivalent} to a contradiction to SFLT.
 Thus it will be necessary to obtain some analytic or geometrical informations
on the splitting of $q$ in $H_L{\st [p]}/L$ (especially in the 
canonical family $(F_{\xi'}/L)_{\xi'\, {\rm of\,order}\ n}$) so as to prove that 
the above particularities do not exist.

\section {A sufficient condition proving Fermat${}'$s last theorem} 

In this section we study a sufficient condition for FLT, which
only involves  congruential properties of  prime ideals over
 $q$ in $\Q(\mu_{q-1})$.

\subsection{Main result}

We suppose that
$p>3$ and that the primes $q$ consi\-dered are such that $f>1$ and
$\kappa := \frac{q^f -1}{p} \not \equiv 0 \pmod {p}$;
we will then use  Remark 3 using Furtw\"angler${}'$s theorems.
Thus any divisor $n$ of $q-1$ is prime to $p$.

\smallskip
From a nontrivial solution $(u,v)$ of the SFLT equation, for which 
$\frac{v}{u}$ is of order $n>2$ modulo $q \notdiv u\,v$, we consider the
pair $(\xi, {\mathfrak q}_\xi)$, defined up to $\Q$-conjugation
 in $L:=\Q(\mu_n)$ (see Definition~3).

\smallskip
The integer $n$  and the pair are uneffective since if we fix an ideal
${\mathfrak q} \div q$ in $L$, the root $\xi$ such that 
${\mathfrak q}={\mathfrak q}_\xi$ is  unknown, or if we fix a 
primitive $n$th
root $\xi$, then the ideal ${\mathfrak q} \div q$ such that
${\mathfrak q}={\mathfrak q}_\xi$ is  unknown.

\smallskip
Let ${\mathfrak Q}_\xi$ be any prime ideal of $M:= LK$
above ${\mathfrak q}_\xi$.
Then the pair $(\eta_1,{\mathfrak Q}_\xi)$, where
$\eta_1 = (1+\xi\,\zeta)^{e_\omega}\,\zeta^{-\frac{1}{2}} \in M^+$, is also 
unknown in the same manner.

\smallskip
So if we ensure that, for instance for ${\mathfrak q}_0$ fixed 
arbitrarily in $L$,
for  ${\mathfrak Q}_0 \div {\mathfrak q}_0$ in $M$,
$\Big(\Frac{\eta_1}{t {\mathfrak Q}_0}\Big)_{\!\!M} = 1$
for all $t  \in {\rm Gal}(M/K)/\langle\,t_{-1}\,\rangle$, then in 
particular for the ``good''  value of the pair $(\eta_1,{\mathfrak Q})$ (i.e., 
such that ${\mathfrak Q} \div {\mathfrak q}_\xi$), we get:
$$\Big(\frac{\eta_1}{{\mathfrak Q}}\Big)_{\!\!M} =
\zeta^{\frac{1}{2}\,\frac{v-u}{v+u}\,\kappa}=1, $$
 giving $v-u \equiv 0 \pmod {p}$ which is absurd in the case of a 
 solution $(x,y,z)$ of Fermat${}'$s equation
 by choice of the difference $v - u = \pm(y-x)$ or $\pm(y-z)$ (see Corollary 4, (i)).

\smallskip
The problem is to know if there exist such prime numbers $q$
with $F_n$ in the splitting field of ${\mathfrak q}$ in $H_L{\st [p]}/L$,
for all $n \div q-1$, $n>2$.
If so, this will prove FLT (in the first case we know that $q\notdiv 
xyz$ and a single $q$ is sufficient;
for the second case where we know that $q\notdiv x z$, it is necessary 
to have infinitely many such primes $q$ to be certain that $q \notdiv y$).

\smallskip
For the proof of SFLT, we must suppose $u-v \not\equiv 0 \pmod {p}$
in the first case.

\smallskip
If $F_n$ is in the splitting field of ${\mathfrak q}$,
then this does not depend on the choice of ${\mathfrak q} \div q$
in $L$, which is a convenient simplification. In other words, $q$ 
totally splits in $F_n/\Q$.

\smallskip
Since $F_n \subseteq H_L^-{\st [p]}$, a
sufficient condition to have the total splitting of ${\mathfrak q}$ in
$F_n$ is that the Frobenius $\varphi$ of ${\mathfrak q}$ in
$H_L{\st [p]}/L$ be an element of $C_L^+$, which is equivalent to
$\varphi^{t_{-1}} = \varphi$, hence to  $\varphi^{t_{-1}-1} = 1$.
Note  that $\varphi$ is of order $p$ since its restriction to 
$L_1$ is of order $p$ by assumption.

\smallskip
 The image of 
$\varphi\in C_L$  by   the isomorphism ${\rm Gal}(H_L{\st [p]}/L) 
\simeq I/I^p \,R$ of class field theory, is given by the class of ${\mathfrak q}$ in
$I/I^p R$; thus the condition $\varphi^{t_{-1}-1} = 1$ is equivalent to
${\mathfrak q}^{t_{-1}-1} \in I^p R$, i.e.,
$${\mathfrak q}^{t_{-1}-1} = {\mathfrak a}^p\, (\alpha),\ \ \alpha 
\equiv 1 \pmod {p^2}, $$
for an ideal ${\mathfrak a}$ of $L$.
We must realize this for any divisor $n>2$ of $q-1$.

\smallskip
For $\wt n := q-1$, $\wt L := \Q(\mu_{q-1})$, we suppose that
the above condition
$\wt {\mathfrak q}^{\,\wt t_{-1}-1} = \wt {\mathfrak a}^p\, (\wt 
\alpha)$,
$\wt \alpha \equiv 1 \pmod {p^2}$,
is satisfied (for $\wt {\mathfrak q} \div q$ in $\wt L/\Q$).

\smallskip
Then let $n \div q-1$, $n>2$; since $L=\Q(\mu_n)$ is imaginary,
$L^+$ is fixed by the restriction $t_{-1}$ of $\wt t_{-1}$ to $L$,
and taking the norm $\No_{\wt L/L}$ we get:
$$\No_{\wt L/L}(\wt{\mathfrak q}^{\, \wt t_{-1}-1 }) =
\No_{\wt L/L}( \wt {\mathfrak a})^p\,\No_{\wt L/L}(\wt \alpha). $$

Since $q$ is totally split in $\wt L$, we have by definition
$\No_{\wt L/L}(\wt{\mathfrak q}) = {\mathfrak q}$
for some ${\mathfrak q} \div q$ in $L$,
and the above relation is of the form
${\mathfrak q}^{t_{-1}-1} = {\mathfrak a}^p\, (\alpha)$, with $\alpha 
\equiv 1 \pmod {p^2}$,
as expected (this coherent choice of the ideals ${\mathfrak q}$
is possible since the required condition of splitting at each stage 
is independent of the choice of the ideal).
So the whole condition for our purpose is given by the single
condition for $n=q-1$, $L=\Q(\mu_{q-1})$.

\medskip
We have obtained the following criterion,  where $c$
is the complex conjugation:

\begin{theo} Let $p$ be a prime number, $p>3$. If there exists at least 
a prime number $q$, $q \not \equiv 1 \pmod {p}$, $q^{p-1} \not \equiv 
1 \pmod {p^2}$, such that for a prime ideal ${\mathfrak q} \div  q$
in $\Q(\mu_{q-1})$, we have ${\mathfrak q}^{1-c} = {\mathfrak a}^p\,
(\alpha)$ for an ideal ${\mathfrak a}$ and an element $\alpha$
of $\Q(\mu_{q-1})$ with $\alpha \equiv 1 \pmod {p^2}$,%
\,\footnote{\, Since the multiplicative groups of the residue fields of $L$
at $p$ are of order prime to $p$, in any writing ${\mathfrak a}^p\,
(\alpha)$ we can suppose $\alpha = 1 + p\,\beta$, $\beta$ $p$-integer
of $L$.
The condition
 ${\mathfrak q}^{1-c} = {\mathfrak a}^p\,(\alpha)$, $\alpha \equiv 1
 \pmod {p^2}$, is equivalent to ${\mathfrak q}^{1-c}
 = {\mathfrak a}^p\,(1 + p\,\beta)$, where $\beta \equiv \beta^+ \pmod {p}$,
 for a  $p$-integer  $\beta^+$ of $L^+$;
 indeed, this last condition implies  ${\mathfrak q}^{2(1-c)}
 = {\mathfrak a}^{(1-c)p}\,(1 + p\,\beta)^{1-c}$ where
 $(1 + p\,\beta)^{1-c} \equiv 1 + p\,(1-c)\,\beta \equiv 1 \pmod {p^2}$,
 which gives the result thanks to a B\'ezout relation between 2 and $p$.
 
\smallskip
 The condition
 ${\mathfrak q}^{1-c} = {\mathfrak a}^p\,(\alpha)$, $\alpha \equiv 1
 \pmod {p^2}$, is also equivalent to ${\mathfrak q}
 = {\mathfrak b}^{1+c}{\mathfrak a'}^p\,(\alpha')$,  $\alpha' \equiv 1
 \pmod {p^2}$; indeed, a direction being trivial, from 
 ${\mathfrak q}^{1-c} = {\mathfrak a}^p\,(\alpha)$ we get 
 ${\mathfrak q}^2 = {\mathfrak q}^{1+c} {\mathfrak a}^p\,(\alpha)$.
  
\smallskip
 The condition ${\mathfrak q}^{1-c} = {\mathfrak a}^p\,(\alpha)$,
 $\alpha = 1 + p\,\beta$,
 is satisfied as soon as the class of ${\mathfrak q}^{1-c}$
is of order prime to $p$, which is a weak condition; it remains to get 
the stronger condition  $\beta \equiv \beta^+ \pmod {p}$.}
then the first case of FLT (or the first case of SFLT under the 
 supplementary condition $u-v \not\equiv 0\pmod {p}$) holds for~$p$.

\smallskip \noindent
The second case of FLT (or of SFLT) holds as soon as there exist infinitely many 
such primes $q$.~\hfill\fin
\end{theo}

From the \v Cebotarev${}'$s theorem, there exist infinitely many prime 
ideals ${\mathfrak l}$
of $\Q(\mu_{q-1})$ such that their Frobenii $\varphi_{{\mathfrak l}}$
lie in $C_{\Q(\mu_{q-1})}^+$
(which is at least of dimension 1); the problem is to be sure that 
there is no obstruction to the fact that it is
sometimes possible for ${\mathfrak l} = {\mathfrak q} \div q$. 

\smallskip
It is clear that such a set of prime numbers $q$
would be of Dirichlet density~$0$,
as for the set of prime numbers $q$, such that
the ring $\Z[\mu_{q-1}]$ contains a principal ideal of norm $q$, 
a result proved by Lenstra in [Len, Cor.\,7.6].

\smallskip
Theorem 2 may be of empty use due to an  excessive condition on 
the primes $q$. So we intend, in the forthcoming subsection,
to try to give a weaker form of this result (see  Conjecture 2).

 \smallskip
  Proposition 2 shows that the extension $F_{q-1} \subseteq  H_L^-{\st [p]}$,
  for $L = \Q(\mu_{q-1})$,
is of degree less or equal to $\frac{1}{2}[L:\Q] = \frac{1}{2} 
\phi(q-1)$. So, if the torsion group $\ov V_{\!L}$ is trivial,
the equality $F_{q-1} = H_L^-{\st [p]}$ is possible and the
sufficient condition of  Theorem 2 is also necessary; thus if 
there is any hope of success of the method, this condition cannot be improved in 
practice.

\subsection{Some related viewpoints}

We will examine if some effective (or numerical) aspects allow us to 
justify the method of proof of FLT based on Theorem 2 for $p>3$.
    
\smallskip
{\bf a)} In this first approach, we fix $q$ and 
${\mathfrak q}  \div q$ in $L = \Q(\mu_{q-1})$, and we try to 
find some suitable values of $p$ for which $\varphi_{\mathfrak q}
\in C_{\Q(\mu_{q-1})}^+$.
    
 Suppose that  ${\mathfrak q}^k = (\alpha)$ in $L = \Q(\mu_{q-1})$ for 
some $k>0$ and suppose that we find  $d>0$ such that: 
    $\alpha^d \equiv \alpha^+ \pmod {p^2}$,
for some prime $p$ such that $p \notdiv k\,d$, and   some $\alpha^+ 
\in L^+$; then $\alpha^{d(1-c)} \equiv 1$
$\pmod {p^2}$ giving a solution of the problem for the prime $p$
(then a posteriori $k$ may be chosen as the order of the ideal class
of ${\mathfrak q}$ and $d$  as a suitable divisor of the 
order of the multiplicative group of the residue field of $L$ at~$p$).
    
\smallskip
    Of course this relation looks like the general problem of the 
    Fermat quotients of algebraic numbers as studied by Hatada in [Hat].
Considering the work of Hatada and others, a serious conjecture would 
be that there exist 
infinitely many solutions $p$ for $q$ fixed.
   
\smallskip
    Since the numerical values of  $p$ are out of range of any 
    computer, this conjectural property is not of a practical use,
    but connect FLT to deep properties of algebraic numbers.

\medskip
Meanwhile, we have found the following example which gives a very 
partial illustration but shows that there is, a priori, no
systematic obstruction for this question.

\begin{ex}{\rm  Let $q=5$ and $p=463$. We then have 
$L=\Q(\mu_4)=\Q(i)$, where $i:=\sqrt{-1}$, and
${\mathfrak q} = (2+i)$. We see that $q$ is totally inert in $K$ 
(i.e., $f=462$) and that $p$ is also inert in $L$.
  
\smallskip
We obtain the following numerical informations:
    
  \smallskip \smallskip
$\ \bullet \ \ $ $(5^{463-1}-1)/463 \not\equiv 0 \pmod {463}$
(i.e., $\kappa \not\equiv 0 \pmod {p}$),
    
  \smallskip \smallskip
$\ \bullet \ \ $ $(2+i)^{463+1} \equiv 43990 \pmod {463^2}$.
    
  \smallskip
This implies immediately:
$${\mathfrak q}^{1-c} =\big( \Frac{2+i}{2-i}\big) \ \ 
{\rm and}\ \  {\mathfrak q}^{(p+1)(1-c)} = 
\big(\Frac{2+i}{2-i}\big)^{p+1}
\equiv 1 \pmod {p^2}, $$
giving the relation
${\mathfrak q}^{1-c} = {\mathfrak a}^p\,(\alpha)$ 
with ${\mathfrak a} = {\mathfrak q}^{c-1}$ and 
$\alpha \equiv 1 \pmod {p^2}$.~\hfill\fin}
\end{ex}

{\bf b)}  In a slightly different point of view,
we must consider that in general,
for a solution $(u,v)$ of the SFLT equation,
the order $n$ of $\frac{v}{u}$
modulo $q$ may be a strict divisor of $q-1$, even if
it is clear directly that $n$ tends to infinity with $q$
(Corollary 5).

\smallskip
Thus we have the following comments.

\smallskip
Let $m$ be a fixed integer, $m>2$, $p\notdiv m$.
Put $K':= \Q(\mu_{p^2}) \supset K$, $L:=\Q(\mu_m)$,
$H:=  H_L^-{\st [p]}$ (see Section 5), and $H':=HK'$.
Then $H/\Q$ and $K'/\Q$ are linearly disjoint.

\medskip
Let $\varphi\in {\rm Gal}(H'/H)$ of order $pf$, $f \div p-1$.
From the \v Cebotarev${}'$s theorem, there exist infinitely many
prime numbers $q$ such that, for a suitable ${\mathfrak Q}' \div q$
in $H'$, the Frobenius automorphism satisfies the equality
$\Big(\frac{H'/\Q}{{\mathfrak Q}'}\Big) = \varphi$.

\smallskip
This implies the following properties:

\smallskip
 $\ \bullet$  $\ \ q \equiv 1 \pmod {m}$ (since $q$ splits in $L/\Q$),

\smallskip
 $\ \bullet$  $\ \, q^f \not\equiv 1 \pmod {p^2}$
 (since $q$ is  inert in $K'/K$),

\smallskip
 $\ \bullet$  $\ \ q$ is totally split in $H/L$ (since $\varphi$ 
fixes $H$).

\medskip
The condition ${\mathfrak q}^{1-c} = {\mathfrak a}^p (\alpha)$,
$\alpha \equiv 1 \pmod {p^2}$, is satisfied for any prime ideal
${\mathfrak q} \div q$ in $L=\Q(\mu_m)$ but not necessarily 
for $\wt {\mathfrak q}$ in 
$\wt L = \Q(\mu_{q-1})$ (i.e., the Frobenius of ${\mathfrak q}$ in 
$H_L{\st [p]}$ fixes $H_L^-{\st [p]}$, but this is not necessarily 
true for the Frobenius of $\wt {\mathfrak q}$ in $H_{\wt L}{\st [p]}$
giving possible inertia in $H_{\wt L}^-{\st [p]}/\wt L\,H_L^-{\st [p]}$).

\smallskip
The order of $\frac{v}{u}$ modulo $q$ is $n \div q-1$
and not necessarily $m$, and the obvious analogue of Theorem 2 applies only if
$n \div m$. In other words, we try to replace the order
$q-1$ \big (probably too big under the condition that the Frobenius 
of ${\mathfrak q}$
lies in $C_{\Q(\mu_{q-1})}^+$\big) by a strict divisor $m$ (depending 
on $q$), for infinitely many
$q$ for which we hope that the Frobenius of ${\mathfrak q}$ lies in 
$C_{\Q(\mu_{m})}^+$.

\smallskip
Then, under a nontrivial solution $(u,v)$ of the SFLT equation
($u-v \not\equiv 0 \pmod {p}$), there is an obstruction
to the fact that there exists at least a pair $(m,q)$
($m$ and $q$ defined as above with a Frobenius
in $C_{\Q(\mu_{m})}^+$) such that a divisor $n$
of $m$ is the order of $\frac{v}{u}$ modulo $q$. 

\smallskip
This remark may constitute a way of access to a proof of FLT
by means of analytic investigations and we can propose the following 
independent conjecture.

\begin{conj}  Let $p$ be a prime number, $p>3$, and let 
$\rho = \frac{v}{u}$, with g.c.d.\,$(u,v) = 1$, be a 
rational distinct from $0$ and $\pm 1$.
 
\smallskip \noindent
 There exists a divisor function $m \,:\, \N \Sauf\{0\} \too \N \Sauf\{0\}$
 (i.e., such that $m{\st (e)} \div  e$ for all $e \in  \N \Sauf\{0\}$)
 such that  there exist infinitely many prime numbers $q$, with
 $\kappa \not\equiv 0 \pmod {p}$, totally split in 
 $F_{m{\st (q-1)}}$ (see Subsection 5.2), for which the order $n$
 of $\rho$ modulo $q$ divides  $m{\st (q-1)}$.~\hfill\fin
\end{conj}

Since $n$ tends to infinity with $q$, this means that $m{\st (q-1)}$ is 
unbounded with $q$. The existence of infinitely many primes $q$ 
satisfying the conditions of  Theorem 2 is equivalent to the 
conjecture with $m{\st (e)} = e$ for all $e$.

\medskip
The existence of such a function 
depends on two phenomena:

\medskip
(i) The order of magnitude of the primes $q$ discussed above from the 
\v Cebotarev${}'$s theorem.

\smallskip
(ii) The minimal value of the order modulo $q$ of a given rational 
$\rho$.

\begin{ex}{\rm  For $p=5$, $m=4$, we have $L=\Q(i)$,
and an obvious family of ideals 
${\mathfrak q}$ of $L$ such that ${\mathfrak q}^{1-c} = (\alpha)$,
$\alpha \equiv 1 \pmod {25}$, is given by the following expression:
$${\mathfrak q} = (e + 5a + 25 b\,i)\,\Z[i],\ \, e \in \{1,2,3,4\},\ 
\, a,\, b\in \Z, $$
$e$, $a$, $b$ being such that $(e + 5a)^2+(25 b)^2$ is a prime number 
$q$.

\smallskip
The prime numbers $q<10000$,
 $q \not\equiv 1$ $\pmod {5}$ and  $q^{4} \not\equiv 1 \pmod {25}$,
 of the above form, are the following:
$769$,  $1109$,  $1409$, $2069$, $2389$, $2789$, $3229$, $3329$, 
$3989$, $5309$, $5689$, $6469$, $6709$, $7069$, $7829$, $8329$, $8369$, 
$8429$.~\hfill\fin }
\end{ex}
 
It is clear that such a construction does exist for any $p$ and any 
$m>2$, and the question 
is the following: $p$, $u$, and $v$ being given, is it possible to 
find in such infinite
lists of prime numbers (corresponding to arbitrary values of $m$),
a prime  $q$ for which the order of $\frac{v}{u}$
modulo $q$ is a divisor of $m$
(which is equivalent to $q \div  u^m-v^m$)?
Note that for each $m$, only a finite number of $q$
in the list can be solution.

\smallskip 
 The existence of one solution $(m,q)$
gives the proof of  the first case of FLT for $p$ and
 the existence of infinitely many solutions $(m,q)$
gives a complete proof of FLT for $p$. 

\subsection {Explicit formula for the $p$th power residue symbol
\footnotesize{$\Big(\ds \frac{\eta_1}{{\mathfrak Q}}\Big)_{\!\!M}$} }
We suppose that $q \ne p$ is a given prime, and that $n \div  q-1$ is such 
that $p\notdiv n$. Let $\xi$ of order $n$ and let ${\mathfrak q}$
be a prime ideal of $L = \Q(\mu_n)$ dividing $q$.

\smallskip
We consider  the real cyclotomic unit
$\eta_1 := (1 + \xi\,\zeta)^{e_\omega}\,\zeta^{-\frac{1}{2}}$
(see Definition~4). Recall that for $n \leq 2$, $\eta_1 \in K^{\times p}$, so we 
suppose $n>2$.

\smallskip
Let $c$ be the complex conjugation.
We suppose in this subsection that the ideal class of ${\mathfrak q}^{1-c}$
is in the $p$th power of the class group of $L$, which is equivalent 
to  ${\mathfrak q}^{1-c} = {\mathfrak a}^p\,(\alpha)$ for an ideal 
${\mathfrak a}$ of $L$ and an $\alpha \in L^\times$ prime to $p$.
This condition is also equivalent to ${\mathfrak q} = {\mathfrak b}^{1+c}
{\mathfrak a'}^p\,(\alpha')$ for ideals ${\mathfrak a'}$,
${\mathfrak b}$ of $L$ and an $\alpha' \in L^\times$.

\smallskip
We can suppose $\alpha \equiv 1 \pmod {p}$ (see the footnote in Theorem~2),
so that we get ${\mathfrak q}^{1-c} = {\mathfrak a}^p\,(1 + p\,\beta)$, $\beta$
$p$-integer in $L$.
Taking the absolute norm gives
$\No_{L/\Q} (1 + p\,\beta) = \No_{L/\Q} ({\mathfrak a})^{-p}$
which is a rational congruent to 1 modulo $p^2$. Thus since
$\No_{L/\Q} (1 + p\,\beta) \equiv 1 + p\,{\rm Tr}_{L/\Q}(\beta)
\pmod {p^2}$, where ${\rm Tr}_{L/\Q}$ is the absolute trace,
we obtain  ${\rm Tr}_{L/\Q}(\beta)  \equiv 0 \pmod {p}$.
This remark will be used later.

\smallskip
We note that, as for the context of  Theorem 2, if 
$q-1 =: d\,p^r$, $p\notdiv d$, and if the condition
$\wt {\mathfrak q}^{1-c} = \wt {\mathfrak a}^p\,(1 + p\,\wt \beta)$ is 
satisfied for $\wt n = d$ and 
$\wt {\mathfrak q} \div q$ in $\wt L = \Q(\mu_{\wt n})$,
then it is satisfied for any divisor $n>2$ of $\wt n$ and the 
corresponding ideal ${\mathfrak q} = \No_{\wt L/L} (\wt {\mathfrak q})$;
we then have $\beta \equiv {\rm Tr}_{\wt L/L}(\wt \beta) \pmod {p}$.

\smallskip
In $M= LK$ we have:
$$\Big(\frac{\eta_1}{({\mathfrak q})^{1-c}}\Big)_{\!\!M} =
\Big(\frac{\eta_1}{\prod_{{\mathfrak Q} \div {\mathfrak q}}
{\mathfrak Q}^{1-c}}\Big)_{\!\!M} = 
\prd_{{\mathfrak Q} \div {\mathfrak q}}
\Big(\frac{\eta_1}{{\mathfrak Q}^{1-c}}\Big)_{\!\!M} =
\Big(\frac{\eta_1}{{\mathfrak Q}}\Big)_{\!\!M}^{2\,\frac{p-1}{f}}, $$
where $f$ is the residue degree of $q$ in $K/\Q$; indeed, we have:
$$\Big(\frac{\eta_1}{{\mathfrak Q}^{1-c}}\Big)_{\!\!M} =
\Big(\frac{\eta_1}{\mathfrak Q}\Big)_{\!\!M} .
\Big(\frac{\eta_1}{{\mathfrak Q}^{c}}\Big)^{-1}_{\!\!M} =
\Big(\frac{\eta_1}{\mathfrak Q}\Big)_{\!\!M} .\
c\,\Big(\frac{\eta_1}{{\mathfrak Q}}\Big)^{-1}_{\!\!M}=
\Big(\frac{\eta_1}{\mathfrak Q}\Big)^2_{\!\!M}, $$
since $\eta_1$ is real, hence the result since the symbol of $\eta_1$ 
does not depend on the choice of
${\mathfrak Q}$ above ${\mathfrak q}$.
But $\Big(\Frac{\eta_1}{({\mathfrak q})^{1-c}}\Big)_{\!\!M} =
\Big(\Frac{\eta_1}{({\mathfrak a}^p)\, (\alpha)}\Big)_{\!\!M} =
\Big(\Frac{\eta_1}{ (\alpha)}\Big)_{\!\!M}$.
Then using the general 
reciprocity law (see e.g. [Gr2, II.7.4.4]) we obtain,
since $\eta_1$ is a unit:
$$\Big(\frac{\eta_1}{\alpha}\Big)_{\!\!M} = 
\Big(\frac{\eta_1}{\alpha}\Big)_{\!\!M}\,
\Big(\frac{\alpha}{\eta_1}\Big)_{\!\!M}^{-1} =
\prd_{{\mathfrak P} \div p} 
\big(\eta_1,\alpha \big)_{\mathfrak P}^{-1}, $$
product over the prime ideals ${\mathfrak P}$
of $M$ above $p$; since
$M/L$ is totally ramified at $p$, we will write by abuse
$\big(\eta_1,\alpha \big)_{\mathfrak p}$ for these 
Hilbert symbols, where ${\mathfrak p} \div  p$ in $L$,
knowing that they are defined on $M^\times\times M^\times$
with values in $\mu_p$.\,\footnote{\,Warning: in the literature, two 
definitions are possible, which give the Hilbert symbol or 
its inverse; this is the case with the reference [Ko] used below,
by comparison with our${}'$s (see  e.g. [Gr2, II.7.3.1]).}

\medskip
Thus we have obtained:
$$\Big(\frac{\eta_1}{{\mathfrak Q}}\Big)_{\!\!M} = 
\prd_{{\mathfrak p} \div p} \big(\eta_1,\alpha \big)_{\mathfrak 
p}^{\frac{f}{2}}. $$

We refer now to the Br\"uckner--Vostokov 
explicit formula proved in [Ko, 6.2, Th.\,2.99]
by giving some details for the convenience of the reader,
using similar notations.

\smallskip
Consider the uniformizing parameter $\pi := \zeta-1$
of the completions $M_{\mathfrak P}$ of $M$ at ${\mathfrak P} \div \mathfrak 
p \div p$.
The inertia field is $L_{\mathfrak p}$.
We need the formal series
$t(x) := 1-(1+x)^p$ since $1-\zeta=-\pi$ here, for which
$t(x)^{-1}$ is the Laurent series:
$$ -\Frac{1}{x^p}\big( 1 - p\,\big(\Frac{c_1}{x} + \cdots + 
\Frac{c_{p-1}}{x^{p-1}}\big) + p^2\,\big(\Frac{c_1}{x} + \cdots + 
\Frac{c_{p-1}}{x^{p-1}}\big)^2 - \cdots \big),$$
where the $c_i$ are integers.

\smallskip
We associates with $\eta_1 \equiv 1 + \theta\,\pi \pmod {\pi^2}$,
where $\theta := \frac{1}{2}\frac{\xi-1}{\xi+1}$
(see Subsection 4.1), and with $\alpha = 1 + p\,\beta$,
the series:
\begin{eqnarray*}
    &&F(x) \equiv 1 + \theta\,x \pmod {(x^2)}, \\
    &&G(x) := 1 + p\,\beta \ \rm (a\ constant\ series),
\end{eqnarray*}
such that $F(\pi) \equiv \eta_1 \pmod {\pi^2}$ and
$G(\pi) = \alpha$. Recall that 
log is the $p$-adic logarithm
and  dlog the loga\-rithmic derivative;
so ${\rm dlog}(G) = 0$ giving:
$$(F,G) = - \Frac{1}{p^2} \,\cdot\,{\rm log} \big(\Frac{G^p}{\sigma_p(G)}\big)
\,\cdot\,{\rm dlog}(\sigma_p(F)), $$
where $\sigma_p$ is the Frobenius automorphism on $L_{\mathfrak p}$
extended to series by putting $\sigma_p(x) := x^p$.

\smallskip
Thus $\sigma_p(G) =  1 + p\,\sigma_p(\beta)$,  \ 
$\sigma_p(F) \equiv  1 + \sigma_p(\theta)\,x^p \pmod {(x^{2p})}$, 
giving:
\begin{eqnarray*}
     {\rm log} \big(\Frac{G^p}{\sigma_p(G)}\big) \!\!\! &\equiv& \!\!\!
    -p\,\sigma_p(\beta) \pmod {p^2} \\
     {\rm dlog}(\sigma_p(F)) \!\!\!&\equiv& \!\!\! p\,\sigma_p(\theta)\,x^{p-1}
    \pmod {(x^{2p}, p\,x^{2p-1})},
\end{eqnarray*}
and finally:
$$(F,G) \equiv \sigma_p(\theta\,\beta) \,x^{p-1} \pmod {\big (p\,x^{p-1} , 
\,x^{2p-1}, \Frac{x^{2p}}{p}\big)}. $$

Then the residue of $t(x)^{-1}\,(F,G)$ is that of:
$$-\Frac{1}{x^p}\,\sigma_p(\theta\,\beta) \,x^{p-1} =
-\Frac{1}{x}\,\sigma_p(\theta\,\beta) 
\pmod {\big(\Frac{p}{x} ,x^{p-1}, \Frac{x^{p}}{p}\big)}, $$
hence it is $-\sigma_p(\theta\,\beta) \pmod {p}$ since the generator
$\Frac{x^{p}}{p}$ of the above ideal gives rise to a residue only with a term 
of the form $\Frac{c}{x^{p+1}}$ of $t(x)^{-1}$
(to give $\Frac{c}{p\,x}$) in which case $c$ is a 
multiple of $p^2$ (see the expression of $t(x)^{-1}$).

\smallskip
To conclude we have to take the 
absolute local trace (which eliminates the action of the Frobenius):
$${\rm Tr}_{M_{\mathfrak P}/\Q_p} (-\theta\,\beta) =
(p-1)\,{\rm Tr}_{L_{\mathfrak p}/\Q_p} (-\theta\,\beta)
\equiv {\rm Tr}_{L_{\mathfrak p}/\Q_p} (\theta\,\beta) \pmod {p}. $$

Then $\big(\eta_1,\alpha \big)_{\mathfrak p} =
\zeta_{}^{-{\rm Tr}_{L_{\mathfrak p}/\Q_p}
\big(\frac{1}{2}\frac{\xi-1}{\xi+1}\, \beta\big)}$ because of our definition of
the Hilbert symbol, and
$\prd_{\mathfrak p} \big(\eta_1,\alpha \big)_{\mathfrak p}
= \zeta_{}^{-\sum_{\mathfrak p}{\rm Tr}_{L_{\mathfrak p}/\Q_p}
\big(\frac{1}{2}\frac{\xi-1}{\xi+1}\, \beta\big)} =
\zeta_{}^{-{\rm Tr}_{L/\Q}
\big(\frac{1}{2}\frac{\xi-1}{\xi+1}\, \beta \big)}$,
the global trace being the sum of the local ones.

\smallskip
We have $\Frac{1}{2}\Frac{\xi-1}{\xi+1}\, \beta =
\big(\Frac{1}{2} - \Frac{1}{\xi+1}\big)\, \beta$, so
the final expression of the trace is 
$-\,{\rm Tr}_{L/\Q}\big(\Frac{\beta}{\xi+1}\big)$
since that of $\beta$ is zero modulo $p$.

\smallskip
This yields to
$\Big(\Frac{\eta_1}{{\mathfrak Q}}\Big)_{\!\!M} =
\prd_{\mathfrak p} \big(\eta_1,\alpha \big)_{\mathfrak p}^{\frac{f}{2}} =
\zeta_{}^{\frac{1}{2}f\,{\rm Tr}_{L/\Q}\big(\frac{\beta}{\xi+1}\big)}$.

\smallskip
We have obtained the following explicit formula.

\begin{theo} Let $q\ne p$ be a  prime number, let $n \div  q-1$ be such 
that $p\notdiv n$ and $n>2$. Let $\xi$ of order $n$ and let ${\mathfrak q}$
be any prime ideal of $L = \Q(\mu_n)$ divi\-ding~$q$.
We suppose that the ideal class of ${\mathfrak q}^{1-c}$
is the $p$th power of a class of $L$, which is equivalent 
to  ${\mathfrak q}^{1-c} = {\mathfrak a}^p\,(1 + p\,\beta)$ for an ideal 
${\mathfrak a}$ of $L$ and $\beta$ $p$-integer in $L$.\,\footnote{\,As 
we know, this condition is also equivalent to 
 ${\mathfrak q} = {\mathfrak b}^{1+c}
{\mathfrak a'}^p\,(1 + p\,\beta')$ for ideals ${\mathfrak a'}$,
${\mathfrak b}$ of $L$ and  $\beta'$ $p$-integer in $L$.
It is satisfied as soon as the class of ${\mathfrak q}^{1-c}$ is of 
order prime to $p$. }
Put $\eta_1 := (1 + \xi\,\zeta)^{e_\omega}\,\zeta^{-\frac{1}{2}}$
(see Definition~4).

\smallskip\noindent
Then for any ${\mathfrak Q}\div {\mathfrak q}$ in $M:= LK$,
$\Big(\Frac{\eta_1}{{\mathfrak Q}}\Big)_{\!\!M} =
\zeta_{}^{\frac{1}{2}f\,{\rm Tr}_{L/\Q}\big(\frac{\beta}{\xi+1}\big)}$,
where $f$ is the residue degree of $q$ in $K/\Q$ and ${\rm Tr}_{L/\Q}$
the absolute trace in $L/\Q$.~\hfill\fin
\end{theo}

This gives again the situation of  Theorem 2 when 
$\beta  \equiv \beta^+ \pmod {p}$,  $\beta^+ \in L^+$, since we 
then have:
$${\rm Tr}_{L/\Q}\big(\Frac{\beta}{\xi+1}\big) \equiv
{\rm Tr}_{L^+/\Q}\big(\Frac{\beta^+}{\xi+1} + \Frac{\beta^+}{\xi^c+1}\big)
\equiv {\rm Tr}_{L^+/\Q}(\beta^+)\equiv 0 \pmod {p}, $$
since ${\rm Tr}_{L/\Q}(\beta)\equiv 0 \pmod {p}$.

\smallskip
This theorem confirms the independence, with the SFLT problem,
of the  class field theory properties of the fields $\Q(\mu_n)$.
Meanwhile,  under a nontrivial solution of the SFLT equation, 
for suitable values of $q$, $n \div q-1$ and $\xi$
of order $n$, the quantity ${\rm Tr}_{L/\Q}(\frac{\beta_\xi}{\xi+1})$,
where $\beta_\xi$ corresponds to ${\mathfrak q}_\xi$,
is imposed, which yields to infinitely many conditions.
But as usual we need to explain how the 
case $p=3$ interferes appropriately with the arithmetic of the fields $\Q(\mu_n)$
(see Section 9).

\begin{rema} {\rm
Suppose, as in  Theorem 3, that ${\mathfrak q}^{1-c} = {\mathfrak a}^p\,(1 + p\,\beta)$
for an ideal ${\mathfrak a}$ of $L$ and $\beta$ $p$-integer in $L = \Q(\mu_n)$,
with $n \div  q-1$ such that $p\notdiv n$ and $n>2$.
To obtain that ${\mathfrak q}$ is totally split in $F_n/L$,
we study the equivalent condition $\Big(\Frac{\eta_1^t}{{\mathfrak Q}}\Big)_{\!\!M} = 1$
for all $t \in {\rm Gal}(M/K)/\langle\,t_{-1}\,\rangle$; from the 
theorem this is equivalent to
${\rm Tr}_{L/\Q}\big(\Frac{\beta}{\xi^t + 1}\big)  \equiv 0
\pmod {p}$ for all $t \in {\rm Gal}(L/\Q)/\langle\,t_{-1}\,\rangle$. 

\smallskip
This can be written in the following two forms:
$$\sm_{\tau \in {\rm Gal}(L/\Q)}\Frac{\beta^\tau}{\xi^{t \tau} + 1} \equiv 0
\pmod {p},\ \ {\rm for\ all\ } t \in {\rm 
Gal}(L/\Q)/\langle\,t_{-1}\,\rangle. $$
$$\sm_{\tau \in {\rm Gal}(L/\Q)}\Frac{\beta^{t \tau}}{\xi^{\tau} + 1} \equiv 0
\pmod {p},\ \ {\rm for\ all\ } t \in {\rm Gal}(L/\Q)/\langle\,t_{-1}\,\rangle. $$

So we obtain  linear systems (with ``variables'' $\beta^\tau$ and
$\frac{1}{\xi^{\tau} + 1}$, respectively),
whose matrices have $\phi(n)$ columns and
$\frac{1}{2}\,\phi(n)$ lines, and the rank over $\F_p$ of the first matrix
(less than or equal to $\frac{1}{2}\,\phi(n)$) gives a more precise 
approach of the required conditions on $\beta$; the condition $\beta
\equiv \beta^+ \pmod {p}$ is sufficient (use the second system)
but not necessary as soon as 
the rank of the matrix is less than $\frac{1}{2}\,\phi(n)$.

\smallskip
Let $Z'_L$ be the ring of $p$-integers of $L$. Then the
knowledge of the image of $\beta$ in $Z'_L/p\,Z'_L$ summarizes all the 
needed local properties of $\eta_1$ at the prime~$q$. Since
$Z'_L/p\,Z'_L$ is the product of the residue fields of $L$ at the 
primes ${\mathfrak p} \div p$ in $L$, any analytic approach is 
available.~\hfill\fin}
\end{rema}

\begin{ex}{\rm  Take $p=5$, $n=4$, and $q \ne 5$ prime congruent to 1 modulo~4. 
Put $q = a^2 + b^2$ as usual; then ${\mathfrak q} = (a+i\,b)$ and
 ${\mathfrak q}^4 = (A+i\,B)$, with $A = a^4+b^4 - 6\,a^2b^2$,
 $B = 4\,ab(a^2-b^2)$. We then have:
 $${\mathfrak q}^{1-c} = {\mathfrak q}^{5(1-c)} 
 \Big(\Frac{A-i\,B}{A+i\,B}\Big)=:
  {\mathfrak q}^{5(1-c)} \big( 1 + 5\,\beta \big). $$
  
 Since $A+i\,B \equiv 1 \pmod {5}$, we get $A \equiv 1$
and $B \equiv 0$ $\pmod {5}$, and a straightforward computation gives:
$$\beta \equiv - \Frac{8\,i\, ab(a^2-b^2)}{5}\ \, {\rm and}\ \,
\Frac{\beta}{i+1} \equiv -  \Frac{4\,(i+1)\,ab(a^2-b^2)}{5} \pmod {5}, $$
which yields to $\frac{1}{2}{\rm Tr}_{L/\Q}\big(\Frac{\beta}{i+1}\big)
\equiv - \frac{1}{2}\Frac{8\,ab(a^2-b^2)}{5}  \pmod {5}$, hence:
$$\Big(\Frac{\eta_1}{{\mathfrak Q}}\Big)_{\!\!M} =
\zeta_{}^{f\,\frac{ab(a^2-b^2)}{5}}. $$

So the symbol is trivial if and only if $ab(a^2-b^2)
\equiv 0 \pmod {25}$. We find the values
$q = 313$ ($a=13$, $b=12$), $q = 317$ ($a=14$, $b=11$),\,\ldots\,

\smallskip
 For $q = 457$ ($a=21$, $b=4$), we
have $\kappa  \equiv 0 \pmod {5}$. A case with $25 \div  ab$
is given by $q = 641$ ($a=25$, $b=4$). 

\smallskip
The symbol is nontrivial for the values 
$q = 13$ ($a=3$, $b=2$) where $\Big(\Frac{\eta_1}{{\mathfrak Q}}\Big)_{\!\!M} =
\zeta^{4}$, $q = 17$ ($a=4$, $b=1$) where $\Big(\Frac{\eta_1}{{\mathfrak Q}}\Big)_{\!\!M} =
\zeta^3$,\,\ldots~\hfill\fin }
\end{ex}

\section {Decomposition law of $q$ in $H_{\Q(\mu_{q-1})}/\Q(\mu_{q-1})$ and conjectures}

In this section we study in full generality
the situation that we have encountered in the previous sections.

\subsection{Law of $\rho$-decomposition relative to the family  
${\mathcal F}_n$}

Let $p>2$ be a fixed prime number and let $\rho = \frac{v}{u}$,
with ${\rm g.c.d.}\,(u,v) = 1$, be a 
fixed rational distinct from $0$ and $\pm 1$. We do not suppose
any relation of SFLT type between $u$ and $v$.

\smallskip
For any prime number $q \ne p$ 
let $f$ be the residue degree of $q$ in
$K$ and put $\kappa := \frac{q^f-1}{p}$.
Note that we have the relation (see Definition 2, (i)):
$$\ov \kappa :=\Frac{q^{p-1}-1}{p} \equiv \Frac{p-1}{f}\,\kappa
\equiv -\Frac{1}{p}\,{\rm log}(q)  \pmod {p}. $$

We consider the {\it infinite} set of prime numbers:
$$Q_\rho := \big\{ q , \hbox{ $q \notdiv u\,v\, (u^2-v^2)$}
\hbox{ and the order of $\rho$ modulo $q$ is  prime to $p$} \big\}. $$ 

For  $q \in Q_\rho$, let $n$ be the order of $\rho$ 
modulo $q$ (by definition we have $p\notdiv n$, $n>2$); 
from Lemma 2, $q \in Q_\rho$ is equivalent to $q\notdiv n$,
$q \div  \Phi_n(u,v)$, for $n>2$, $p\notdiv n$).

\smallskip
We consider the fields $K:=\Q(\mu_p)$, $L:=\Q(\mu_n)$,
and $M:=LK$ which only depend on $q$ (for $\rho$ fixed).

\smallskip
We associate with 
$q$ a pair $(\xi, {\mathfrak q})$ where the primitive $n$th root of 
unity  $\xi \in L$ and the prime ideal ${\mathfrak q} \div q$ of $L$
are characterized by the congruence $\xi \equiv \rho$ $\pmod {\mathfrak q}$;
thus, ${\mathfrak q} = (q, u\,\xi-v)$ is also denoted ${\mathfrak q}_\xi$
as in the previous sections (see Definition 3).
As we know, this pair is defined up to $\Q$-conjugation and
we obtain an equivalence relation. The class associated to $q$ is 
well-defined.
Of course, the classes of $(\xi_1, {\mathfrak q}_1)$ and $(\xi_2, 
{\mathfrak q}_2)$,
corresponding to  different primes $q_1$ and $q_2$, are relative to 
the fields 
$L_1=\Q(\mu_{n_1})$, $n_1 \div  q_1-1$, and $L_2=\Q(\mu_{n_2})$,
$n_2 \div  q_2-1$, and one of the main problem would be to try to
connect the two situations.

\smallskip
From the construction of the extensions
$F_\xi$ and  $F_n \subseteq H_L^-{\st [p]}$
given in Subsections 4.2 and 5.2 via the real cyclotomic unit:
$$\eta_1 := (1 + \xi\,\zeta)^{e_\omega}\,\zeta^{-\frac{1}{2}}, $$
the pair $(F_\xi,{\mathfrak q}_\xi)$ is  defined up to $\Q$-conjugation since
$(t F_\xi,{\mathfrak q}_\xi^t) = (F_{\xi^t},{\mathfrak q}_{\xi^t})$
corresponds to $(\xi^t, {\mathfrak q}_{\xi^t})$;
thus the class of the pair $(F_\xi,{\mathfrak q}_\xi)$ 
(or similarly of the pair $(\eta_1,{\mathfrak Q}_{\xi}\div {\mathfrak q}_{\xi})$)
characterizes  the class of $(\xi, {\mathfrak q}_\xi)$
and reciprocally.
Recall that $F_\xi = F_{\xi^{-1}}$ is diedral over $L^+$.

\smallskip
The following lemma is elementary but  gives details on the action of
${\rm Gal}(L/\Q)$ on the family of Frobenii:
 
\begin{lemm} Let $\varphi_\xi^{} :=
\Big(\Frac{F_\xi/L}{{\mathfrak q}_\xi}\Big)$ be the Frobenius
automorphism of the prime ideal 
${\mathfrak q}_\xi = (q, u\,\xi-v)$ in $F_\xi/L$.

\noindent
Then $\varphi_{\xi^t}^{} := \Big(\Frac{F_{\xi^t}/L}{{\mathfrak q}_{\xi^t}}\Big)
= \varphi_{\xi}^{t} := t\,\varphi_{\xi}\,t^{-1}$ for all $t \in {\rm Gal}(L/\Q)$.

\noindent
If $t=t_{-1}$, then $\varphi_{\xi^{-1}} = \varphi_{\xi}^{t_{-1}} = \varphi_{\xi}^{-1}$
in $F_{\xi^{\pm 1}}/L$.
\end{lemm}

\begin{proof} From the defining congruence $\varphi_\xi^{}\,(\alpha) 
\equiv \alpha^q \pmod {{\mathfrak q}_\xi}$ for all integers $\alpha$ of 
$F_\xi$, we get easily $t'\,\varphi_\xi^{}\, (\alpha) \equiv 
t' (\alpha)^q \pmod {{\mathfrak q}_{\xi^t}}$,
for any $\Q$-isomorphism $t'$ of $F_{\xi}$ such that $t'{\vert_L} = t$.
Put $t' (\alpha) =: 
 \beta \in F_{\xi^t}$; this yields $t'\,\varphi_\xi^{}\,t'{}^{-1} 
(\beta) \equiv \beta^q \pmod {{\mathfrak q}_{\xi^t}}$ for all integers $\beta$ of 
$F_{\xi^t}$, proving the lemma by uniqueness of the Frobenius.
\end{proof}

The Frobenius of ${\mathfrak q}_\xi$ in $F_\xi/L$
is characteristic of the class of $(\xi, F_\xi)$
since we still have $(\xi^t, \varphi_\xi^t) =
(\xi^t, \varphi_{\xi^t})$ by conjugation. This leads to give
the following definition.

\begin{defi} {\rm Let $\rho := \frac{v}{u}$, with g.c.d.\,$(u,v) = 1$,
 be a fixed rational,  distinct from 0 and $\pm 1$.    
 For $n>2$ prime to $p$,
let $K:=\Q(\mu_p)$, $L=\Q(\mu_n)$,  $M=LK$, and for $\xi$ 
of order $n$, let $F_\xi$ be such that $F_\xi M = M\big( \sqrt[p]
{(1+\xi\,\zeta)^{e_\omega}\,\zeta^{-\frac{1}{2}}}\,\big)$.

\smallskip
 (i) For any prime $q$, $q \notdiv n$, $q \div \Phi_n(u,v)$
 (i.e., $q \notdiv u\,v$ and $\rho$ is of order $n$ modulo $q$),
and for ${\mathfrak q}_\xi = (q, u\,\xi-v) \div q$, we consider
the class of Frobenii:
$$\Big(\Frac{F_{\xi^t}/L}{{\mathfrak q}_{\xi^t}}\Big) = 
\Big(\Frac{F_\xi/L}{{\mathfrak q}_\xi}\Big)^t, \ \,
t \in {\rm Gal}\,(L/\Q), $$
that we normalize in the following way:

-- if $\kappa \not\equiv 0 \pmod {p}$, we put
$\Big[\Frac{F_*/L}{{\mathfrak q}_*}\Big]_{\rho,n} := 
\Big( \Big(\Frac{F_{\xi^t}/L}{{\mathfrak q}_{\xi^t}}
\Big)^{\frac{p}{{\rm log}(q)}}\Big)_{t \in {\rm Gal}\,(L/\Q)}$;

-- if $\kappa \equiv 0 \pmod {p}$, we put
$\Big[\Frac{F_*/L}{{\mathfrak q}_*}\Big]_{\rho,n} := 1$.

\smallskip
(ii) Call ${\mathcal F}_n$ the canonical family 
$(F_{\xi^t})_t = (F_{\xi'})_{\xi'\,{\rm of \,order\ } n}$ defining $F_n/L$,
where $F_n \subseteq H_L^-{\st [p]}$
is the compositum of the $F_{\xi^t}$,
$t \in {\rm Gal}(L/\Q)/<t_{-1}>$.\,\footnote{\,Remark that the 
only knowledge of $n$ determines the field $L = \Q(\mu_n)$ then the 
family ${\mathcal F}_n$.}

(iii) The symbol $\Big[\Frac{F_*/L}{{\mathfrak q}_*}\Big]_{\rho,n}\ $
is called, by abuse of language, the {\it law of $\rho$-decom\-position of $q$ for
the family ${\mathcal F}_n$}.~\hfill\fin }
\end{defi}

This object depending on $\rho$ and $n$ is, for each $q$,
relative to a universal family  ${\mathcal F}_n$ which is independent of any hypothetic
nontrivial solution of the SFLT equation.

\smallskip
Let $\sigma$ be a generator of ${\rm Gal}(F_\xi/L)$;
since the Frobenius $\varphi_\xi$ in $F_\xi/L$
is well defined, it is of the form $\sigma^r$, $r \in \Z/p\Z$,
so that (when $\kappa \not \equiv 0 \pmod {p}$)
the  symbol $\Big[\Frac{F_*/L}{{\mathfrak q}_*}\Big]_{\rho,n}$ represents
the family (or class):
$$\Big( \sigma^t \Big)^{r\,\frac{p}{{\rm log}(q)}}_{t\in {\rm Gal}(L/\Q)}
= \Big(t\,.\,\sigma\, .\, t^{-1} \Big)^{r\,\frac{p}{{\rm log}(q)}}_{t\in {\rm Gal}(L/\Q)}\,. $$

Thus the symbol $\Big[\Frac{F_*/L}{{\mathfrak q}_*}\Big]_{\rho,n}$ can take
$p-1$ nontrivial ``values'' (called the cases of $\rho$-inertia  of 
$q$  for ${\mathcal F}_n$,
when $r\not\equiv 0 \pmod {p}$) and
a trivial one (the $\rho$-splitting of $q$ for ${\mathcal F}_n$).
The case $\kappa \equiv 0 \pmod {p}$ gives the $\rho$-splitting of $q$
 for ${\mathcal F}_n$.

\smallskip
Note that the Frobenii  $\wt\varphi_{\xi^t} :=
\Big(\Frac{F_n/L}{{\mathfrak q}_{\xi^t}}\Big)$,
 $t \in {\rm Gal}(L/\Q)$,
are a priori unknown and must not be confused with 
$\Big[\Frac{F_*/L}{{\mathfrak q}_*}\Big]_{\rho,n}$;
they are conjugated, of order 1 or $p$, and the case of order 1
 is  very rare since it means that $q$ totally splits in $F_n/\Q$,
i.e., $\Big(\Frac{F_n/L}{{\mathfrak q}_{\xi^t}}\Big) = 1$
for all $t \in {\rm Gal}(L/\Q)$ (situation of   Theorem 2). 

\smallskip
The restriction of $\wt\varphi_\xi$ to $F_\xi$ gives by 
definition $\varphi_\xi^{}$.
Its restrictions to the other $F_{\xi^t}$ are the 
$\Big(\Frac{F_{\xi^t}/L}{{\mathfrak q}_\xi}\Big) =
\Big(\Frac{F_{\xi}/L}{{\mathfrak q}_{\xi^{t^{-1}}}}\Big)^t$.

\smallskip
In the previous sections, in the case $\kappa \not\equiv 0 \pmod {p}$
for $p>3$, we have used, as a contradiction for the existence
of a solution of Fermat${}'$s equation, the splitting of
${\mathfrak q}_{\xi}$ in $F_{\xi}$ for infinitely many values of $q$
(taking for instance $(u,v) = (x,y)$, $(y,x)$, $(z,y)$, or $(y,z)$).
Same remark for a solution $(u,v)$ of the SFLT equation under the 
condition $u-v \not\equiv 0 \pmod {p}$.

\smallskip
This property ``~${\mathfrak q}_{\xi} = (q, u\,\xi -v)$ splits in 
$F_\xi$\,'', independent of the choice of the representative
pair as Lemma 8 shows, will be called by analogy the
``\,$\rho$-splitting of $q \in Q_\rho$ for ${\mathcal F}_n$\,'',
$\rho := \frac{v}{u}$. 
It is equivalent to $\Big[\Frac{F_*/L}{{\mathfrak q}_*}\Big]_{\rho,n} = 1$.

\begin{rema} {\rm In a 
probabilistic point of view, the $\rho$-splitting of $q \in Q_\rho$
for ${\mathcal F}_n$ has a probability around $\frac{1}{p}$,
and we can hope a strong incompatibility for analytic reasons
since $Q_\rho$ is infinite.
If we ask that $q$ be totally split 
in $F_n$, this means that each ${\mathfrak q} \div q$ splits in
$F_\xi=F_{-\xi}$ (for any fixed $\xi$) and the probability is around
$\big(\frac{1}{p}\big)^{\frac{1}{2}\,\phi(n)}$ which tends to 0 
rapidly with $q \rightarrow\infty$.~\hfill\fin }
\end{rema}

Put:

\centerline{$\ Q_\rho^{\rm spl} := \big\{ q \in Q_\rho, \ q \ \hbox{has a 
$\rho$-splitting for ${\mathcal F}_n$} \big\}. $}

\medskip
With a counterexample $(u,v)$ to SFLT, we have, from a pair 
$(\xi,{\mathfrak q}_\xi)$, the following results proved in Theorem 1.
Put $\rho := \frac{v}{u}$; we may have
$u \equiv 0 \pmod {p}$ in which case $\rho$ is not defined modulo $p$, but
is always defined as a rational, so we preserve $u$ and $v$ in the 
congruences modulo $p$.

\smallskip
In the nonspecial cases
 (i.e.,  $v+u \not \equiv 0\pmod {p}$):
$$\Big(\frac{\eta_1}{{\mathfrak Q}}\Big)_{\!\!M} \  = \ \zeta^{\frac{1}{2}\, 
\frac{v-u}{v+u}\,\kappa},\ \, {\rm for\  all}\ {\mathfrak Q}
\div {\mathfrak q}_\xi, \ \, p\geq 3 ; $$

 In the special case (i.e., $v+u \equiv 0 \pmod {p}$):
  \begin{eqnarray*}
      \Big(\frac{\eta_1}{{\mathfrak Q}}\Big)_{\!\!M}  &=& 1,
 \ \, {\rm for\  all}\ {\mathfrak Q}
 \div {\mathfrak q}_\xi,  \ \,  p>3, \\
 \Big(\frac{\eta_1}{{\mathfrak Q}}\Big)_{\!\!M}  &=& 
 \zeta^{\frac{1}{2}\,\frac{v+u}{3\,v}\,\kappa},
 \ \, {\rm for\  all}\ {\mathfrak Q}
 \div {\mathfrak q}_\xi,  \ \, p=3.
 \end{eqnarray*}

Recall that for SFLT we cannot exclude the 
case $u-v \equiv 0 \pmod {p}$ contrary to  FLT
 for $(u,v) = (x,y)$, $(y,x)$, $(z,y)$, or $(y,z)$. This explain that for SFLT
(first case and $\kappa \not\equiv 0 \pmod {p}$)
we cannot use, as a general contradiction, the $\rho$-splitting
of $q$ for ${\mathcal F}_n$.

 \smallskip
 This does not matter since the existence of a nontrivial solution to SFLT 
 is equivalent to a precise law of $\rho$-decomposition of $q$
 for ${\mathcal F}_n$, i.e., a precise value of the symbol 
 $\Big[\Frac{F_*/L}{{\mathfrak q}_*}\Big]_{\rho,n}$ (which can be 
trivial even if $\kappa \not\equiv 0 \pmod {p}$
when  $u-v \equiv 0 \pmod {p}$).

\smallskip
 More precisely, we have the following lemma giving
 the action of the Frobenius, which determines explicitly the 
 law of $\rho$-decomposition (the case $p=3$ being immediate 
 from Theorem 1, we assume for simplicity $p>3$):
 
 \begin{lemm} We suppose given, for the prime $p> 3$, a relation of the form 
 $(u+v\,\zeta)\,\Z[\zeta] = {\mathfrak w}_1^p\ { or}\ 
 {\mathfrak p}\, {\mathfrak w}_1^p$, with {\rm g.c.d.}\,$(u,v) = 1$. 
 
 \smallskip\noindent
Let  $q$ be a prime number such that
 $q \notdiv u\,v$, and such that the order 
$n$ of  $\rho  := \frac{v}{u}$ modulo $q$ is  prime to $p$.
 Let ${\mathfrak Q} \div  {\mathfrak q}_\xi$ in $M$,
  where $(\xi,{\mathfrak q}_\xi)$ represents the class cor\-respon\-ding 
to $q$.
     
 \smallskip\noindent
  Let  $\Big(\Frac{M(\sqrt[p]{\eta_1}\,)/M}{{\mathfrak Q}}\Big)$
  be the Frobenius automorphism of ${\mathfrak Q}$ in 
  $M(\sqrt[p]{\eta_1}\,)/M$,  where $\eta_1 := (1 + \xi\,\zeta)^{e_\omega}\, 
 \zeta^{-\frac{1}{2}}$. We have:
   
\smallskip\smallskip
 (i) Nonspecial cases. If $v+u \not\equiv 0 \pmod {p}$,  then
$\ \Big(\Frac{M(\sqrt[p]{\eta_1}\,)/M}{{\mathfrak Q}}\Big)
\, .\,\sqrt[p]{\eta_1}
  = \zeta^{\frac{1}{2}\,\frac{v-u}{v+u}\,\kappa}
  \cdot\,\sqrt[p]{\eta_1}$. 
  
 \smallskip\smallskip
 (ii) Special case. If $v+u \equiv 0 \pmod {p}$, then
$\ \Big(\Frac{M(\sqrt[p]{\eta_1}\,)/M}{{\mathfrak Q}}\Big)
  \,\cdot\sqrt[p]{\eta_1}= \sqrt[p]{\eta_1}$.  
\end{lemm}
 
 \begin{proof} From the  defining congruence
 $\big(\sqrt[p]{\eta_1}\big)^{\sigma} \equiv 
\big(\sqrt[p]{\eta_1}\big) ^{\,q^{f}}
 \pmod {\mathfrak Q}$, for the Frobenius auto\-morphism
 $\sigma := \Big(\Frac{M(\sqrt[p]{\eta_1}\,)/M}{{\mathfrak Q}}\Big)$, 
 we get:
$$\big(\sqrt[p]{\eta_1}\big)^{\sigma - 1} \equiv 
\big(\sqrt[p]{\eta_1}\big)^{q^{f}-1} \equiv \eta_1^{\,\kappa} \equiv
\Big(\Frac{\eta_1}{{\mathfrak Q}}\Big)_{\!\!M} \pmod {\mathfrak Q}. $$
 Hence the result since $\Big(\Frac{\eta_1}{{\mathfrak 
Q}}\Big)_{\!\!M} = \zeta^{\frac{1}{2}\, \frac{v-u}{v+u}\,\kappa}$
(resp. $1$) in the nonspecial cases (resp. in the special case).
\end{proof}

We intend now, in the following theorem,
to translate this property into a property of the symbol
 $\Big[\Frac{F_*/L}{{\mathfrak q}_*}\Big]_{\rho,n}$ (see Definition 5), which will give the 
 main phenomenon about the existence of a nontrivial solution to the SFLT equation
 (see also Remark~9).

\begin{theo} Let $p$ be a prime number, $p>3$.
 We suppose given a solution of the SFLT equation
 $(u+v\,\zeta)\,\Z[\zeta] = {\mathfrak w}_1^p$ or
 ${\mathfrak p}\, {\mathfrak w}_1^p$, with ${\rm g.c.d.}\,(u,v) = 1$. 
 
 \smallskip\noindent
Let  $q$ be a prime number such that
 $q \notdiv u\,v$, and such that the order 
$n$ of  $\rho  := \frac{v}{u}$ modulo $q$ is  prime to $p$ and $>2$.

 \smallskip\noindent
Then the symbol $\Big[\Frac{F_*/\Q(\mu_n)}{{\mathfrak q}_*}\Big]_{\rho,n}$
\!\!only depends on $\rho$ and $n$ when $q$ varies
in  $Q_\rho  := \big\{ q , \hbox{ $q \notdiv u\,v\, (u^2-v^2)$}
\hbox{ and the order of $\rho$ modulo $q$ is  prime to $p$} \big\}$.
In other words, the law of $\rho$-decomposition
of $q \in Q_\rho$ for ${\mathcal F}_n$  only depends on  $\rho$ and $n$.
\end{theo}

\begin{proof}
 Let ${\mathfrak Q} \div {\mathfrak q}_\xi$ in $M$,
  where $(\xi,{\mathfrak q}_\xi)$ represents the class cor\-respon\-ding 
to $q$.  The  Frobenius automorphism of ${\mathfrak q}_\xi$
in $F_\xi/L$ is given, by  restriction,  by the relation
$\Big(\Frac{F_\xi/L}{{\mathfrak q}_\xi}\Big)^{f} =
\Big(\Frac{M(\sqrt[p]{\eta_1}\,)/M}{{\mathfrak Q}}
\Big)_{\hbox{$\vert_{F_\xi}$}}$\!. Indeed,
in the projection ${\rm Gal}(M(\sqrt[p]{\eta_1}\,)/M)
 \too {\rm Gal}(F_\xi/L)$, the Frobenius of the prime ideal ${\mathfrak 
Q}$ gives the Artin symbol
of the norm in $M/L$ of ${\mathfrak Q}$, which is ${\mathfrak q}_\xi^f$; 
hence the result.

\smallskip\smallskip
If  $\kappa \not \equiv 0 \pmod {p}$, using 
the relation $f\,\kappa^{-1} \equiv - \ov\kappa^{-1} \pmod {p}$
(see Definition 2, (i)) we get from Lemma 9 that
 $\Big(\Frac{F_\xi/L}{{\mathfrak q}_\xi}\Big)^{-\ov 
\kappa^{-1}}\!\!\!\! = \Big(\Frac{M(\sqrt[p]{\eta_1}\,)/M}
{{\mathfrak Q}}\Big)^{\kappa^{-1}}_{\hbox{$\vert_{F_\xi}$}}$ 
only depends on $\rho$ and $n$ when $q$ varies.
 This proves the theorem in this case since $-\ov\kappa \equiv 
 \frac{1}{p}\,{\rm log}(q) \not\equiv 0 \pmod {p}$
 (see Definition 5).
 
 \smallskip
 If  $\kappa  \equiv 0 \pmod {p}$, we get
 $\Big(\Frac{F_\xi/L}{{\mathfrak q}_\xi}\Big)=1$ in any case.
\end{proof}

\begin{rema}{\rm We can justify the expression ``only depends on $\rho$
and $n$ when $q$ varies in $Q_\rho$'' in the following way.

\smallskip
    Let $\ov F_n := L_1F_n$, where $L_1K = M(\sqrt[p]\zeta\,)$,
    and let $\ov\varphi_\xi:= \Big(\Frac{\ov F_n/L}{{\mathfrak 
    q}_\xi}\Big)$; we know that $\ov\varphi_\xi$ projects on 
$\varphi_\xi$ in $F_\xi/L$ and on $\varphi_1:= 
\Big(\Frac{L_1/L}{{\mathfrak q}_\xi}\Big)$ in $L_1/L$.
We treat the case $\kappa  \not\equiv 0 \pmod {p}$,
i.e., $\varphi_1 \ne 1$.

\smallskip
In the same manner as in the proof of the theorem,
in the projection ${\rm Gal}(M(\sqrt[p]{\zeta}\,)/M)
\too {\rm Gal}(L_1/L)$,
we obtain that:
$$\Big(\Frac{L_1/L}{{\mathfrak q}_\xi}\Big)^{\frac{p}{{\rm log}(q)}} =
\Big(\Frac{M(\sqrt[p]{\zeta}\,)/M}{{\mathfrak Q}}
\Big)^{\kappa^{-1}}_{\hbox{$\vert_{L_1}$}}$$
is independent of $q$ because of the equality 
$\Big(\Frac{M(\sqrt[p]{\zeta}\,)/M}{{\mathfrak Q}}\Big)^{\kappa^{-1}}
\!\!\! .\,\sqrt[p]{\zeta}  = \zeta\, .\,\sqrt[p]{\zeta}$.

\smallskip
Moreover, this is independent of the choice of $\xi$ (of order $n$) 
since for all $t\in {\rm Gal}(L/\Q)$, $\ov\varphi_{\xi^t} =
t\,\ov\varphi_\xi\, t^{-1}$ projects, in $L_1/L$, on 
$\ov\varphi_{\xi^t}{}_{{ \vert}_{L_1}} = 
t\, \ov\varphi_\xi{}_{{ \vert}_{L_1}} t^{-1} = t\,\varphi_1\,t^{-1}
= \varphi_1$ since ${\rm Gal}(L_1/\Q)$ is abelian.

\smallskip
Which justifies the normalization and the fact that, in some sense,
under the existence of a nontrivial solution to the SFLT equation, the symbol
 $\Big[\Frac{F_*/L}{{\mathfrak q}_*}\Big]_{\rho,n}$ does not depend on $q$
 but only on $\rho$ and $n$ (of course $n$ depends on $q$, but not in 
 a deep arithmetical manner).}~\hfill\fin
\end{rema}

From Theorem 3, when the condition ${\mathfrak q}_{\xi}^{1-c} = 
{\mathfrak a}^p\,(1 + p\,\beta_\xi)$ is satisfied, for an ideal 
${\mathfrak a}$ of $L$ and $\beta_\xi$ $p$-integer of $L$,
then $\Big(\Frac{\eta_1}{{\mathfrak Q}}\Big)_{\!\!M} =
\zeta_{}^{\frac{1}{2}f\,{\rm Tr}_{L/\Q}\big(\frac{\beta_\xi}{\xi+1}\big)}$,
where ${\rm Tr}_{L/\Q}$ is the absolute trace in $L/\Q$.
So with  a counterexample to SFLT we must have:
$${\rm Tr}_{L/\Q}\big(\Frac{\beta_\xi}{\xi+1}\big) \equiv f^{-1}\, 
\Frac{v-u}{v+u} \,\kappa \equiv \Frac{v-u}{v+u}\,\Frac{{\rm log}(q)}{p}
\pmod {p}$$ (nonspecial cases, $p\geq 3$)
or 
$${\rm Tr}_{L/\Q}\big(\Frac{\beta_\xi}{\xi+1}\big)\equiv 0 \pmod {p}$$
(special case, $p>3$).

\smallskip
This means that, under a nontrivial counterexample to SFLT:
$$\Big(\Frac{F_\xi/L} { {\mathfrak q}_\xi}  \Big)^{\frac{p}{{\rm log}(q)} }
\ \ {\rm and}\ \  \Frac{p}{{\rm log}(q)}\,{\rm 
Tr}_{L/\Q}\Big(\Frac{\beta_\xi}{\xi+1}\Big), \ {\rm if}\  \kappa \not \equiv 0 \pmod {p},$$
$$\Big(\Frac{F_\xi/L} { {\mathfrak q}_\xi}  \Big)
\ \ {\rm and}\ \  {\rm Tr}_{L/\Q}\Big(\Frac{\beta_\xi}{\xi+1}\Big), 
 \ {\rm if}\  \kappa \equiv 0 \pmod {p},$$ 
both equivalent to the knowledge of $\Big[\Frac{F_*/L}{{\mathfrak q}_*} \Big]_{\rho,n}$,
only depend on $\rho$  and $n$ for prime numbers $q \in Q_\rho$.

\smallskip 
So we can hope that this fact, summarized in Theorem 4,
is incompatible  with the arithmetic of the cyclotomic 
fields $\Q(\mu_n)$ for $p>3$.

\begin{rema}{\rm In the context of Fermat${}'$s equation with 
$r=\frac{y}{x}$, $r'=\frac{y}{z}$, or $r''=\frac{x}{z}$
(supposed of orders $n$, $n'$, $n''$ modulo $q$, prime to $p$),
we have  the same conclusion as in Lemma 9
by using the units  $\eta_1$, $\eta'_1$, and $\eta''_1$; from the 
relation $x+y+z \equiv 0 \pmod {p}$, the values $r'$, 
$r''$ can be computed $\pmod {p}$ from $r$,\,\footnote{\,The notations
$r$, $r'$, and $r''$ correspond to $\rho = \frac{v}{u}$
 in the equation $(u+v\,\zeta)\,\Z[\zeta] = {\mathfrak w}_1^p$ or
 ${\mathfrak p}\,{\mathfrak w}_1^p$, for $(u,v) = (x,y)$, $(y,x)$, 
 $(z,y)$, or $(y,z)$, and $(u,v) = (x,z)$ or $(z,x)$ (nonspecial cases and special 
 case,  respectively); this explains the changes of notations in the Fermat 
 context. We obtain easily $r' \equiv \frac{-r}{r+1}$,
 $r'' \equiv \frac{-1}{r+1}$ modulo $p$.}
 and we get the following relations valid for $p\geq 3$ 
 since in Fermat${}'$s equation, the special case corresponds to $x+z 
 \equiv 0 \pmod {9}$.
 
 \smallskip
(i) If $\kappa \not\equiv 0 \pmod {p}$, then:
\begin{eqnarray*}
&&\Big(\Frac{M(\sqrt[p]{\eta_1}\,)/M}{{\mathfrak Q}}\Big)^{\!\kappa^{-1}}
\!\!\! \,\cdot\,\sqrt[p]{\eta_1}
  = \zeta^{\frac{1}{2}\frac{r-1}{r+1}} \,\cdot\,\sqrt[p]{\eta_1},  \\
&&\Big(\Frac{M(\sqrt[p]{\eta'_1}\,)/M}{{\mathfrak Q}'}\Big)^{\!\kappa^{-1}}
\!\!\! \,\cdot\,\hbox{$\sqrt[p]{\eta'_1}$}
  = \zeta^{-\frac{1}{2}-r} \,\cdot\,\hbox{$\sqrt[p]{\eta'_1}$},  \\
&&\Big(\Frac{M(\sqrt[p]{\eta''_1}\,)/M}{{\mathfrak Q}''}\Big)^{\!\kappa^{-1}}
\!\!\! \,\cdot\,\hbox{$\sqrt[p]{\eta''_1}$}
  = \zeta^{-\frac{1}{2}-\frac{1}{r}} \,\cdot\,\hbox{$\sqrt[p]{\eta''_1}$}, 
  \hbox{ if $r \not\equiv 0 \pmod {p}$, } \\
&&\Big(\Frac{M(\sqrt[p]{\eta''_1}\,)/M}{{\mathfrak Q}''}\Big)^{\!\kappa^{-1}}
\!\!\! \,\cdot\,\hbox{$\sqrt[p]{\eta''_1}$}
 = \hbox{$\sqrt[p]{\eta''_1}$},  \hbox{ if $r \equiv 0 \pmod {p}$. }
 \end{eqnarray*}
 
(ii)  If $\kappa \equiv 0 \pmod {p}$,  the three symbols 
 $\Big(\Frac{M(\sqrt[p]{\ \bullet\ }\,)/M}{\bullet}\Big)$
 are trivial.}~\hfill\fin
 \end{rema}
 
\subsection{ Law of $\rho$-decomposition relative to
the family $\wh{\mathcal F}_n$, for $n>2$}

We still suppose $p>3$.
We have, under a counterexample $(u,v)$ to SFLT,
the following interpretation of the equality:
$$\hbox{ $\Big(\Frac{\eta_1}{{\mathfrak Q}_\xi}\Big)_{\!\!M} =
\zeta^{\frac{1}{2}\,\frac{v-u}{v+u} \,\kappa }$ 
\Big(resp. $\Big(\Frac{\eta_1}{{\mathfrak Q}_\xi}\Big)_{\!\!M} = 
1$\Big)} $$
in the nonspecial cases $v+u \not\equiv 0$
(resp. the special case 
$v+u \equiv 0$) $\!\!\pmod {p}$, which is also valid in the 
cases $\kappa \equiv 0$ or $u-v \equiv 0 \pmod {p}$.
This will give also another formulation of Theorem 4.

\smallskip
Consider the unit:
$$\wh \eta_1 :=
\eta_1\,\zeta^{-\frac{1}{2}\,\frac{v-u}{v+u} } \ \, ({\rm resp.} \ \,
\wh \eta_1 :=\eta_1) $$
in the nonspecial cases (resp. in the special case).

\smallskip
(i) In the nonspecial cases we have: 
$$\wh \eta_1 = (1+\xi\,\zeta)^{e_\omega} 
\,\zeta^{-\frac{1}{2}-\frac{1}{2}\,\frac{v-u}{v+u}} =
(1+\xi\,\zeta)^{e_\omega}\, \zeta^{-\frac{v}{v+u}}, $$
which is by construction such that
$\Big(\Frac{\wh \eta_1}{{\mathfrak Q}_\xi}\Big)_{\!\!M} = 1$,
but the unit $\wh \eta_1$ is not anymore real; its definition 
from $\eta_1$ is independent of $q$ under a given solution  of the SFLT
equation.

\smallskip
(ii) In the special case we obtain
 $\wh \eta_1 :=\eta_1 = (1 + 
 \xi\,\zeta)^{e_\omega}\,\zeta^{-\frac{1}{2}}$,
 which is real and by construction such that
$\Big(\Frac{\wh \eta_1}{{\mathfrak Q}_\xi}\Big)_{\!\!M} = 1$.

\medskip
The extension $M(\sqrt[p]{\wh \eta_1\,})/M$ is splitted over $L$ by a 
$p$-cyclic $p$-ramified extension $\wh F_\xi$ similar to $F_\xi$ 
except that it is not diedral over $L^+$ in the nonspecial cases.

\smallskip
We note that the relation $\wh \eta_1 = 
\eta_1\,\zeta^{-\frac{1}{2}\,\frac{v-u}{v+u}}$ in the nonspecial 
cases shows  that $\wh F_\xi$ is a subfield of the compositum $F_\xi L_1$ obtained 
in an obvious systematic way ($\wh F_\xi/L$ is still
of degree $p$ and $p$-ramified since $n>2$).
But $\wh F_\xi$ is effective only if $\rho$
is known, which is not in general
the case in the nonspecial cases of the SFLT problem.

\smallskip
It is clear that $\wh F_\xi = F_\xi$ if and only if 
$u^2 - v^2\equiv 0 \pmod {p}$.

\smallskip
We still have $t\,\wh F_\xi = \wh F_{\xi^t}$.
We call  $\wh F_n$ the compositum of the $\wh F_{\xi^t}$,
$t \in {\rm Gal}(L/\Q)$. Hence  $F_n\,L_1 =\wh F_n\,L_1$.

\smallskip
We denote, as in Definition 5, by $\wh {\mathcal F}_n$
the family $(\wh F_{\xi'})_{\xi'\, {\rm of\, order\ } n}$.

\smallskip
Then under a nontrivial solution of the equation attached to SFLT, we must have the 
splitting of ${\mathfrak q}_\xi$ in $\wh F_\xi$ (i.e., a 
$\rho$-splitting for $\wh {\mathcal F}_n$).
In other words if we  define,
as in Definition 5, for $\kappa \not\equiv 0 \pmod {p}$, the symbol: 
$$\Big[\Frac{\wh F_*/L}{{\mathfrak q}_*}\Big]_{\rho,n} := 
\Big(\Big(\Frac{\wh F_{\xi^t}/L}{{\mathfrak q}_{\xi^t}}\Big)^{\frac{p}{{\rm log}(q)}}
\Big)_{t \in {\rm Gal}(L/\Q)}\, ,$$
the analog of Theorem 4 is
$\Big[\Frac{\wh F_*/L}{{\mathfrak q}_*}\Big]_{\rho,n} = 1$ for all
$q \in Q_\rho$, where:
$$Q_\rho:= \big\{ q , \hbox{ $q \notdiv u\,v\, (u^2-v^2)$}
\hbox{ and the order of $\rho$ modulo $q$ is  prime to $p$} 
\big\}. $$

A contradiction would be that there exist prime numbers $q$
such that $\Big[\Frac{\wh F_*/L}{{\mathfrak q}_*}\Big]_{\rho,n} \ne 1$
i.e., ${\mathfrak q}_\xi$ is inert in $\wh F_\xi$, which is 
independent of the representative pair $(\wh F_{\xi^t}, {\mathfrak q}_{\xi^t})$
and has a probability very near from 
$\frac{p-1}{p}$ since $p-1$ values of the symbol are possible.
About the class of pairs $(\wh F_{\xi^t}, {\mathfrak q}_{\xi^t})$,
when $\Big[\Frac{\wh F_*/L}{{\mathfrak q}_*}\Big]_{\rho,n} \ne 1$,
we can speak of ``\,$\rho$-inertia of $q$ for $\wh {\mathcal F}_n$\,''.

\smallskip 
In a similar way, in the context of Fermat${}'$s equation,
we deduce from the units $\eta_1$,
$\eta'_1$, and $\eta''_1$ (see Remark 10), the units, where
$r := \frac{y}{x} \not\equiv \pm 1 \pmod {p}$:

\vspace{-0.6cm} 

 \begin{eqnarray*}
&&\wh \eta_1 := (1 + \xi\,\zeta)^{e_\omega}\,\zeta^\frac{-r}{r+1},  \\
&&\wh \eta'_1 := (1 + \xi'\,\zeta)^{e_\omega}\,\zeta^r,  \\
&&\wh \eta''_1 := (1 + \xi''\,\zeta)^{e_\omega}\,\zeta^\frac{1}{r}, 
  \hbox{ if $r \not\equiv 0 \pmod {p}$, }\\
&&\wh \eta''_1 := (1 + \xi''\,\zeta)^{e_\omega}\,\zeta^{-\frac{1}{2}} ,
  \hbox{ if $r \equiv 0 \pmod {p}$, }
 \end{eqnarray*} 

\noindent
giving a trivial $p$th power residue symbol at ${\mathfrak Q}$,
${\mathfrak Q}'$, and ${\mathfrak Q}''$, respectively.
    
\smallskip 
We have the
same conclusion as above for the extensions $\wh F_{\xi}/L$,
$\wh F_{\xi'}/L'$,  $\wh F_{\xi''}/L''$
defined from $M \big(\sqrt[p]{\wh \eta_1}\,\big)\big/M$,
$M' \big(\sqrt[p]{\wh \eta'_1}\,\big)\big/M'$, 
$M''\big(\sqrt[p]{\wh \eta''_1}\,\big)\big/M''$.

\medskip
Returning to SFLT with a nontrivial solution $(u,v)$,
we put as above:
$$\wh Q_\rho^{\rm in} := \big\{q \in Q_\rho,\ q  \hbox{ has a 
$\rho$-inertia for  $\wh {\mathcal F}_n$} \big\}. $$

\begin{lemm} Suppose $p>3$ and $\kappa \not\equiv 0 \pmod {p}$.
If $u^2 - v^2 \not\equiv 0 \pmod {p}$
then we have  $Q_\rho^{\rm spl} \subseteq \wh Q_\rho^{\rm in}$.
If  $u^2 - v^2\equiv 0 \pmod {p}$ then
$Q_\rho^{\rm spl}\cap \wh Q_\rho^{\rm in} = \ev$.
\end{lemm}

\begin{proof} We know that $\wh F_{\xi}$ is contained in the 
compositum $L_1 F_{\xi}$ and is distinct from $L_1$
since $\xi \ne \pm 1$.

\smallskip
Suppose that $\wh F_{\xi}$ is distinct from $F_{\xi}$;
 if $q \in Q_\rho^{\rm spl}$, ${\mathfrak q}_\xi$ splits in
$F_{\xi}/L$ and the Frobenius of ${\mathfrak q}_\xi$
in $L_1F_{\xi}/L$ fixes $F_{\xi}$
 and  since this Frobenius must be nontrivial in $L_1/L$ ($\kappa \not\equiv
 0 \pmod {p}$) then projects to a nontrivial Frobenius in $\wh F_{\xi}/L$.
 When $\wh F_{\xi} = F_{\xi}$, the result is clear.
 The  lemma comes from the  characterization of the equality
 $\wh F_{\xi} = F_{\xi}$ (i.e., $u^2 - v^2 \equiv 0 \pmod {p}$).
\end{proof}

It will be interesting to examine the problem, for any $\rho$,
independently of any equation giving exceptional values of $\rho$.

\smallskip
The natural conjecture in this direction
would be the following, which implies SFLT (we still put 
$K = \Q(\mu_p)$, $L = \Q(\mu_n)$, $M=LK$, to simplify the notations):

\begin{conj}  Let $p$ be a prime number, $p>3$, and let 
$\rho = \frac{v}{u}$, with g.c.d.\,$(u,v) = 1$,
be a  rational distinct from $0$ and $\pm 1$.
Put:
$$Q_\rho := \big\{\hbox{$q$,  $q \notdiv u\,v\,(u^2-v^2)$
and the order of $\rho$ modulo $q$ is  prime to $p$} \big\}. $$

\noindent
For $q\in Q_\rho$,
let $n$ be the order of $\rho$ modulo $q$ and let $\wh {\mathcal F}_n$
be the family of  the  cyclic extensions $\wh F_{\xi}$ of $L$, for $\xi$ of order $n$,
 defined by the relations
$\wh F_\xi K = M \Big(\!\sqrt[p] {(1+\xi\,\zeta)^{e_\omega} 
\,\hbox{$\zeta^{-\frac{v}{v+u}}$}\,}\,\Big)$
\big(resp. $\wh F_\xi K = M \Big(\!\sqrt[p]
{(1+\xi\,\zeta)^{e_\omega}\, \zeta^{-\frac{1}{2}}\,}\,\Big)$
if $v+u \not\equiv 0$ (resp. $v+u \equiv 0$\big) $\!\!\pmod {p}$.
Say that $q$ has a $\rho$-inertia for $\wh {\mathcal F}_n$ if
$\Big[\Frac{\wh F_*/L}{{\mathfrak q}_*}\Big]_{\rho,n} \ne 1$,
i.e., ${\mathfrak q}_\xi := (q, u\,\xi-v)$ is inert in $\wh F_{\xi}/L$
(condition independent of the choice of $\xi$ of order $n$).

\smallskip\noindent
Then the set of primes $q\in Q_\rho$ having a $\rho$-inertia
for  $\wh {\mathcal F}_n$, is infinite.~\hfill\fin
\end{conj}

The extension $\wh F_n$ (depending on $\rho$ contrary to $F_n$) is a subfield of
the maximal $p$-ramified $p$-elementary extension  $H_{L}{\st [p]}$
of $L$, and the arithmetical properties of
$H_{L}{\st [p]}$ and of its subextensions
are, a priori, independent of any diophantine problem
as Fermat${}'$s equation. 

\smallskip
Recall that to prove the first case of FLT for $p$,Ê
the existence of a single $q\in Q_\rho$ ($\rho=\frac{y}{x}$ for a
solution $(x,y,z)$) having a $\rho$-inertia for $\wh {\mathcal F}_n$ is sufficient,Ê
contrary to the second case  which needs infinitely many such primes.

\smallskip
In the first case, Ê$p\notdiv xy\,(x^2-y^2)$  (from Lemma 1)Ê
and so, if $\kappa\not\equiv 0 \pmod {p}$ then  $q\notdiv xy\,(x^2-y^2)$
from the two theorems of Furtw\"angler (see Corollaries 2, 3, and Remark 3).

\smallskip
Hence $q \in Q_\rho$ as soon as $\kappa\not\equiv 0 \pmod {p}$
and $q \not\equiv 1 \pmod {p}$.
These two conditions on $q$ are effective
and the first case of ÊFLT is easier than the second one
because it is generally possible to Êcheck Êthe conjecture
for Êsmall values of $q$. The second case supposes to find $q$ 
large enough; this shows that the first case is likely a
weaker conjecture.

\smallskip
If we examine, for logical reasons, the case $p=3$ for SFLT, we know 
that for any of the six 
families of solutions $(u,v)$ of the SFLT equation (see Remark 1), 
 we have (supposing $\kappa \not\equiv 0 \pmod {3}$ and defining $\wh \eta_1$ 
in an analogous way): 

\smallskip
(i) $\Big(\Frac{\eta_1}{{\mathfrak Q}_\xi}\Big)_{\!\!M} =
\zeta^{\frac{1}{2}\,\frac{v-u}{v+u}\,\kappa } = 1$,
in the first  case (i.e., $u\,v\,(u+v) \not \equiv 0 \pmod {3}$), which
implies $u-v \equiv 0  \pmod {3}$, hence
$\wh \eta_1 = \eta_1$ and  $\wh F_\xi =  F_\xi$;

\smallskip
(ii) $\Big(\Frac{\eta_1}{{\mathfrak Q}_\xi}\Big)_{\!\!M} =
\zeta^{\pm \frac{1}{2}\,\kappa}$ in the second
case (i.e., $u\,v \equiv 0 \pmod {3}$), thus 
$\wh \eta_1 = \eta_1\,\zeta^{\mp \frac{1}{2}}$ and  $\wh F_\xi \ne  F_\xi$;

\smallskip
(iii) $\Big(\Frac{\eta_1}{{\mathfrak Q}_\xi}\Big)_{\!\!M} =
\zeta^{\frac{1}{2}\,\frac{v+u}{3\,v}\,\kappa }$ in the special case
(i.e., $u+v \equiv 0 \pmod {3}$) for which 
$\wh \eta_1 = \eta_1\,\zeta^{-\frac{1}{2}\,\frac{v+u}{3\,v}}$ and
$\wh F_\xi =  F_\xi$ if and only if $v + u\equiv 0 \pmod {9}$.

\medskip
 If $v+u \equiv 0 \pmod {3}$ and
$v+u \not\equiv 0 \pmod {9}$ then,  for $\rho:=\frac{v}{u}$, we get
$Q_\rho^{\rm spl} \subseteq \wh Q_\rho^{\rm in}$;
if  $v+u \equiv 0 \pmod {9}$ or $u-v \equiv 0 \pmod {3}$
then $Q_\rho^{\rm spl}\cap \wh Q_\rho^{\rm in} = \ev$.

\smallskip
We see that 
$u-v \equiv 0  \pmod {3}$ in case (i), $u\,v\equiv 0\pmod {3}$
in case (ii); for (iii), we verify from Remark 1 that $\frac{v}{u}\in
\{-1, 2, 5  \}$ modulo 9, which gives $\frac{1}{2}\frac{v+u}{3\,v}\in
\{0, 1, 2 \}$ modulo 3. 
So, for $q$ fixed we can find solutions 
$(u_i,v_i)$ giving the same order $n$ of $\frac{v_i}{u_i}$
modulo $q$ and any of the above value of $\frac{v}{u}$ modulo 9.

\smallskip
See  Section 9 to go thoroughly into the exceptional
case $p=3$.

\subsection{Construction of universal defining polynomials}

The group $g$ operates canonically on the field $K(Y)$ of rational 
fractions, where $Y$ is an indeterminate.
Consider: 
$$F(Y):= (1+Y\,\zeta)^{e_\omega} \,\zeta^{-\frac{1}{2}} \in K(Y). $$
Then if $s = s_r$ is a generator of $g$ we have:
$$s . F(Y):= \big((1+Y\,\zeta^s) \,\zeta^{-\frac{1}{2}s}\big)^{e_\omega} =
\big((1+Y\,\zeta) \,\zeta^{-\frac{1}{2}}\big)^{s\,e_\omega}\!\! =
\big((1+Y\,\zeta) \,\zeta^{-\frac{1}{2}}\big)^{re_\omega+p\Lambda}, $$
since $s\,e_\omega \!\! = r\,e_\omega+p\,\Lambda$ for some $\Lambda 
\in \Z[g]$ (see Definition 1, (iii)). Then we obtain:
$$s . F(Y) = F(Y)^{r}\,\cdot\,
\big((1+Y\,\zeta) \,\zeta^{-\frac{1}{2}}\big)^{p\Lambda}. $$

\smallskip
Consider the Kummer extension $K(Y)(\sqrt[p]{F(Y)}\,)/K(Y)$; since 
this extension 
is abelian over $\Q(Y)$, the $K(Y)$-automorphism  of $K(Y)(\sqrt[p]{F(Y)}\,)$, still 
denoted $s$, defined by
$s\cdot\sqrt[p]{F(Y)} := (\sqrt[p]{F(Y)}\,)^r\,\cdot\,
\big((1+Y\,\zeta) \,\zeta^{-\frac{1}{2}}\big)^{\Lambda}$ is of order
$p-1$ and it is a classical result that the trace
$\Psi := \sm_{k=1}^{p-1}s^k \cdot\,\sqrt[p]{F(Y)}$ defines a primitive 
element
of the subextension cyclic of degree $p$ contained in 
$K(Y)(\sqrt[p]{F(Y)}\,)/\Q(Y)$.

\smallskip
An easy way to find ${\rm Irr}(\Psi,\Q(Y) )$ is to use the Newton 
formulas from the computations of the traces:
$${\rm Tr} \big((\sqrt[p]{F(Y)}\,)^i\big) := 
\sm_{k=1}^{p-1}s^k \cdot\,(\sqrt[p]{F(Y)}\,)^i,\ \,i = 1,\ldots,p-1. $$

For instance, for $p=3$, $e_\omega = s-1$, $s = s_2$,
$s\,e_\omega = 1-s = - e_\omega$ (thus $r=2$, $\Lambda =- e_\omega$),
 $F(Y) = (1+Y\,j)^{e_\omega} j =
((1+Y\,j) \,j)^{s-1}$; we have
$\Psi =\Big( \Frac{(1+Y\,j^2) \,j}{1+Y\,j}\Big)^\frac{1}{3}
 + \Big( \Frac{(1+Y\,j) \,j^2}{1+Y\,j^2}\Big)^\frac{1}{3}$,
for which we get $\Psi^3 =
\Frac{(1+Y\,j^2) \,j}{1+Y\,j}+
\Frac{(1+Y\,j) \,j^2}{1+Y\,j^2} + 3\,\Psi$,
giving the irreducible polynomial:
$${\rm Irr}(\Psi,\Q(Y)) = X^3 - 3\,X + \Frac{Y^2-4Y+1}{Y^2-Y+1}, $$
or the unitary polynomial $X^3 - 3\, (Y^2-Y+1)^2\,X + (Y^2-4Y+1)(Y^2-Y+1)$
taking the representative idempotent $e_\omega = s+2$.
 
\begin{defi} {\rm 
The general case of degree $p$ can be written:
$${\rm Irr}(\Psi,\Q(Y)) = A_p(Y) X^p+\cdots+ A_1(Y) X + A_0(Y), \ 
\,A_i(Y) \in \Z[Y], $$
and will be called a universal polynomial of degree $p$
for the SFLT problem.

\smallskip
For any given $n$th root of unity $\xi$, $n>2$, the polynomial:
$$A_p(\xi) X^p+\cdots+ A_1(\xi) X + A_0(\xi) \in L[X],\ \, L:= 
\Q(\mu_n), $$
is the irreducible polynomial of the primitive element:
$$\psi := {\rm Tr}_{M(\sqrt[p]{\eta_1\,})/F_\xi}(\sqrt[p]{\eta_1})$$
defining the extension $F_\xi$, with the usual notations.~\hfill\fin}
\end{defi}

\smallskip
We  have the following result where we recall that, for 
g.c.d.\,$(u,v)=1$, we have put
$\Phi_n(u,v) := \prd_{\xi'\,{\rm of\, order}\ n} (u\,\xi' - v)$.
 
\begin{theo}  Let $p$ be a prime number, $p>3$, and let 
$\rho = \frac{v}{u}$, with g.c.d.\,$(u,v) = 1$, be a fixed
rational distinct from $0$ and $\pm 1$; suppose
$u-v \not\equiv 0 \pmod {p}$.
 
 \smallskip
 (i) Nonspecial cases ($u+v \not\equiv 0 \pmod {p}$).
Let $n>2$ be prime to $p$, and
  let $q \notdiv n$ be a prime number such
  that $\kappa \not\equiv 0 \pmod {p}$
and  $q \div \Phi_n(u,v)$.

\smallskip\noindent
If the polynomial $A_p(\rho) X^p+\cdots+ A_1(\rho) X + A_0(\rho)$
is not irreducible modulo $q$, then
$(u,v)$ cannot be a solution  of the  equation
$(u+v \,\zeta)\,\Z[\zeta] = {\mathfrak w}_1^p$ attached to
the nonspecial cases of SFLT.

\smallskip
 (ii) Special case ($v+u \equiv 0 \pmod {p}$).
Let $n>2$ be prime to $p$, and
  let $q \notdiv n$ be a prime number such that
  $\kappa \not\equiv 0 \pmod {p}$
and $q \div \Phi_n(u,v)$.

  \smallskip\noindent
  If the polynomial $A_p(\rho) X^p+\cdots+ A_1(\rho) X + A_0(\rho)$
is  irreducible modulo~$q$, then
$(u,v)$ cannot be a solution  of the  equation
$(u+v \,\zeta)\,\Z[\zeta] ={\mathfrak p}\,{\mathfrak w}_1^p$ attached to
the special case of SFLT.

\smallskip
 (iii) Let $n>2$ be prime to $p$, and
  let $q \notdiv n$ be a prime number such that
$\kappa \equiv 0 \pmod {p}$ and  $q \div \Phi_n(u,v)$.

  \smallskip\noindent
  If the polynomial $A_p(\rho) X^p+\cdots+ A_1(\rho) X + A_0(\rho)$
is  irreducible modulo $q$, then
$(u,v)$ cannot be a solution  of the SFLT equation.
\end{theo}

\begin{proof}  From Lemma 2, $q \notdiv n$ and $q \div \Phi_n(u,v)$ 
 is equivalent to $q \notdiv u\,v$ and  $\rho$ is of order $n$
modulo $q$; then $\rho \equiv \xi \pmod {{\mathfrak q}_\xi}$, for any
choice of the $n$th root of unity $\xi$, and in case (i) there exists
a root  $\lambda \in \Z$   modulo $q$, of the polynomial, such that:
 \begin{eqnarray*}
A_p (\rho) \lambda^p+\cdots+ A_1 (\rho) \lambda + 
A_0 (\rho) &\equiv&  A_p(\xi) \lambda^p+\cdots+ A_1(\xi) \lambda + 
A_0(\xi)\\
&\equiv& 0 \pmod {{\mathfrak q}_\xi}, 
 \end{eqnarray*}
since $q$ divides the left member. This means that
${\rm Irr}(\psi,L)$ has the root $\lambda$ modulo ${\mathfrak q}_\xi$
and that ${\mathfrak q}_\xi$ splits in $F_\xi/L$
(i.e., $\Big[\Frac{F_*/L}{{\mathfrak q}_*}\Big]_{\rho,n} = 1$).

\smallskip
If we suppose that $(u,v)$ is a counterexample to SFLT,
 Theorem 1 in the nonspecial cases gives
$\Big(\Frac{\eta_1}{{\mathfrak 
Q}}\Big)_{\!\!M}=\zeta^{\frac{1}{2}\,\frac{v-u}{v+u}\,\kappa}
\ne 1$ by assumption, equivalent to the inertia of ${\mathfrak q}_\xi$ in $F_\xi/L$
(contradiction). 
 
 \smallskip
The proofs of cases (ii) and (iii) are similar but inverted (the hypothesis 
implies $\Big[\Frac{F_*/L}{{\mathfrak q}_*}\Big]_{\rho,n} \ne 1$ while
$\Big(\Frac{\eta_1}{{\mathfrak Q}}\Big)_{\!\!M}=1$ for a solution in 
these  cases).
\end{proof}
    
In other words, the corresponding conjecture giving a proof of SFLT
under the assumption $u-v \not\equiv 0 \pmod {p}$,
which implies the two cases of FLT, is the following.

\begin{conj}  Let $p$ be a prime number, $p>3$, and let 
$\rho = \frac{v}{u}$, with g.c.d.\,$(u,v) = 1$, be a 
rational distinct from $0$ and $\pm 1$; suppose
$u-v \not\equiv 0 \pmod {p}$.
    
  (i) Case $u+v \not\equiv 0 \pmod {p}$.
 There exist infinitely many prime numbers~$q$ with
$\kappa \not\equiv 0 \pmod {p}$ such that 
$A_p(\rho) X^p+\cdots+ A_1(\rho) X + A_0(\rho)$
is not  irreducible modulo $q$ and  $q\div\Phi_n(u,v)$ 
for suitable values of $n>2$ prime to~$p$.

 (ii) Case $v+u \equiv 0 \pmod {p}$.
 There exist infinitely many prime numbers $q$
  with $\kappa \not\equiv 0 \pmod {p}$
 such that $A_p(\rho) X^p+\cdots+ A_1(\rho) X + A_0(\rho)$
is irreducible modulo $q$ and  $q\div\Phi_n(u,v)$
for suitable values of $n>2$ prime to $p$.

\smallskip
 (iii)  There exist infinitely many prime numbers $q$ with $\kappa 
 \equiv 0 \pmod {p}$
 such that $A_p(\rho) X^p+\cdots+ A_1(\rho) X + A_0(\rho)$
is irreducible modulo $q$ and  $q\div\Phi_n(u,v)$
for suitable values of $n>2$ prime to $p$.~\hfill\fin
\end{conj}

Of course, without an independent approach (analytic or geometric),
the problem has no longer solution since
the polynomial:
$$A_p(\rho) X^p+\cdots+ A_1(\rho) X + A_0(\rho) $$
can be in case (i) that of a primitive element of $L_1$,
in which case all the primes which split in $L_1/\Q$ are
such that $\kappa \equiv 0 \pmod {p}$, and in cases (ii) and (iii), the 
polynomial may be splitted over $\Q$.
Meanwhile the universal polynomial $A_p(Y) X^p+\cdots+ A_1(Y) X + A_0(Y)$
has the nontrivial property that for any 
primitive $n$th root of unity $\xi$, $n>2$,
$A_p(\xi) X^p+\cdots + A_1(\xi) X + A_0(\xi)$
is irreducible in $\Q(\mu_n)[X]$ and defines a $p$-ramified cyclic 
extension  of $\Q(\mu_n)$.

\section {Normic relations for cyclotomic units} 

In this section we give a relation between the units $\eta_1$
and $\eta'_1$ associated to the classes of two prime numbers $q$ and 
$q'$ for which the pairs $(\xi,{\mathfrak q}_{\xi})$, 
$(\xi',{\mathfrak q}_{\xi'})$ are such that the order $n'$ of
$\xi'$ divides the order $n$ of $\xi$, $p \notdiv n$.

\smallskip
Put $n = n' d$. We introduce the following notations:
 \begin{eqnarray*}
&& L = \Q(\mu_n), \  L' = \Q(\mu_{n'}),\\
&& M = L\,K, \  M' = L'\,K,\\
&& \eta_1 = (1 + \xi\,\zeta)^{e_\omega}\,\zeta^{-\frac{1}{2}},\ 
\eta'_1 = (1 + \xi'\,\zeta)^{e_\omega}\,\zeta^{-\frac{1}{2}};\ 
 \end{eqnarray*} 
 to fix the notations, we suppose that $\xi' = \xi^d$.

\smallskip
Since $\eta_1$ is a cyclotomic unit, the action of 
relative norms on this unit is well-known and we now recall the result in 
our particular context.

\begin{prop} Denote by $\No$ the relative norm $\No_{M/M'}$ and by $S$ 
the set of distinct prime numbers dividing $d$ and not dividing $n'$.

 \smallskip\noindent
Then we have $\No(\eta_1) = (\eta'_1)^{\Lambda'}$,  where
${\Lambda'} \equiv d \ .\, \prd_{\ell\in S}(1 - \ell^{-1}\,t'_\ell{}^{-1}) \pmod {p}$,
$t'_\ell \in {\rm Gal}(M'/K)$ being the Artin automorphism defined 
by $t'_\ell(\xi') := (\xi')^\ell$.
\end{prop}

\begin{proof} By induction we can suppose that $d$ is a prime 
number $\ell$.

\smallskip
Let $\psi := \xi^{n'}$ which is a primitive $\ell$th root of unity.

\smallskip
(i) Case $\ell  \div {n'}$. In this case $S=\ev$, $[M:M']= \ell$, and:
 \begin{eqnarray*}
\No(1 + \xi\,\zeta) &=& 
\prd_{\lambda = 0}^{\ell - 1} (1 + \xi^{1+\lambda {n'}}\zeta) =
\prd_{\lambda = 0}^{\ell - 1} (1 + \xi\, \psi^{\lambda}\zeta)\\
&=& 1 + \xi^\ell\,\zeta^\ell = 1 + \xi'\,\zeta^\ell = (1 + \xi'\,\zeta)^{s_\ell}. 
   \end{eqnarray*} 

Then $\No(\eta_1) =  (1 + \xi'\,\zeta)^{s_\ell e_\omega }
\, \No(\zeta)^{-\frac{1}{2}} \sim 
(1 + \xi'\,\zeta)^{\ell\, e_\omega} \,\zeta^{-\frac{1}{2}\ell} = (\eta'_1)^\ell$
 since $s_\ell e_\omega \equiv \ell e_\omega 
\pmod {p}$.

\smallskip
(ii) Case $\ell \notdiv {n'}$. In this case $S = \{\ell\}$ and:
$$\No(1 + \xi\,\zeta) =
\prd_{\lambda = 0,\,\lambda \ne \lambda_0}^{\ell - 1}
(1 + \xi^{1+\lambda {n'}}\zeta), $$
where $\lambda_0$ is the unique value 
modulo $\ell$ such that $1+\lambda_0 {n'} \equiv 0 \pmod {\ell}$,
giving from the computation in (i):
$$\No(1 + \xi\,\zeta) = 
\Frac{1 + \xi^\ell\, \zeta^\ell}{1 + \xi^{1+\lambda_0 {n'}}\, \zeta} =
\Frac{(1 + \xi'\, \zeta)^{s_\ell}}{1 + (\xi')^\mu\, \zeta}, $$
where $1+\lambda_0 {n'} = \mu\,\ell$, so 
that $\mu \equiv \ell^{-1} \pmod {n'}$.
Thus:
\begin{eqnarray*}
\No(1 + \xi\,\zeta) &=& 
\Frac{(1 + \xi'\, \zeta)^{s_\ell}}{1 + (\xi')^{\ell^{-1}} \zeta} =
\Frac{(1 + \xi'\, \zeta)^{s_\ell}}{1 + (\xi')^{t'_\ell{}^{-1}} \zeta} \\
&=&
\Big( \Frac{1 + \xi'\, \zeta}{1 + (\xi')^{t'_\ell{}^{-1}}\,(\zeta)^{s_\ell^{-1}}}\Big)^{s_\ell}=
\Big( \Frac{1 + \xi'\, \zeta}{1 + (\xi'\, \zeta)^{\sigma'_\ell{}^{-1}}}\Big)^{s_\ell}, 
 \end{eqnarray*} 
where $\sigma'_\ell \in {\rm Gal}(M'/\Q)$ is the Artin 
automorphism defined by $\sigma'_\ell(\theta)= \theta^\ell$ for any
$p{n'}$th root of unity $\theta$; thus, since  $\sigma'_\ell = s_\ell\,t'_\ell$, 
this yields:
$$\No(1 + \xi\,\zeta)^{e_\omega} \sim 
\Big( \Frac{1 + \xi'\, \zeta}{1 + (\xi'\, 
\zeta)^{\sigma'_\ell{}^{-1}}}\Big)^{\ell\,e_\omega} 
= \big (1 + \xi'\, \zeta \big)^{\ell\, (1- 
\sigma'_\ell{}^{-1})\,e_\omega}; $$
 from $\sigma'_\ell{}^{-1}\,e_\omega = s_\ell^{-1}\,t'_\ell{}^{-1}\,e_\omega \equiv
\ell^{-1}\,t'_\ell{}^{-1}\,e_\omega \pmod {p}$,
we get $\No(1 + \xi\,\zeta)^{e_\omega} \sim  \big (1 + \xi'\, \zeta 
\big)^{\ell\,(1 - \ell^{-1}t'_\ell{}^{-1})\,e_\omega}$.
Finally, since  in this case $[M:M'] = \ell-1$
and $\No(\zeta) =  \zeta^{\ell - 1}= \zeta^{\ell\,(1 - 
\ell^{-1}\,t'_\ell{}^{-1})}$, we get:
$$\No(\eta_1) \sim  (\eta'_1)^{\ell\,(1 - \ell^{-1}\,t'_\ell{}^{-1})}$$
 and the proposition follows.
\end{proof}

If for instance ${\Lambda'}$ is invertible modulo $p$, with inverse $\Omega'$,
then $\eta'_1 \sim  \No(\eta_1)^{\Omega'}$ and, over $L$, we 
can see the abelian extension $F'_{n'}$ (compositum of the conjugates of 
the $F'_{\xi'}$ over $L'$) as a subfield of $F_n$, in which case, for suitable 
primes $q$ and $q'$, the properties of the Frobenii studied in this 
paper can be compared to give strengthened conditions. 

\section {Analysis of the case $p=3$ versus $p\ne 3$}

In this section we suppose $p=3$ and consider the solutions of the 
equation associated to SFLT (see Remark 1) which are, for $p>3$,
a logical obstruction to the relevance of general statements
similar to Theorem 2 and to the property of
$\rho$-law of decomposition of Theorem 4.
We intend to explain why this obstruction 
 actually exists for $p=3$ but a priori not for $p>3$ when we suppose that the set 
of nontrivial  solutions is nonempty.

\smallskip
The main differences between the cases $p=3$ and $p>3$,
are that there are infinitely many solutions for the case $p=3$,
contrary to the case $p>3$, even if we have not proved this fact (which was known
for Fermat${}'$s equation before Wiles  proof), and that we will
exhibit a group of automorphisms, acting on the set of solutions for 
$p=3$, which creates some exceptional relations of compatibility with 
density theorems. So we conjecture that this fact
does not exist for $p>3$.

\subsection {Analysis of the case $p=3$ for the principle of Theorem 2}

Recall  Theorem 1 in that case, for the choice of a prime $q \ne p$.
We have $\eta_1 = (1+\xi\,\zeta)^{e_\omega}\,\zeta^{-\frac{1}{2}}$, with
$\zeta = j$ and $e_\omega = s-1$, where $\xi\equiv \frac{v}{u}
\pmod {{\mathfrak q}_\xi}$ is supposed of order $n \not\equiv 0 \pmod {3}$
for a nontrivial solution $(u,v)$, g.c.d.\,$(u,v)=1$, of the SFLT equation.

\smallskip
 Recall that if we put $\rho := \frac{v}{u}$ (distinct from 0 and $\pm 1$) we may have
$u \equiv 0 \pmod {3}$ in which case $\rho$ is not defined modulo $3$, but
is always defined as a rational. Put $L=\Q(\mu_n)$ and $M = LK$:

\smallskip
(i) First case. Since $u\,v\,(u+v) \not\equiv 0 \pmod {3}$, we get 
$u \equiv v\equiv\pm 1  \pmod {3}$, so that 
$\Big(\Frac{\eta_1}{{\mathfrak Q}}\Big)_{\!\!M} =
j^{\frac{1}{2}\, \frac{v-u}{v+u}\,\kappa} =1$
for any ${\mathfrak Q} \div {\mathfrak q}_\xi$ in $M$.

(ii) Second case. We get $\Big(\Frac{\eta_1}{{\mathfrak Q}}\Big)_{\!\!M}
= j^{\pm\frac{1}{2}\,\kappa}$ for any ${\mathfrak Q} \div {\mathfrak q}_\xi$
since $3 \div u\, v$.

\smallskip
(iii) Special case. Then  $\Big(\Frac{\eta_1}{{\mathfrak Q}}\Big)_{\!\!M}
= j^{\frac{1}{2}\,\frac{v+u}{3\,v}\,\kappa }$ for any ${\mathfrak Q} \div {\mathfrak q}_\xi$
with $3 \div  v+u$; we have seen, at the end of  Subsection 7.2,
that $\frac{v+u}{3\,v}$ can take any value modulo~3.

\smallskip
From this, we see that the existence of $q$ 
totally split in $H_L^-{\st [3]}/\Q$ for $L= \Q(\mu_{q-1})$,
or at least $L=\Q(\mu_m)$ for a large $m  \div q-1$, may be
in contradiction with the existence of the solutions of the second  and 
special cases when  $\kappa \not\equiv 0 \pmod {3}$,
i.e., 3 inert in $\Q_1/\Q$ where $\Q_1 = \Q(\mu_9)^+$.

\begin{defi} {\rm Consider the field $k(Y)$, where $k$ is any field of
characteristic distinct from $2$ and $3$, and the automorphism:

\smallskip\smallskip
\centerline{ $\begin{array}{rcl}
T : k(Y) & \tooo & \, k(Y) \,\, . \\ [3pt]
Y \ \ \  & \longmapsto & \  \frac{2\,Y-1}{Y+1}
\end{array}$ }

Let $F(Y):= (1+Y\,\zeta)^{e_\omega} \,\zeta^{-\frac{1}{2}} \in K(Y)$
be the formal cyclotomic unit, yet defined in Subsection 7.3.~\hfill\fin }
\end{defi}

We intend to prove below various properties of compatibility,
of this automorphism, with the method of cyclotomic units developed here.

\begin{theo}
(i) The automorphism $T$ is of order 6 and we have the following
orbit of $Y$:
$$\begin{array}{lllllllll}
&T^{\ }(Y)=\Frac{2\,Y-1}{Y+1};& T^2(Y)=\Frac{Y-1}{Y};& 
T^3(Y)=\Frac{Y-2}{2\,Y-1}; \\[5pt]
&T^4(Y)=\Frac{-1}{Y-1};& T^5(Y)=\Frac{-Y-1}{Y-2};& T^6(Y)=Y.
\end{array}$$

(ii) We have for $\zeta = j$ of order 3 and for $F(Y) = (1+Y j)^{e_\omega} j$,
the following formulas (equalities up to $3$th powers in $K(Y)$):
$$T(F(Y)) = (1+Y j)^{e_\omega}\,;\ 
T^2(F(Y)) = (1+Y j)^{e_\omega} j^2\,;\ T^3(F(Y)) = F(Y) , $$
summarized by the identity 
$T^i(F(Y)) \sim F(Y)\,j^{\frac{1}{2}\,i}$, $0 \leq i < 3$.
\end{theo}

\begin{proof} We have:
    $$T(F(Y))\! =\!\big (1+\Frac{2\,Y-1}{Y+1} j\big)^{e_\omega}j
   \! = \!( Y+1+(2\,Y-1)j)^{e_\omega} j\! =\! (1-j+(2j+1)Y)^{e_\omega}j; $$
since $2j+1 =  j \, (1-j)$, we get finally
$T(F(Y)) = (1-j)^{e_\omega} (1+Yj)^{e_\omega} j$;
but $(1-j)^{e_\omega} = -j^2$, hence the result in this case.
The other computations are obtained by induction.
\end{proof}

We apply now the automorphism $T$ to the solutions $(u,v)$ in the 
following way. We put $T(\frac{v}{u}) =: \frac{V}{U}$ where
$(U,V)$ is defined up to the sign.
We start for instance from the solution:
$$(u,v) = (-s^3 - t^3 + 3s^2 t, -s^3-t^3+3s t^2) $$
 (see Remark 1) to determine its orbit.

\begin{theo} We obtain  the following identities:
$$\begin{array}{lllllllll}
T^{0 }\big(\Frac{v}{u}\big) &\!\! = \!\!& \Frac{v}{u} &\!\!\! = \!\!&
\Frac{ -s^3-t^3+3s t^2}{-s^3 - t^3 + 3s^2 t}\,\virg \\ [6pt]
T^{1}\big(\Frac{v}{u}\big) &\!\! = \!\!& \Frac{2v-u}{v+u} &\!\!\! = \!\!&
\Frac{ -s^3-t^3 -3 s^2 t +6 s t^2}{-2s^3 -2t^3+3s^2 t +3s t^2}\,\virg \\ [6pt]
T^2\big(\Frac{v}{u}\big) &\!\! = \!\!&\Frac{v-u}{v} &\!\!\! = \!\!&
\Frac{ 3 s^2 t -3s t^2}{s^3+t^3-3s t^2} \,\virg  \\ [6pt]
T^3\big(\Frac{v}{u}\big) &\!\! = \!\!& \Frac{v-2u}{2v-u} &\!\!\! = \!\!&
\Frac{ -s^3-t^3+6s^2 t -3s t^2}{s^3+t^3+3s^2 t -6s t^2}\,\virg  \\ [6pt]
T^4\big(\Frac{v}{u}\big) &\!\! = \!\!& \Frac{-u}{v-u} &\!\!\! = \!\!&
\Frac{ s^3+t^3 - 3 s^2 t}{3s t^2 -3s^2 t}\,\virg  \\ [6pt]
T^5\big(\Frac{v}{u}\big) &\!\! = \!\!& \Frac{-v-u}{v-2u}&\!\!\! = \!\!&
\Frac{ 2s^3 + 2t^3 -3 s^2 t -3s t^2}{s^3+t^3 - 6 s^2 t + 3s t^2}\, ,
\end{array}$$

\noindent
which leads to the six fundamental families of solutions of the SFLT equation for 
$p=3$.~\hfill\fin
\end{theo}

\begin{rema} {\rm For $q \not \equiv 1 \pmod {3}$, $q\ne 2$, all the orbits in 
$\F_q\cup \{\infty\}$ have six 
elements (indeed, all the equations of the form $\Frac{ay+b}{cy+d} = y$,
deduced from the rational fractions of Theorem 6, (i),
reduce to $y^2-y+1$ which is
irreducible over $\F_q$). We remark the orbit of 0 which is:
$$ 0\to -1\to \infty\to 2\to 1 \to \Frac{1}{2}\ \, {\rm in} \ \, \F_q\cup 
\{\infty\}; $$
this is consistent with $\vert\,\F_q\cup \{\infty\}\,\vert = q+1
 \equiv 0 \pmod {6}$.~\hfill\fin}
\end{rema}

Let $q$ be a prime number; to simplify we suppose $q\not\equiv 1 
\pmod {3}$. Call $n_i  \div  q-1$ the orders modulo $q$ of 
$T^i\big(\Frac{v}{u}\big)$, $0 \leq i < 6$, for any solution
$(u,v)$.

\smallskip
As usual we put $\Frac{v}{u} \equiv \xi \pmod {{\mathfrak q}_{\xi}
= (q,u\,\xi - v)}$
and more generally:
$$T^i\big(\Frac{v}{u}\big) =: \Frac{v_i}{u_i} \equiv \xi_i
\pmod {{\mathfrak q}_{\xi_i} = (q,u_i\,\xi_i - v_i)},\ 
\, 0 \leq i < 6, $$
where we recall that the pair $(\xi_i, {\mathfrak q}_{\xi_i})$ is 
defined up to conjugation, so that we can replace $(\xi_i,{\mathfrak q}_{\xi_i})$ by any 
conjugate $(\xi'_i,{\mathfrak q}_{\xi'_i})$.

\smallskip
Thus we have put $(u_0,v_0) := (u,v)$ and $\xi_0 := \xi$.

\smallskip
Consider for instance $T\big(\Frac{v}{u}\big) = \Frac{v_1}{u_1} \equiv 
\xi_1 \pmod {{\mathfrak q}_{\xi_1}}$ noting that 
$\Frac{v}{u}  \equiv \xi \pmod {{\mathfrak q}_{\xi}}$.

To compare the two  congruences
we can take a fixed prime ideal $\wt {\mathfrak q} \div q$
in $\wt L := \Q(\mu_{q-1})$ such 
that $\wt {\mathfrak q} \div {\mathfrak q}_{\xi}$ and
$\wt {\mathfrak q} \div {\mathfrak q}_{\xi_1}$ 
by suitable conjugation of $(\xi_1, {\mathfrak q}_{\xi_1})$, which gives the 
congruences $\Frac{v}{u}  \equiv \xi \pmod {\wt{\mathfrak q}}$
and $\Frac{v_1}{u_1}  \equiv \xi_1 \pmod {\wt{\mathfrak q}}$, hence
$\xi_1 \equiv \Frac{v_1}{u_1} = T\big(\Frac{v}{u}\big) \equiv
T(\xi) \pmod {\wt {\mathfrak q}}$.
More generally we can write for suitable choices of the $\xi_i$: 
$$\xi_i \equiv T^i(\xi) \pmod {\wt {\mathfrak q}},\ \, 0 \leq i < 6 ,$$
which yields, for the units  $\eta_1^i$ associated to the $\xi_i$
(with $\eta_1^0 = \eta_1$):
\begin{eqnarray*}
\eta_1^i& :=& (1 + \xi_i\,j)^{e_\omega} j \\
&  \equiv & (1 + T^i(\xi)\,j)^{e_\omega} j \\
& \equiv & \eta_1\,j^{\frac{1}{2}\,i} \ 
\pmod {\wt {\mathfrak Q}},\ \,0 \leq i < 3,
\end{eqnarray*}
(from Theorem 6, (ii)),
for all $\wt{\mathfrak Q}$ above $\wt{\mathfrak q}$ in
$\wt M := \wt L K$.  Thus we have:
$$\Big(\Frac{\eta_1^i}{\wt{\mathfrak Q}}\Big)_{\!\!\wt M} =
\Big(\Frac{\eta_1}{\wt{\mathfrak Q}}\Big)_{\!\!\wt M}
\Big(\Frac{j^{\frac{1}{2}\,i}}{\wt{\mathfrak Q}}\Big)_{\!\!\wt M} =
\Big(\Frac{\eta_1}{\wt{\mathfrak Q}}\Big)_{\!\!\wt M}\,j^{\frac{1}{2}\, i\,\kappa}
\ \,{\rm for\ all\  \wt{\mathfrak Q}\div \wt{\mathfrak q}},
\ \,0 \leq i < 3, $$
proving that the three symbols never 
coincide when $\kappa \not \equiv 0 \pmod {3}$.

\smallskip
These symbols are identical to the 
symbols $\Big(\Frac{\eta_1^i}{{\mathfrak Q}_{\xi_i}}\Big)_{\!\! M^i}$, for 
all ${\mathfrak Q}_{\xi_i} \div  {\mathfrak q}_{\xi_i}$,
$0 \leq i < 3$, where $M^i=L^i K$, with $L^i = 
\Q(\mu_{n_i})$.\,\footnote{\,The coherent choice of these ideals 
supposes that if $\wt{\mathfrak q} = (q,\wt\xi-\wt e)$ ($\wt\xi$
of order $q-1$, $\wt e \in \Z$ of order $q-1$ modulo $q$),
$\frac{v_i}{u_i} \equiv \wt e^{d_i} \pmod {q}$ (of order $n_i$
modulo $q$), we must have chosen $\xi_i = \wt\xi^{d_i}$ so that 
${\mathfrak q}_{\xi_i} = (q,\xi_i - \wt e^{d_i}) =
(q,\wt\xi^{d_i} - \wt e^{d_i}) \equiv 0 \pmod {\wt{\mathfrak q}}$,
$0\leq i < 3$.}

\smallskip
This proves that if for instance ${\mathfrak q}_{\xi_0}$
splits in $F_{\xi_0}/L^0$ then  ${\mathfrak q}_{\xi_1}$ and
${\mathfrak q}_{\xi_2}$ are inert in $F_{\xi_1}/L^1$ and 
$F_{\xi_2}/L^2$, respectively; in other words, the three 
law of $\rho_i$-decomposition, or the three symbols
$\Big[\Frac{F_*/L^i}{{\mathfrak q}_*}\Big]_{\rho_i,n_i}$ 
of  Definition 7.2, yields to the three possibilities
when $\kappa \not\equiv 0 \pmod {p}$.

\smallskip
So, since this phenomenon happens in $\wt L = \Q(\mu_{q-1})$
(even if the fields $L^i$ are distincts in general),
statements like that of Theorem 6.1 are impossible for $p=3$
since $\wt{\mathfrak q}$ cannot be totally split in $F_{q-1}$.

\smallskip
This distribution of the three possible Frobenii, in the context of 
$\rho_i$-decompositions, must be compatible with the 
\v Cebotarev${}'$s theorem (see Subsection 9.2 for this aspect and Subsection
9.3 for some numerical evidence and especially Example 5).

\smallskip
Returning to the general case, it is necessary to
see whether such a nontrivial automorphism $T$
can exist for $p>3$. If not, this will be a favorable 
argument for our purpose.

\begin{theo} {\it Consider ${\mathcal M} := \Z_p \tensorZ K(Y)^\times$,
 as a multiplicative $\Z_p[g]$-module,
and  the idempotent $\varepsilon_\omega := \sum_{s \in g} \omega^{-1}(s)\, s \in \Z_p[g]$
(see Definition~1, (i) and (ii)).

\smallskip\noindent
Then, for $p>3$, there does not exist any
automorphism $T$ of $\Q(Y)$,
distinct from the identity and the inversion $Y \mapsto  Y^{-1}$,
such that $T(1+Y\,\zeta) := 1+T(Y)\,\zeta$ be such that:
$$\big(1+T(Y)\,\zeta\big)^{\varepsilon_\omega} =
\big ( 1+Y\,\zeta^\lambda \big)^{\varepsilon_\omega}\,\zeta^\mu , $$
up to a $p$th power in ${\mathcal M}$,
for some $\lambda, \mu \in \Z$, $\lambda \not\equiv 0 \pmod {p}$.   }
\end{theo}

\begin{proof}
Suppose that such a nontrivial auto\-morphism does exist and
put $T(Y) = \Frac{aY+b}{cY+d}$ with  $a,b,c,d \in \Q$, $ad-bc \ne 0$.
 Note that the associated matrix:
$$\hbox{ $M=$ \small{$\left( 
\begin{array}{ccccccc}
a & b  \\
c & d  
\end{array} \right)$}  }$$
\normalsize
is considered in ${\rm G}\ell_2(\Q)/D$, where $D$ is the subgroup of 
diagonal matrices $e\,I_2$, $e \in \Q^\times$, where
$I_2$ is the unit matrix.
In particular, $T$ is of finite order if and only if there 
exists $n > 0$ such that $M^n = e\,I_2$. For instance, $M=$
\small{$\left( 
\begin{array}{ccccccc}
2 & -1  \\
1 & \ \  1  
\end{array}\right)$}
\normalsize
is such that $M^6 = -27\,I_2$.

\smallskip
    For simplicity we use a representative 
    $e_\omega = \sm_{k=1}^{p-1} u_k\,s_k \in \Z[g]$ of 
    $\varepsilon_\omega$, with  coefficients $u_k$ such that
 $1\leq u_k \leq p-1$, and work in $K(Y)^\times/ K(Y)^{\times p}
 \simeq {\mathcal M}/{\mathcal M}^p$  (so that $e_\omega$ is not here 
 the usual representative; in particular, if $\Omega \in 
  \Q(Y)^{\times}$ then $\Omega^{e_\omega} \in \Q(Y)^{\times p}$).

\smallskip
Then from the above identity we get the relation:
$$(cY+d + (aY+b)\,\zeta)^{e_\omega} =
 (1 + Y\,\zeta^{\lambda})^{e_\omega}\,\zeta^\mu .\, G(Y)^p,\ \, G(Y) 
= \Frac{A(Y)}{B(Y)} \in K(Y)^\times , $$
with  $A, B \in K[Y]$,  g.c.d.\,$(A,B) = 1$,
hence the polynomial identity in $K[Y]$:
$$B(Y)^p\,(cY+d + (aY+b)\,\zeta)^{e_\omega} =
A(Y)^p\, (1 + Y\,\zeta^{\lambda})^{e_\omega}\,\zeta^\mu .$$

Since $(cY+d + (aY+b)\,\zeta)^{e_\omega}$ and 
$(1 + Y\,\zeta^{\lambda})^{e_\omega}$ each have $p-1$
distinct\,\footnote{\,Indeed, for the roots $y_k := 
-\frac{d+b\,\zeta^k}{c+a\,\zeta^k}$, $1 \leq k \leq p-1$, $y_k = 
y_{k'}$ is equivalent to the relation $(ad-bc)\,(\zeta^k - 
\zeta^{k'}) = 0$, hence the result. The other case is trivial.} %
roots of orders of multiplicity $u_k$,
with $1 \leq u_k \leq p-1$, it is clear that 
$(cY+d + (aY+b)\,\zeta)^{e_\omega}$ and 
$(1 + Y\,\zeta^{\lambda})^{e_\omega}$ each are prime to $A$ and $B$,
then  have the same roots; hence there exists $\ell\not 
\equiv 0 \pmod {p}$ such that $\frac{d+b\,\zeta}{c+a\,\zeta} =
\zeta^{-\lambda \ell} =: \zeta^{\psi}$, $\psi \not\equiv 0 \pmod {p}$.
Then $\zeta^{\psi+1}\,a +\zeta^{\psi}\,c - \zeta\,b - d = 0$.

\smallskip
Since $p>3$ we get $\psi \equiv -1$ or $1$ modulo $p$, giving the solutions
$(a,b,c,d) = (1,0,0,1)$ (i.e., the identity), 
$(a,b,c,d) = (0,1,1,0)$ (i.e., the inversion).
\end{proof}
    
\subsection {Analysis of the case $p=3$ for the principle of Theorem 4}

We have now to explain why the phenomenon of $\rho$-law of 
decomposition (Theorem 4, i.e., 
{\small $\Big[\frac{F_*/L}{{\mathfrak q}_*}\Big]_{\rho,n}$} independent of 
$q$ in the sense of Remark 9) is indeed compatible for $p=3$ but 
(conjecturally) not for $p>3$.

\medskip
The following analysis suggests a suitable
property of repartition (in the meaning of \v Cebota\-rev density theorem)
of the values of the Frobenii, due to the infiniteness of the set of
solutions of the SFLT equation for $p=3$ and the fact that this set 
is the union of six families (see Remark 1) having complementary properties
for these values.

\smallskip
Let $q$ be given such that $\kappa \not\equiv 0 \pmod {3}$.
As usual, for the solutions $(u(s,t),v(s,t))$ of the SFLT equation,
 put $\rho := \frac{v}{u}$ and  call $\xi$ any
primitive $n$th root of unity, where $n$ is the order of
$\rho$ modulo $q$, $n$ supposed prime to~$3$. Put
$\eta_1 := (1+ \xi\,j)^{e_\omega} j^{-\frac{1}{2}} =
(1+ \xi\,j)^{s-1} j^{-\frac{1}{2}}$, then
${\mathfrak q} := (q, u\,\xi-v)$, and  denote by ${\mathfrak Q}$
any prime ideal of $M=LK$ above ${\mathfrak q}$.

\smallskip\smallskip
Of course, in this study $n$ is not constant when the solution $(u,v)$ varies,
so that the statistical analysis cannot be 
done for a fixed field $L = \Q(\mu_n) \subseteq \Q(\mu_{q-1})$.

\smallskip
But up to this problem 
(probably not too tricky since the number of divisors $n$ of $q-1$ is 
finite and since it is probably better to work instead in $\Q(\mu_{q-1})$
for this statistical analysis), we 
have the following distribution of the possible cases, in a remarkable
accordance with the definition of the solutions of the SFLT equation,
that we summarize with the diagram of the compositum $L_1 F_\xi$ which is 
very simple for $p=3$ (note that in the general case, $L_1 F_\xi/L$ 
contains $\frac{p^2-1}{p-1} = p+1$ cyclic subextensions of degree $p$).

\smallskip
Indeed, for $p=3$ the compositum $L_1 F_\xi$ contains $L_1$, $F_\xi$,
and two other cubic fields, $F'_\xi$ and its conjugate $c\,F'_\xi$ by the complex 
conjugation $c$ (recall that $F_\xi/L^+$ is diedral, $L_1/L^+$ 
abelian, so that $L_1 F_\xi/L^+$ is Galois). 

\smallskip
Moreover we will get $\wh F_\xi$ among the three extensions 
distinct from $L_1$.

\smallskip
We denote by $\sigma$ a fixed generator of ${\rm Gal}(F_\xi/L)$ and 
call $\varphi_\xi$ the Frobenius of ${\mathfrak q}_\xi$ in $F_\xi/L$.
We refer to Theorem 1 giving the symbol $\Big(\Frac{\eta_1}{\mathfrak Q}\Big)_{\!\!M}$
for $p=3$, where ${\mathfrak Q}\div {\mathfrak q}={\mathfrak q}_\xi$.

\smallskip
(i) First case ($u\,v\,(u+v) \not \equiv 0 \pmod {3}$) corresponding to the relation 
$u+v\,j = j^2\,(s+t\,j)^3$.
%
We have $\Big(\Frac{\eta_1}{\mathfrak Q}\Big)_{\!\!M} = 
j^{\frac{1}{2}\,\frac{v-u}{v+u}\,\kappa} =  1$ since 
$u-v \equiv 0  \pmod {3}$, $\wh F_\xi = F_\xi$, and the diagram:

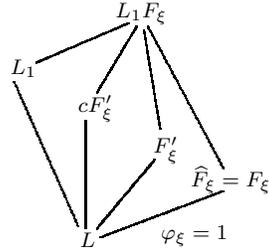
\unitlength=0.5cm
$$\vbox{\hbox{\hspace{-1.cm}
\begin{picture}(9.0,6.4)
\bezier{114}(4.65,5.76)(5.66,3.865)(6.67,1.97)
\bezier{133}(4.55,5.65)(4.765,4.34)(4.98,3.03)
\bezier{103}(4.35,5.75)(3.825,4.87)(3.3,3.99)
\bezier{133}(4.15,5.85)(2.925,5.325)(1.7,4.8)
\bezier{174}(3.4,0.15)(5.03,0.765)(6.66,1.38)
\bezier{121}(3.3,0.38)(4.075,1.305)(4.85,2.23)
\bezier{143}(3.0,0.40)(3.0,1.825)(3.0,3.25)
\bezier{110}(2.8,0.45)(1.95,2.365)(1.1,4.28)
\put(3.8,6.){\Small $L_1 F_\xi$}
\put(5.8,1.5){\Small$\wh F_\xi = F_\xi$}
\put(4.8,2.4){\Small$ F'_\xi$}
\put(2.8,3.5){\Small $c F'_\xi$}
\put(1.,4.5){\Small$L_1$}
\put(2.85,-0.1){\Small$L$}
\put(5.0,0.1){\Small${\varphi_\xi = 1}$}
\end{picture}  }}$$

\smallskip\noindent
in which ${\mathfrak q}$ is inert in $F'_\xi/L$, $c F'_\xi/L$, and 
$L_1/L$.


\smallskip
(ii) Second case ($u\,v \equiv 0 \pmod {3}$) corresponding to the two relations 
$u+v\,j = (s+t\,j)^3$ and
$u+v\,j = j\,(s+t\,j)^3$.
%
We have  $\Big(\Frac{\eta_1}{\mathfrak Q}\Big)_{\!\!M} =
j^{\frac{1}{2}\,\frac{v-u}{v+u}\,\kappa} = 
j^{\pm\frac{1}{2}\,\kappa} = j$ or $j^2$; we get $\wh F_\xi \ne F_\xi$, and the 
two equidistributed diagrams:

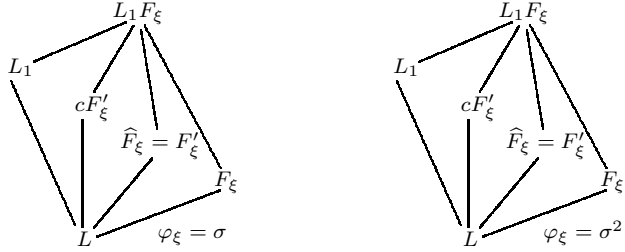
\unitlength=0.5cm
\vspace{-0.3cm}
$$\vbox{\hbox{
\begin{picture}(10.0,6.4)
\bezier{114}(4.65,5.76)(5.66,3.865)(6.67,1.97)
\bezier{133}(4.55,5.65)(4.765,4.34)(4.98,3.03)
\bezier{103}(4.35,5.75)(3.825,4.87)(3.3,3.99)
\bezier{133}(4.15,5.85)(2.925,5.325)(1.7,4.8)
\bezier{174}(3.4,0.15)(5.03,0.765)(6.66,1.38)
\bezier{121}(3.3,0.38)(4.075,1.305)(4.85,2.23)
\bezier{143}(3.0,0.40)(3.0,1.825)(3.0,3.25)
\bezier{110}(2.8,0.45)(1.95,2.365)(1.1,4.28)
\put(3.8,6.){\Small$L_1 F_\xi$}
\put(6.5,1.5){\Small$F_\xi$}
\put(4.0,2.5){\Small$\wh F_\xi = F'_\xi$}
\put(2.8,3.5){\Small $c F'_\xi$}
\put(1.,4.5){\Small$L_1$}
\put(2.85,-0.1){\Small$L$}
\put(5.0,0.1){\Small${\varphi_\xi = \sigma}$}
\end{picture}  
\begin{picture}(10.0,6.4)
\bezier{114}(4.65,5.76)(5.66,3.865)(6.67,1.97)
\bezier{133}(4.55,5.65)(4.765,4.34)(4.98,3.03)
\bezier{103}(4.35,5.75)(3.825,4.87)(3.3,3.99)
\bezier{133}(4.15,5.85)(2.925,5.325)(1.7,4.8)
\bezier{174}(3.4,0.15)(5.03,0.765)(6.66,1.38)
\bezier{121}(3.3,0.38)(4.075,1.305)(4.85,2.23)
\bezier{143}(3.0,0.40)(3.0,1.825)(3.0,3.25)
\bezier{110}(2.8,0.45)(1.95,2.365)(1.1,4.28)
\put(3.8,6.){\Small$L_1 F_\xi$}
\put(6.5,1.5){\Small$F_\xi$}
\put(4.0,2.5){\Small$\wh F_\xi = F'_\xi$}
\put(2.8,3.5){\Small $c F'_\xi$}
\put(1.,4.5){\Small$L_1$}
\put(2.85,-0.1){\Small$L$}
\put(5.0,0.1){\Small${\varphi_\xi = \sigma^2}$}
\end{picture} }} $$

\smallskip\noindent
in which ${\mathfrak q}$ is inert in $F_\xi/L$, $c F'_\xi/L$, and 
$L_1/L$.


\smallskip
(iii) Special case ($u+v \equiv 0 \pmod {3}$) 
corresponding to the three relations 
$u+v\,j = j^h\,(1-j)\,(s+t\,j)^3$, $0 \leq h < 3$.
%
We have  $\Big(\Frac{\eta_1}{\mathfrak Q}\Big)_{\!\!M} =
j^{\frac{1}{2}\,\frac{v+u}{3\,v}\,\kappa} = 1$,
$j$, or $j^2$, and the three equidistributed diagrams:

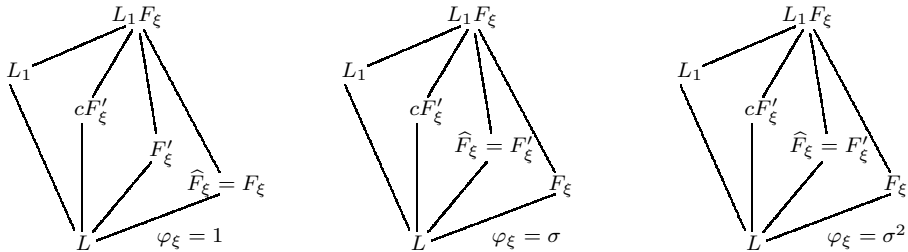
\unitlength=0.5cm
\vspace{-0.1cm}
$$\vbox{\hbox{\hspace{-0.8cm}
\begin{picture}(10.0,6.4)
\bezier{114}(4.65,5.76)(5.66,3.865)(6.67,1.97)
\bezier{133}(4.55,5.65)(4.765,4.34)(4.98,3.03)
\bezier{103}(4.35,5.75)(3.825,4.87)(3.3,3.99)
\bezier{133}(4.15,5.85)(2.925,5.325)(1.7,4.8)
\bezier{174}(3.4,0.15)(5.03,0.765)(6.66,1.38)
\bezier{121}(3.3,0.38)(4.075,1.305)(4.85,2.23)
\bezier{143}(3.0,0.40)(3.0,1.825)(3.0,3.25)
\bezier{110}(2.8,0.45)(1.95,2.365)(1.1,4.28)
\put(3.8,6.){\Small$L_1 F_\xi$}
\put(5.8,1.5){\Small$\wh F_\xi = F_\xi$}
\put(4.8,2.4){\Small$F'_\xi$}
\put(2.8,3.5){\Small $c F'_\xi$}
\put(1.,4.5){\Small$L_1$}
\put(2.85,-0.1){\Small$L$}
\put(5.0,0.1){\Small${\varphi_\xi = 1}$}
\end{picture}  
\hspace{-0.8cm}
\begin{picture}(10.0,6.4)
\bezier{114}(4.65,5.76)(5.66,3.865)(6.67,1.97)
\bezier{133}(4.55,5.65)(4.765,4.34)(4.98,3.03)
\bezier{103}(4.35,5.75)(3.825,4.87)(3.3,3.99)
\bezier{133}(4.15,5.85)(2.925,5.325)(1.7,4.8)
\bezier{174}(3.4,0.15)(5.03,0.765)(6.66,1.38)
\bezier{121}(3.3,0.38)(4.075,1.305)(4.85,2.23)
\bezier{143}(3.0,0.40)(3.0,1.825)(3.0,3.25)
\bezier{110}(2.8,0.45)(1.95,2.365)(1.1,4.28)
\put(3.8,6.){\Small$L_1 F_\xi$}
\put(6.5,1.5){\Small$F_\xi$}
\put(4.0,2.5){\Small$\wh F_\xi = F'_\xi$}
\put(2.8,3.5){\Small $c F'_\xi$}
\put(1.,4.5){\Small$L_1$}
\put(2.85,-0.1){\Small$L$}
\put(5.0,0.1){\Small${\varphi_\xi = \sigma}$}
\end{picture} 
\hspace{-0.8cm}
\begin{picture}(7.,6.2)
\bezier{114}(4.65,5.76)(5.66,3.865)(6.67,1.97)
\bezier{133}(4.55,5.65)(4.765,4.34)(4.98,3.03)
\bezier{103}(4.35,5.75)(3.825,4.87)(3.3,3.99)
\bezier{133}(4.15,5.85)(2.925,5.325)(1.7,4.8)
\bezier{174}(3.4,0.15)(5.03,0.765)(6.66,1.38)
\bezier{121}(3.3,0.38)(4.075,1.305)(4.85,2.23)
\bezier{143}(3.0,0.40)(3.0,1.825)(3.0,3.25)
\bezier{110}(2.8,0.45)(1.95,2.365)(1.1,4.28)
\put(3.8,6.){\Small$L_1 F_\xi$}
\put(6.5,1.5){\Small$F_\xi$}
\put(4.0,2.5){\Small$\wh F_\xi = F'_\xi$}
\put(2.8,3.5){\Small $c F'_\xi$}
\put(1.,4.5){\Small$L_1$}
\put(2.85,-0.1){\Small$L$}
\put(5.0,0.1){\Small${\varphi_\xi = \sigma^2}$}
\end{picture}  }}$$

\smallskip\noindent
in which the decomposition of ${\mathfrak q}$ assembles all the above cases.

\medskip
This suggests that the infiniteness of the solutions of the SFLT 
equation and their particular repartition into six families, is 
a necessary fact for the
compatibility with \v Cebotarev${}'$s density theorem.

\subsection{Numerical data for the case $p=3$}

We give some numerical experimentations, using [PARI], in the case 
$p=3$, to highlight the above properties of this case.

\smallskip
We refer to Remark 1 for the six expressions of
the solutions of the SFLT equation 
for $p=3$; when we speak of  ``\ a solution $(u,v)$\,'', we 
consider {\it one} of the six families
$(u(s,t),v(s,t))$ with parameters $s$ and $t$.

\begin{prop} Let $n>2$ be an integer not divisible by 3 and
 for any integers $u$, $v$ with g.c.d.\,$(u,v) = 1$, let 
$\Phi_n(u,v) := \prd_{\xi'\,{\rm of\, order}\,\,n} (u\,\xi' - v)$.
    
(i)  The set of primes  $q\equiv -1 \pmod {3}$, $q\notdiv n$,
with $\kappa\not\equiv 0 \pmod {3}$,
dividing at least one of the integers $\Phi_n(u,v)$, for a solution 
$(u(s,t),v(s,t))$  of the  SFLT equation,
is infinite when $s, t$ vary in $\Z$ with ${\rm g.c.d.\,}(s,t)=1$,
$s+t \not\equiv 0 \pmod {3}$.

\smallskip\noindent
Then  there exist numbers of the form 
 $\Phi_n(u,v)$  divisible by primes $q$ as large as we need.

\smallskip\noindent
More precisely, the prime number $q$ is solution  if and only if
for an $e\in \Z$, of order $n$ modulo $q$, the polynomial 
$X^3 - 3\,e^{-1}\,X^2 - 3\,(1 - e^{-1})\,X +1$ splits in $\F_q[X]$;
the parameters $(s,t)$ giving the solutions $(u,v)$ such that 
$\frac{v}{u} \equiv e \pmod {q}$, are given via the three roots $\ov\theta_k \in \F_q$
of the polynomial, by the relation $s-t\,\theta_k \equiv 0 
\pmod {q}$, $s, t \in \Z$ satisfying the above conditions, $k = 1,2,3$.

\smallskip
(ii) The condition  $q \div \Phi_n(u,v)$ ($q\notdiv n$),
for a  solution $(u,v)$ of the  SFLT equation, is equivalent to the
$\rho$-splitting of $q$ for $\wh {\mathcal F}_n$ \big(i.e.,
it is equivalent to  $\Big[\Frac{\wh F_*/L}{{\mathfrak 
q}_*}\Big]_{\rho,n} = 1$\big) for $\rho:=\frac{v}{u}$.
\end{prop}

\begin{proof} Let $\xi$ of order $n$ and let $L=\Q(\mu_n)$.
Since ${\rm g.c.d.\,}(s,t)=1$,
this yields immediately ${\rm g.c.d.\,}(u,v)=1$ for any solution,
thus $u$ and $v$ are not divisible by any prime $q$ dividing $\Phi_n(u,v)$
 homogeneous of the
form $u^{\phi(n)} \pm \cdots \pm v^{\phi(n)}$ in coprime integers $u$ and $v$.

\smallskip
From Lemma 2,  $q \notdiv n$ and $q \div \Phi_n(u,v)$ is equivalent to 
the fact that
$\frac{v}{u}$ is of order $n$ modulo $q$, hence it
is equivalent to the congruence
$u\,\xi - v \equiv 0 \pmod {\mathfrak q}$, for a suitable
and unique prime ideal ${\mathfrak q} \div q$ in $L$; 
then ${\mathfrak q} =
(q,u\,\xi-v)$, which depends on $(u,v)$ for $\xi$ fixed, 
is one of the $\phi(n)$ prime ideals above $q$ in $L$; in the 
previous sections it was denoted ${\mathfrak q}_\xi$ for given $u$, $v$.

\smallskip
We will prove that the condition $q\div \Phi_n(u,v)$ ($q\notdiv n$),
for a solution of the SFLT equation, can be tested
independently of the choice of the solution 
among the six possibilities, in the following sense.

\smallskip
Starting from
a parametric solution $(u,v)$ such that $u\,\xi - v \equiv 0
\pmod {\wt{\mathfrak q}}$ for some $\wt{\mathfrak q}\div 
{\mathfrak q}_\xi$ in $\wt L = \Q(\mu_{q-1})$, consider
the solution $(u',v')$ defined by:
$$\Frac{v'}{u'}: = T\big(\Frac{v}{u}\big) = \Frac{2v-u}{v+u}. $$

We have the congruence 
$u'\,\xi' - v' \equiv 0 \pmod {\wt{\mathfrak q}}$ where
 $\xi'$ is the unique $(q-1)$th root of unity congruent to 
$T(\xi) = \frac{2\xi-1}{\xi +1}$ modulo $\wt{\mathfrak q}$
(the order $n'$ of $\xi'$ divides $q-1$). Then we have:
\begin{eqnarray*}
u'\,\xi' - v' &\equiv& (v+u)\,\Frac{2\,\xi-1}{\xi +1} - (2v-u) \\
              &\equiv& \Frac{1}{\xi +1}\,\big((v+u)\,(2\xi-1)
	      - (2v-u)\,(\xi +1)\big)   \\
              &\equiv& \Frac{3}{\xi +1}\,(u\,\xi - v) \pmod {\wt{\mathfrak q}},
\end{eqnarray*}
proving the equivalence of the two congruences.
 Hence the result by induction on the powers of $T$.
From Theorem 7, the six families of 
solutions give the congruences
$u_i\,\xi_i - v_i \equiv 0 \pmod {\wt{\mathfrak q}}$
for which $\frac{v_i}{u_i}: = T^i\big(\frac{v}{u}\big)$,
$\xi_i \equiv T^i(\xi) \pmod {\wt{\mathfrak q}}$; each congruence
reduces to a congruence modulo  ${\mathfrak q}_{\xi_i}$ in
$L^i := \Q(\mu_{n_i})$, where  ${\mathfrak q}_{\xi_i} =
\wt{\mathfrak q} \cap Z_{L^i}$ and $n_i$ is the order of $\xi_i$
(prime to 3 since $q\equiv -1 \pmod {3}$).

\smallskip
Warning: the orders $n_i$ are  divisors  of $q-1$, 
not necessarily equal to $n$ (see Example 5). But the conditions
$q \notdiv n_i$ and $q \div \Phi_{n_i}(u,v)$, $0 \leq i <6$,
are equivalent to each other.

\smallskip
For instance, take the general solution of the second case $3\div v$; then we 
have to study the congruence
$(s^3+t^3 -3\,st^2)\,\xi - 3\,st(s-t) \equiv 0 \pmod {\mathfrak q}$.

Put $\theta := \Frac{s}{t}$, which yields to the congruence:
$$\theta^3 - 3 \xi^{-1}\,\theta^2 - 3 (1-\xi^{-1})\,\theta + 1
\equiv 0 \pmod {\mathfrak q}.$$

Recall that for $n$ fixed, the $\phi(n)$ ideals of $L$ above $q$
are the $(q, \xi - e)$, where $e\in \Z$, defined modulo $q$, is of 
order $n$ in $\F_q^\times$; so the congruence:
$$\theta^3 - 3 \xi^{-1}\,\theta^2 - 3 (1-\xi^{-1})\,\theta + 1
\equiv 0 \pmod {{\mathfrak q} = (q, \xi - e)} $$
is equivalent to:
$$\theta^3 - 3 e^{-1}\,\theta^2 - 3 (1-e^{-1})\,\theta + 1
\equiv 0 \pmod {q}$$
for the choice of $e \equiv \xi\pmod {\mathfrak q}$.
Since the pair $(\xi,{\mathfrak q})$ is defined up to
conjugation, we can select $e$ of order $n$, which 
implies suitable $\xi$ and ${\mathfrak q}$.

\smallskip
When  $q$ is solution, there exist infinitely many $(u,v)$
such that  $q \div \Phi_n(u,v)$: for a root $\ov \theta\in \F_q$,
$\theta \in\Z$, of the above congruence, the parameters $(s,t)$ are obtained 
from the congruence $s \equiv \theta \,t \pmod {q}$ (see  Example 8).
At this step we have proved (i) under the existence of $e$ such that 
the polynomial  $X^3 - 3 e^{-1} \,X^2 - 3 (1-e^{-1})\,X + 1$
splits in $\F_q[X]$.

\smallskip
The polynomial $X^3 - 3 \xi^{-1} \,X^2 - 3 (1-\xi^{-1})\,X + 1$
defines the cyclic extension 
$\wh F_\xi$: indeed, with $X=\xi^{-1}(Y+1)$ one obtains the polynomial:
$$Y^3 - 3(\xi^2-\xi+1)\,Y - (2-\xi)(\xi^2-\xi+1)$$
from the universal polynomial, irreducible of degree 3  over $L$ (see 
Subsection 7.3), obtained from the cubic root of $(1 + \xi \,j)^{s+2} = \wh \eta_1$
up to a $3$th power.

\smallskip
Thus, the condition  $q \div \Phi_n(u,v)$ ($q\notdiv n$) is equivalent to the
$\rho$-splitting of $q$ for $\wh {\mathcal F}_n$, where $\rho:=\frac{v}{u}$
(see  Subsection 7.2) or to the $\rho_i$-splitting of $q$ for
$\wh {\mathcal F}_{n_i}$ where $\rho_i:=\frac{v_i}{u_i} = T^i(\frac{v}{u})$,
and $n_i$ is the order modulo $q$ of $\rho_i$, 
$0\leq i < 6$. This proves (ii).

\smallskip
From the Dirichlet--\v Cebotarev theorem, we get a precise
result taking a nontrivial Frobenius in $L_1\wh F_\xi/\wh F_\xi$,
and we obtain  the prime ideal ${\mathfrak q}_\xi=
(q,u\,\xi - v)$ where the $(u,v)$ are obtained from the three roots 
$\ov \theta_1$, $\ov \theta_2$, $\ov \theta_3$ of
the polynomial as explained above.
We obtain infinitely many values
of $q$ with clearly a nonzero density.
In other words, for $\kappa \not\equiv 0 \pmod {3}$
these primes $q$ give again the splitting of 
${\mathfrak q}_\xi$ in $\wh F_\xi/L$, hence its inertia
in $L_1/L$, $F_\xi/L$, and in the fourth cubic subfield $\wh F'_\xi/L$ of the 
compositum $ L_1\wh F_\xi$ (note that $\wh F'_\xi = c\,\wh F_\xi$
and that $c\,F_\xi = F_\xi$ is diedral over $L^+$). This makes clear the 
point (i) of the proposition.
\end{proof}    

In the case where $(u,v)$ is for instance the general solution for the second case 
of the SFLT equation we get, with $\eta_1 =
(1+\xi \,j)^{e_\omega}\,j$, $\wh \eta_1 := \eta_1\,j^{\frac{1}{2}}
= (1+\xi \,j)^{e_\omega}$, and $c\,\wh \eta_1 := \eta_1 j$,
the following diagram:

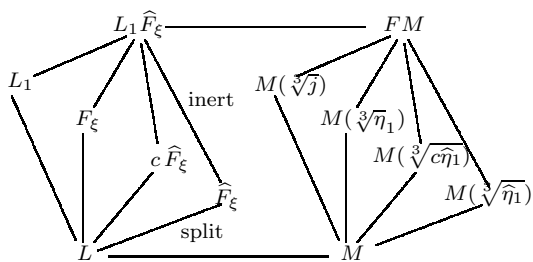
\unitlength=0.5cm
$$\vbox{\hbox{\hspace{-4.8cm}
\begin{picture}(10.0,6.4)
\bezier{200}(11.65,5.76)(12.725,3.8)(13.8,1.84)
\bezier{100}(11.55,5.65)(11.765,4.34)(11.98,3.03)
\bezier{103}(11.35,5.75)(10.825,4.87)(10.3,3.99)
\bezier{133}(11.15,5.85)(9.925,5.325)(8.7,4.8)
\bezier{150}(10.8,0.3)(12.2,0.83)(13.6,1.36)
\bezier{121}(10.3,0.38)(11.075,1.305)(11.85,2.23)
\bezier{143}(10.0,0.40)(10.0,1.825)(10.0,3.25)
\bezier{210}(9.8,0.45)(8.95,2.365)(8.1,4.28)
\put(11.,6.){\Small$FM$}
\put(12.6,1.5){\Small$M(\sqrt[3]{\wh \eta_1})$}
\put(10.7,2.5){\Small$M(\sqrt[3]{c\wh \eta_1})$}
\put(9.3,3.5){\Small $M(\sqrt[3]\eta_1)$}
\put(7.55,4.4){\Small$M(\sqrt[3]j)$}
\put(9.85,-0.1){\Small$M$}
\bezier{300}(3.7,0.)(7,0.)(9.5,0.)
\bezier{300}(5.2,6.1)(7,6.1)(10.5,6.1)
\bezier{214}(4.65,5.76)(5.66,3.865)(6.67,1.97)
\bezier{133}(4.55,5.65)(4.765,4.34)(4.98,3.03)
\bezier{103}(4.35,5.75)(3.825,4.87)(3.3,3.99)
\bezier{133}(4.15,5.85)(2.925,5.325)(1.7,4.8)
\bezier{174}(3.4,0.15)(5.03,0.765)(6.66,1.38)
\bezier{121}(3.3,0.38)(4.075,1.305)(4.85,2.23)
\bezier{143}(3.0,0.40)(3.0,1.825)(3.0,3.25)
\bezier{210}(2.8,0.45)(1.95,2.365)(1.1,4.28)
\put(3.8,6.){\Small$L_1\wh F_\xi$}
\put(6.5,1.4){\Small$\wh F_\xi$}
\put(4.8,2.4){\Small$c\,\wh F_\xi$}
\put(2.8,3.5){\Small $F_\xi$}
\put(1.,4.5){\Small$L_1$}
\put(2.85,-0.1){\Small$L$}
\put(5.6,0.5){\Small split}
\put(5.8,4.){\Small inert}
\end{picture}  }}$$

\medskip
There are six analogous diagrams over each field $L^i$.

\begin{rema} {\rm Let $q$ be a prime number such that $\kappa \not\equiv 
0 \pmod {3}$. Then for a divisor $m>2$ of $q-1$,
there is not necessarily a solution $(u,v) = (s^3+t^3 -3\,st^2,\,3\,st(s-t))$,
$s, t \in \Z$, ${\rm g.c.d.\,}(s,t)=1$, $s+t \not\equiv 0 \pmod {3}$,
such that the order $n$ of $\frac{v}{u}$ modulo $q$ is equal to $m$
(see  Example 6).

\smallskip
 The cases $m \leq 2$ correspond, for the above solution, to the congruences:
$$s^3+t^3 -3\,st^2 \pm 3\,st(s-t)) \equiv 0 \pmod {q}, $$ equivalent
to the splitting, modulo $q$, of $X^3+1 -3\,X \pm 3\,X(X-1))$.
One verifies that these  polynomials of $\Q[X]$
define the number field $\Q_1$; so, as by assumption $\kappa\not\equiv 0 
\pmod {3}$, we obtain that the orders 1 and 2 are never possible. The 
case $m >2$ is less trivial.~\hfill\fin }
\end{rema}

\begin{ex} {\rm We illustrate Proposition 5 with the prime $q=41$
    and the solution $(u,v) = (139193, 76626)$ obtained with the 
    parameters $(s,t) = (-11,43)$; we note that for
    $e = \ov {22} \in \Z/41\Z$  the polynomial:
    $$X^3 - 3 e^{-1}\,X^2 - 3 (1-e^{-1})\,X + 1$$
    splits  in $\Z/41\Z[X]$
    into $(X -\ov {38})\,(X-\ov {31})\,(X-\ov {15})$ and we have chosen
    $\ov\theta =\ov {15}$ for which $s - 15\,t \equiv 0 \pmod {41}$.
    Using the automorphism $T$, we obtain the six steps:{\small
\begin{eqnarray*}
 T^0(e)  =    e &=& \ov {22} \ \,{\rm of \ order\ } 40 \\
 T^0( \Frac{v}{u})= \Frac{v}{u} &=& \Frac{76626}{139193}\,\virg  \ \,{\rm 
 solution\ of\ the\  second\ case,} 
 \end{eqnarray*} 
 \begin{eqnarray*}
 T^{\ }(e)  =  e_1 &=& \ov {9} \ \,{\rm of \ order\ } 4 \\
 T^{\ }(\Frac{v}{u}) =  \Frac{v_1}{u_1} &=& \Frac{14059}{215819}\,\virg  \ \,{\rm 
 solution\ of\ the\  special\ case,} 
 \end{eqnarray*} 
 \begin{eqnarray*}
  T^2(e)  =   e_2 &=& \ov {14} \ \,{\rm of \ order\ } 8 \\
  T^2(\Frac{v}{u}) =    \Frac{v_2}{u_2} &=& \Frac{-62567}{76626}\,\virg  \ \,{\rm 
 solution\ of\ the\  second\ case,} 
 \end{eqnarray*} 
 \begin{eqnarray*}
   T^3(e)  =   e_3 &=& \ov {10} \ \,{\rm of \ order\ } 5 \\
   T^3(\Frac{v}{u}) =   \Frac{v_3}{u_3} &=& \Frac{-201760}{14059}\,\virg  \ \,{\rm 
 solution\ of\ the\  special\ case,} 
 \end{eqnarray*} 
 \begin{eqnarray*}
   T^4(e)  =   e_4 &=& \ov {39} \ \,{\rm of \ order\ } 20 \\
   T^4(\Frac{v}{u}) =  \Frac{v_4}{u_4} &=& \Frac{139193}{62567}\,\virg  \ \,{\rm 
 solution\ of\ the\  first\ case,} 
 \end{eqnarray*} 
 \begin{eqnarray*}
   T^5(e)  =    e_5 &=& \ov {5} \ \,{\rm of \ order\ } 20 \\
   T^5(\Frac{v}{u}) =   \Frac{v_5}{u_5} &=& \Frac{215819}{201760}\,\virg  \ \,{\rm 
 solution\ of\ the\  special\ case.} 
 \end{eqnarray*}  }
 
 As a consequence, we have:\footnotesize{
  \begin{eqnarray*}
 &&\Phi_{40}(139193,76626) \equiv \Phi_{4}(215819,14059) \equiv 
      \Phi_{8}(76626,-62567)\equiv  \\
 &&\Phi_{5}(14059,-201760) \equiv \Phi_{20}(62567,139193)\equiv 
 \Phi_{20}(201760,215819) \equiv  0  \pmod {41}.
  \end{eqnarray*} }
  \normalsize
{ We have obtained the set of orders $\{40, 4, 8, 5, 20 \}$. 
 For instance, this implies the inertia of ${\mathfrak q}_{\xi_{40}}$ in 
 $F_{\xi_{40}}/\Q(\mu_{40})$ and that of ${\mathfrak q}_{\xi_{4}}$ in 
 $F_{\xi_{4}}/\Q(\mu_{4})$, which illustrates the incompatibility 
 with statements like Theorem 2 for  $p=3$}.~\hfill\fin  }
\end{ex} 

\begin{ex} {\rm
 We have found the following numerical example to
 illustrate  Remark 12, with 
$m=5$ for which  $L=\Q(\mu_5)$ is principal.
Consider the prime $q=48738631$ for which 
$q-1 = 2\,\cdot\,3\,\cdot\,5\,\cdot\,163\,\cdot\,9967$
and $\kappa \not\equiv 0 \pmod {3}$.

Then ${\mathfrak q} = (\xi^2+\xi^3 - 3 - 
90\,(3\,\xi^2 + 5\,\xi + 3))\,\Z[\xi]$,
where $\xi$ is a primitive $5$th root of unity, is
a prime ideal above $q$.

\smallskip
Since $\xi^2+\xi^3 - 3 \in L^+$, this ideal 
satisfies the relation ${\mathfrak q}^{1-c} = (\alpha)\,\Z[\xi]$,
$\alpha \equiv 1 \pmod {9}$, which means that $q$
totally splits in $H_L^-{\st [3]}/\Q$. 

\smallskip
Concerning the solutions $(u,v) = (s^3+t^3 -3\,st^2,\,3\,st(s-t))$,
$s, t \in \Z$, ${\rm g.c.d.\,}(s,t)=1$, $s+t \not\equiv 0 \pmod {3}$,
such that $\Phi_5(u,v) \equiv 0 \pmod {q}$,
we try to find the smallest values of the order $n$ of $\frac{v}{u}$
modulo $q$.
It is clear that the value $n=5$ is by construction impossible.
There is also no solution for $n=10$ since $\Q(\mu_{10}) = \Q(\mu_5) =L$
with $q$ totally split in $H_L^-{\st [3]}/\Q$.

\smallskip
We find the values:
\footnotesize{
\begin{eqnarray*}
&&n = 6\ \ \ {\rm for} \ (s,t)=(357, 42643), \\      
&&n = 15\ \ {\rm for} \ (s,t)=(1531, 3232), \\
&&n = 163\ {\rm for} \ (s,t)=(143, 947), \\
&&n = 326\ {\rm for} \ (s,t)=(132, 883), \\
&&n = 489\ {\rm for} \ (s,t)=(79, 526), \\
&&n = 815\ {\rm for} \ (s,t)=(9, 971)\ldots
 \end{eqnarray*}
 \normalsize
 As we have seen, the orders $n=1$ and 2 are impossible.~\hfill\fin } }
\end{ex}

\begin{ex} {\rm In another point of view, in the following example
    we fix the solution $(u,v) = (19,18)$ corresponding to $(s,t) = 
    (3,1)$ and we give the order $n$ of $\frac{v}{u}$ modulo
    $q$ for primes $q< 3.10^6$ with $\kappa \not\equiv 0 \pmod {3}$,
    such that $n< q^{\frac{1}{3}}$ to limit the data.
\footnotesize{
$$\begin{array}{lllllllll}
q  &  n  &
\hspace{0.8cm} q  & n & 
\hspace{0.8cm} q &  n &
\hspace{0.8cm} q  &  n & \\ \vspace{-0.2cm} \\
79  &   3  &
\hspace{0.8cm} 137  & 4 & 
\hspace{0.8cm} 751 &  5 &
\hspace{0.8cm} 17341  &  17 & \\
46663  &  11 &
\hspace{0.8cm} 49999  &   13 & 
\hspace{0.8cm} 97373  &   44 &
\hspace{0.8cm} 225751  &   43 & \\
352771  &   55& 
\hspace{0.8cm} 419693  &   13 & 
\hspace{0.8cm} 464549  &   47 & 
\hspace{0.8cm} 536609  &   41 & \\
809359  &   22& 
\hspace{0.8cm} 816401  &  52 & 
\hspace{0.8cm} 1037471  &   35 & 
\hspace{0.8cm} 1115447  &   41 & \\
1167937  &   84& 
\hspace{0.8cm} 1252057  &   104 & 
\hspace{0.8cm} 1403627  &   14 & 
\hspace{0.8cm} 1529249  &   32 & \\
1995781  &   29& 
\hspace{0.8cm} 2040601 &    25 & 
\hspace{0.8cm} 2743501 &    59 & 
\hspace{0.8cm} 2912521  &   39 &
\end{array}$$  } }
\end{ex}

\begin{ex} {\rm  Let $q=113 = 1+ 2^4\,.\,7$.
    In the following example we fix $n$ and use a polynomial
$X^3 - 3 e^{-1} \,X^2 - 3 (1-e^{-1})\,X + 1$ which splits modulo 
$113$; for $e=83$, of order $n=14$ modulo 113, 
its roots are $\ov 5$,  $\ov {28}$, and $\ov {46}$ modulo $113$.

\smallskip
Recall that for $\xi$ of order $n$ and $e\in 
\Z$ defining the prime ideal ${\mathfrak q} = (q, \xi - e)$ above $q$, the solutions $(s,t)$ giving
$q \div  \Phi_n(u,v)$ for the corresponding solutions
$(u,v) = (s^3+t^3-3st^2, 3st(s-t))$, are defined for instance via the 
congruence $s-5 t \equiv 0 \pmod {113}$, g.c.d.\,$(s,t) = 1$,
and $s+t \not \equiv 0 \pmod {3}$.
\footnotesize{
$$\begin{array}{llll}
s    &   t  &     \Phi_n(u,v)   \\  \vspace{-0.2cm} \\
118  &   1  &     113 \,\cdot\, 3557 \,\cdot\, 3942401 \,\cdot\, 744072113 \,\cdot\, 16254128953756891   \\
231  &   1  &     113 \,\cdot\, 211 \,\cdot\, 239 \,\cdot\, 116929 \,\cdot\, 550757191489 \,\cdot\, 9432961248517529143   \\
457  &   1  &     113 \,\cdot\, 8821 \,\cdot\, 18484859 \,\cdot\, 4489993033 \,\cdot\, 9077382763538364383220967   \\
123  &   2  &     29 \,\cdot\, 43 \,\cdot\, 113 \,\cdot\, 3011 \,\cdot\, 11047 \,\cdot\, 1005000683 \,\cdot\, 8371388009051383   \\
128  &   3  &     113 \,\cdot\, 385897 \,\cdot\, 8800908691961 \,\cdot\, 205376563933889209  \\
241  &   3  &     29 \,\cdot\, 113 \,\cdot\, 3557 \,\cdot\, 26209 \,\cdot\, 136067 \,\cdot\, 2120693 \,\cdot\, 2348198329 \,\cdot\, 34945284137  \\
467  &   3  &     113 \,\cdot\, 1451130199 \,\cdot\, 6673578443419738169458023356294472959 
\end{array}$$
$$\begin{array}{llll}
133  &   4  &     113 \,\cdot\, 421 \,\cdot\, 43270571265013 \,\cdot\,  74514155796456659333  \\
138  &   5  &     113 \,\cdot\, 2577267166287809480749101354040384043  \\
251  &   5  &     113 \,\cdot\, 547 \,\cdot\, 2381 \,\cdot\, 75688397 \,\cdot\, 318274119451 \,\cdot\, 4136563302302243   \\
477  &   5  &     29 \,\cdot\, 113 \,\cdot\, 5503 \,\cdot\, 26385694924317373 \,\cdot\, 3324436493654921921540503   \\
143  &   6  &     113 \,\cdot\, 1847609 \,\cdot\, 2588587173822250293234785701459   \\
\end{array}$$  }

\normalsize{
 We observe a unique case where $113^2$ divides $\Phi_n(u,v)$.}~\hfill\fin }
\end{ex}

\begin{ex} {\rm We consider the prime number $q=401 = 1 + 
2^4\,.\,5^2$ and for all possible values of $\rho:=\frac{v}{u}$ 
modulo $q$, for the general solution of the second case,
 we give the order of  $\rho$ modulo $q$.
The resolution of $\Frac{3st(s-t)}{s^3+t^3-3st^2} \equiv \rho \pmod {q}$
is of course equivalent to get the values $\rho$ such that the 
polynomial $X^3 - 3 \rho^{-1} \,X^2 - 3 (1-\rho^{-1})\,X + 1$ splits
modulo $q$.

\smallskip
There are $133 = 7\,.\,13$ distinct values of such $\rho$
with the following repartition of the orders $n$: 53 for order 400; 28 
for 200; 13 for 80; 12 for 100; 7 for 50 and 25; 4 for 40; 3 for 20; 2 
for 10; 1 for 16, 8, 5, and 4. As we know, orders 1 and 2 cannot exist.
These densities are in accordance with the expression $\frac{1}{3} 
\phi(n)$.~\hfill\fin  }
\end{ex}

The above computations for  $p=3$ suggest the following conjecture.

\begin{conj}  For all $m>0$ and for all prime numbers $q \equiv 1 
\pmod {m}$, $q\equiv 2 {\rm\ or\ } 5 \pmod {9}$, and $q$  totally split in 
$F_m/\Q$, there exists a solution $(u,v)$ of the SFLT 
equation for $p=3$, for which the order of $\frac{v}{u}$ modulo $q$ is 
$\geq m$.~\hfill\fin
\end{conj}

\smallskip
This conjecture (to be compared with Conjecture 2 for $p>3$)  is 
very reasonable since, in practice, the order of $\frac{v}{u}$ modulo $q$ is
often $q-1$.

\section{Conclusion}

In Subsection 9.1,  we have given a justification of the fact that
 Theorem~2 (or any weak form) cannot exist for $p=3$.
 
\smallskip
For $p>3$, if the number of solutions $(u,v)$ of the SFLT equation is finite,
for any bound $N$ the number of primes $q$,
such that the $\frac{v}{u}$ are of order modulo $q$ less than $N$,
is finite.
So for  primes $q'$ for which we assume that all the prime ideals above $q'$
in $L' = \Q(\mu_{m'})$, $m' \div q'-1$ large enough, totally split in 
$H_{L'}^-{\st [p]}/L'$,
we get large values of $q'$, hence large values of the orders $n'$
of the $\frac{v}{u}$ modulo $q'$, say $n' \!>\!\!>\! N$.
So, contrary to the case $p=3$, the effectiveness of the statement
of a weak form of  Theorem 2 is more credible.
 
\smallskip
We have justified, in Subsection 9.2, why the 
case $p = 3$ is specific for the arithmetic of the fields 
$\Q(\mu_n)$ in relation with the abelian $3$-ramification;
which suggests that,  for $p>3$, a result like  Theorem 4, on the constraints
on the $\rho$-laws of decomposition of infinitely many primes~$q$,
gives a non trivial obstruction  and is likely to lead to a proof of SFLT.

\smallskip
In other words, we can hope that for $p>3$ any statistical
analysis of the decomposition laws is legitimate.

\smallskip
To summarize it is not excluded that the two main principles of approach of the SFLT 
problem that we have developped in this paper 
may be successful for $p>3$.

\smallskip
However, it should be noted that results like Theorem 2 are sufficient
diophantine conditions,
probably too strong, and that it would be better to return to the principle of laws
of $\rho$-decomposition of infinitely many primes $q$ for
the canonical families ${\mathcal F}_n$
(see Subsection 7.1, Theorem 4, and Conjecture 3); this aspect can be 
approached from an analytic point of view  with the aim 
to show that such a constraint is impossible for $p>3$.

\smallskip
In this direction, an interesting fact would be that the case $p=3$ 
would have, in some sense, a reciprocal statement, namely that the infiniteness of
the set of solutions of the SFLT equation and their particular repartition
into six families, is in fact necessary for the \v Cebotarev${}'$s
density theorem.
Thus for $p>3$, in the same spirit as  the case $p=3$, the set of nontrivial solutions
(if nonempty) would be necessarily  infinite with some structural properties in
order to be compatible with the above principle, which seems clearly impossible
for geometrical reasons (Theorem 8 for instance).

\end{document}